\documentclass[11pt]{article}
\newcommand{\taille}{}
\usepackage{latexsym} 
\usepackage{theorem}
\usepackage{graphics}
\usepackage{amsmath}
\usepackage{amssymb}

\newtheorem{defeng}{Definition}[section]
\newtheorem{theorem}[defeng]{Theorem}

\newtheorem{lemma}[defeng]{Lemma}
\newtheorem{lemmaS}{Lemma}[subsection]

\newtheorem{conjecture}[defeng]{Conjecture}

{\theorembodyfont{\rmfamily} }
{\theorembodyfont{\rmfamily} }
{\theorembodyfont{\rmfamily} }
{\theoremstyle{break}\theorembodyfont{\rmfamily} }
{\theoremstyle{break}\theorembodyfont{\rmfamily} }

\newcommand{\bp}{}
\newcommand{\tp}{\!-\!}
\newcommand{\ep}{}

\newcounter{claim}

\newenvironment{proof}[1][]%
 {\noindent {\setcounter{claim}{0}\sc proof ---
   }{#1}{}}{\hfill$\Box$\vspace{2ex}} 

\newenvironment{claim}[1][]%
{\refstepcounter{claim}\vspace{1ex}\noindent{(\it\arabic{claim}){#1}{}}\it}{\vspace{1ex}}

\newenvironment{proofclaim}[1][]%
	{\noindent {}{#1}{}}{ This proves~(\arabic{claim}).\vspace{1ex}}

\newcommand{\even}{balanced}
\newcommand{\Even}{Balanced}
\newcommand{\an}{a}
\newcommand{\An}{A}
\newcommand{\ESPD}{BSP}

\title{Decomposing Berge graphs and detecting \even\  skew partitions}
\date{April 15, 2006\\Revised July 14, 2007} 

\author{Nicolas Trotignon\thanks{Universit\'e Paris I
    Centre d'\'Economie de la Sorbonne, 106--112 boulevard de
    l'H\^opital, 75647 Paris cedex 13, France\newline
    nicolas.trotignon@univ-paris1.fr\newline This work has been
    partially supported by ADONET network, a Marie Curie training
    network of the European Community.}}
    
\begin{document}

\maketitle

\section*{Abstract}
{
A hole in a graph is an induced cycle on at least four vertices. A
graph is Berge if it has no odd hole and if its complement has no odd
hole. In 2002, Chudnovsky, Robertson, Seymour and Thomas proved a
decomposition theorem for Berge graphs saying that every Berge graph
either is in a well understood basic class, or has some kind of
decomposition. Then, Chudnovsky proved stronger theorems. One of them
restricts the allowed decompositions to 2-joins and \even\  skew
partitions.

We prove that the problem of deciding whether a graph has
\an\ \even\ skew partition is NP-hard. We give an $O(n^9)$-time
algorithm for the same problem restricted to Berge graphs. Our
algorithm is not constructive: it only certifies whether a graph has
\an\ \even\ skew partition or not. It relies on a new decomposition
theorem for Berge graphs that is more precise than the previously
known theorems. Our theorem also implies that every Berge graph can be
decomposed in a first step by using only \even\ skew partitions, and
in a second step by using only 2-joins. Our proof of this new theorem
uses at an essential step one of the theorems of Chudnovsky.}

\noindent AMS Mathematics Subject Classification: 05C17, 05C75

\noindent Key words: perfect graph, Berge graph, 2-join, \even\  skew
partition, decomposition, detection, recognition.

\taille

\section*{Outline of the article}

Section~\ref{intro} surveys the decomposition theorems for Berge
graphs. Section~\ref{algomotiv} motivates and sketches the algorithm
that detects balanced skew partitions in Berge
graphs. Section~\ref{decth} gives the new definitions necessary to
state properly our new decomposition Theorem~\ref{th.th}, states it,
sketches its proof and explains why it is a generalization of the
previously known decomposition theorems for Berge
graphs. Section~\ref{lemmas} gives some useful technical lemmas and
studies how 2-joins and \even\ skew partitions can overlap in a Berge
graph. Section~\ref{proof} gives the proof of Theorems~\ref{th.th}.
Its corollary~\ref{th.case} is proved in
Section~\ref{s:proofcase}. Section~\ref{algos} describes the
algorithms in detail. Section~\ref{NPC} proves that the detection of
\even\ skew partitions is NP-hard for general
graphs. In Section~\ref{conclu}, two conjectures are given.

\section{Decomposing Berge graphs: a survey}
\label{intro}

In this paper graphs are simple and finite. A \emph{hole} in a graph
is an induced cycle of length at least~4. An \emph{antihole} is the
complement of a hole. A graph is said to be Berge if it has no odd
hole and no odd antihole. A graph $G$ is said to be \emph{perfect} if
for every induced subgraph $G'$ of $G$, the chromatic number of $G'$
is equal to the maximum size of a clique of $G'$. In 1961,
Berge~\cite{berge:61} conjectured that every Berge graph is
perfect. This was known as the \emph{Strong Perfect Graph Conjecture},
was the object of much research and was finally proved by Chudnovsky,
Robertson, Seymour and Thomas in
2002~\cite{chudnovsky.r.s.t:spgt}. Actually, they proved a stronger
result: a decomposition theorem, conjectured by Conforti, Cornu\'ejols
and Vu\v skovi\'c~\cite{conforti.c.v:square}, stating that every Berge
graph is either in a well understood basic class of perfect graphs, or
has a structural fault that cannot occur in a minimum counter-example
to Strong Perfect Graph Conjecture. Before stating this decomposition
theorem, we need some definitions.

We call \emph{path} any connected graph with at least one vertex of
degree~1 and no vertex of degree greater than~2. A path has at most
two vertices of degree~1, which are the \emph{ends} of the path. If
$a, b$ are the ends of a path $P$ we say that $P$ is \emph{from $a$
  to~$b$}. The other vertices are the \emph{interior} vertices of the
path. We denote by $v_1 \tp \cdots \tp v_n$ the path whose edge set is
$\{v_1v_2, \dots, v_{n-1}v_n\}$.  When $P$ is a path, we say that $P$
is \emph{a path of $G$} if $P$ is an induced subgraph of $G$. If $P$
is a path and if $a, b$ are two vertices of $P$ then we denote by $a
\tp P \tp b$ the only induced subgraph of $P$ that is path from $a$ to
$b$.  The \emph{length} of a path is the number of its edges. An
\emph{antipath} is the complement of a path.  Let $G$ be a graph and
let $A$ and $B$ be two subsets of $V(G)$. A path of $G$ is said to be
\emph{outgoing from $A$ to $B$} if it has an end in $A$, an end in
$B$, length at least~2, and no interior vertex in $A\cup B$.

If $X, Y \subset V(G)$ are disjoint, we say that $X$ is
\emph{complete} to $Y$ if every vertex in $X$ is adjacent to every
vertex in $Y$. We also say that $(X, Y)$ is a \emph{complete pair}. We
say that $X$ is \emph{anticomplete} to $Y$ if there are no edges
between $X$ and $Y$. We also say that $(X, Y)$ is an
\emph{anticomplete pair}. We say that a graph $G$ is anticonnected if
its complement $\overline{G}$ is connected.

A \emph{cutset} in a graph $G$ is a set $C\subset V(G)$ such that
$G\setminus C$ is disconnected ($G\setminus C$ means $G[V(G) \setminus
  C]$).

Skew partitions were first introduced by
Chv\'atal~\cite{chvatal:starcutset}. A \emph{skew partition} of a
graph $G = (V,E)$ is a partition of $V$ into two sets $A$ and $B$ such
that $A$ induces a graph that is not connected, and $B$ induces a
graph that is not anticonnected. When $A_1, A_2, B_1, B_2$ are
non-empty sets such that $(A_1, A_2)$ partitions $A$, $(A_1, A_2)$ is
an anticomplete pair, $(B_1, B_2)$ partitions $B$, and ($B_1, B_2$) is
a complete pair, we say that $(A_1, A_2, B_1, B_2)$ is a \emph{split}
of the skew partition $(A, B)$. \An\ \emph{\even\ skew partition}
(first defined in~\cite{chudnovsky.r.s.t:spgt}) is a skew partition
$(A, B)$ with the additional property that every induced path of
length at least~2 with ends in $B$, interior in $A$ has even length,
and every antipath of length at least~2 with ends in $A$, interior in
$B$ has even length. If $(A, B)$ is a skew partition, we say that $B$
is a \emph{skew cutset}. If $(A, B)$ is \even\ we say that the skew
cutset $B$ is \emph{\even}. Note that Chudnovsky et
al.~\cite{chudnovsky.r.s.t:spgt} proved that no minimum
counter-example to the strong perfect graph conjecture has
\an\ \even\ skew partition.

Call \emph{double split graph} (first defined
in~\cite{chudnovsky.r.s.t:spgt}) any graph $G$ that may be constructed
as follows.  Let $m,n \geq 2$ be integers. Let $A = \{a_1, \dots,
a_m\}$, $B= \{b_1, \dots, b_m\}$, $C= \{c_1, \dots, c_n\}$, $D= \{d_1,
\dots, d_n\}$ be four disjoint sets. Let $G$ have vertex set $A\cup B
\cup C \cup D$ and edges in such a way that:

\begin{itemize}
\item 
  $a_i$ is adjacent to $b_i$ for $1 \leq i \leq m$.  There are no
  edges between $\{a_i, b_i\}$ and $\{a_{i'}, b_{i'}\}$ for $1\leq i <
  i' \leq m$;
\item 
  $c_j$ is non-adjacent to $d_j$ for $1 \leq j \leq n$. There are all
  four edges between $\{c_j, d_j\}$ and $\{c_{j'}, b_{j'}\}$ for
  $1\leq j < j' \leq n$;
\item
  there are exactly two edges between $\{a_i, b_i\}$ and $\{c_j,
  d_j\}$ for $1\leq i \leq m$, $1 \leq j \leq n$ and these two
  edges are disjoint.
\end{itemize}

Note that $C\cup D$ is a non-\even\ skew cutset of $G$ and that
$\overline{G}$ is a double split graph. Note that in a double split
graph, vertices in $A \cup B$ all have degree $n+1$ and vertices in
$C\cup D$ all have degree $2n + m - 2$. Since $n \geq 2, m \geq 2$
implies $2n -2 + m > 1 + n$, it is clear that given a double split
graph the partition $(A\cup B, C \cup D)$ is unique. Hence, we call
\emph{matching edges} the edges that have an end in $A$ and an end in
$B$.

A graph is said to be \emph{basic} if one of $G, \overline{G}$ is
either a bipartite graph, the line-graph of a bipartite graph or a
double split graph.

The 2-join was first defined by Cornu\'ejols and
Cunningham~\cite{cornuejols.cunningham:2join}.  A partition $(X_1,
X_2)$ of the vertex set is a \emph{2-join} when there exist disjoint
non-empty $A_i, B_i \subseteq X_i$ ($i=1, 2$) satisfying:

\begin{itemize} 
\item
  every vertex of $A_1$ is adjacent to every vertex of $A_2$ and every
  vertex of $B_1$ is adjacent to every vertex of $B_2$;
\item
  there are no other edges between $X_1$ and $X_2$.
\end{itemize}

The sets $X_1, X_2$ are the two \emph{sides} of the 2-join.  When sets
$A_i$'s $B_i$'s are like in the definition we say that $(X_1, X_2,
A_1, B_1, A_2, B_2)$ is a \emph{split} of $(X_1, X_2)$.  Implicitly,
for $i= 1, 2$, we will denote by $C_i$ the set $X_i \setminus (A_i
\cup B_i)$.

A 2-join $(X_1, X_2)$ in a graph $G$ is said to be \emph{connected}
when for $i= 1, 2$, every component of $G[X_i]$ meets both $A_i$ and
$B_i$.  A 2-join $(X_1, X_2)$ is said to be \emph{substantial} when
for $i= 1, 2$, $|X_i| \geq 3$ and $X_i$ is not a path of length~2 with
an end in $A_i$, an end in $B_i$ and its unique interior vertex in
$C_i$.  A 2-join $(X_1, X_2)$ in a graph $G$ is said to be
\emph{proper} when it is connected and substantial.

A 2-join is said to be a \emph{path 2-join} if it has a split $(X_1,
X_2, A_1, B_1, A_2, B_2)$ such that $G[X_1]$ is a path with an end in
$A_1$, an end in $B_1$ and interior in $C_1$. Implicitly we will then
denote by $a_1$ the unique vertex in $A_1$ and by $b_1$ the unique
vertex in $B_1$. We say that $X_1$ is the \emph{path-side} of the
2-join. Note that when $G$ is not a hole then only one of $X_1, X_2$
is a path side of $(X_1, X_2)$. A \emph{non-path 2-join} is a 2-join
that is not a path 2-join.

The homogeneous pair was first defined by Chv\'atal and
Sbihi~\cite{chvatal.sbihi:bullfree}. The definition that we give here
is a slight variation used in~\cite{chudnovsky.r.s.t:spgt}. A
\emph{homogeneous pair} is a partition of $V(G)$ into six non-empty
sets $(A, B, C, D, E, F)$ such that:

\begin{itemize} 
\item 
  every vertex in $A$ has a neighbor in $B$ and a non-neighbor in $B$,
  and vice versa; 
\item the pairs $(C,A)$, $(A,F)$, $(F,B)$, $(B,D)$ are complete; 
\item the pairs $(D,A)$, $(A,E)$, $(E,B)$, $(B,C)$ are anticomplete. 
\end{itemize}

A graph $G$ is path-cobipartite if it is a Berge graph obtained by
subdividing an edge between the two cliques that partitions a
cobipartite graph.  More precisely, a graph is
\emph{path-cobipartite} if its vertex set can be partitioned into
three sets $A, B, P$ where $A$ and $B$ are non-empty cliques and $P$
consist of vertices of degree~2, each of which belongs to the interior
of a unique path of odd length with one end $a$ in $A$, the other one
$b$ in $B$. Moreover, $a$ has neighbors only in $A \cup P$ and $b$ has
neighbors only in $B \cup P$. Note that a path-cobipartite graph such
that $P$ is empty is the complement of bipartite graph. Note that our
path-cobipartite graphs are simply the complement of the
\emph{path-bipartite} graphs defined by Chudnovsky
in~\cite{chudnovsky:these}. For convenience, we prefer to think about
them in the complement as we do.

 A \emph{double star} in a graph is a subset $D$ of the
vertices such that there is an edge $ab$ in $G[D]$ satisfying: $D
\subset N(a) \cup N(b)$.

Now we can state the known decomposition theorems of Berge graphs. The
first decomposition theorem for Berge graph ever proved is the
following:

\begin{theorem}[Conforti, Cornu\'ejols and Vu\v skovi\'c, 
    2001, \cite{conforti.c.v:dstrarcut}]
  \label{th.ccv}
  Every graph with no odd hole is either basic or has a proper 2-join
  or has a double star cutset.
\end{theorem}

It could be thought that this theorem is useless to prove the Strong
Perfect Graph Theorem since there are minimal imperfect graphs that
have double star cutsets: the odd antiholes of length at
least~7. However, by the Strong Perfect Graph Theorem, we know that
the following fact is true: for any minimal non-perfect graph $G$, one
of $G, \overline{G}$ has no double star cutset. A direct proof of this
--- of which we have no idea --- would yield together with
Theorem~\ref{th.ccv} a new proof of the Strong Perfect Graph Theorem.

The following theorem was first conjectured in a slightly different
form by Conforti, Cornu\'ejols and Vu\v skovi\'c, who proved it in the
particular case of square-free graphs~\cite{conforti.c.v:square}.  A
corollary of it is the Strong Perfect Graph Theorem.

\begin{theorem}[Chudnovsky, Robertson, Seymour and Thomas, 2002, \cite{chudnovsky.r.s.t:spgt}]
  \label{th.0}
  Let $G$ be a Berge graph. Then either $G$ is basic or $G$ has a
  homogeneous pair, or $G$ has \an\  \even\  skew partition or one of $G,
  \overline{G}$ has a proper 2-join.
\end{theorem} 

 The two theorems that we state now are due to Chudnovsky who proved
 them from scratch, that is without assuming Theorem~\ref{th.0}. Her
 proof uses the notion of \emph{trigraph}.  The first theorem shows
 that homogeneous pairs are not necessary to decompose Berge
 graphs. Thus it is a result stronger than Theorem~\ref{th.0}. The
 second one shows that path 2-joins are not necessary to decompose
 Berge graphs, but at the expense of extending \even\ skew partitions
 to general skew partitions and introducing a new basic class. Note
 that a third theorem can be obtained by viewing the second one in the
 complement of $G$.

\begin{theorem}[Chudnovsky, 2003, \cite{chudnovsky:trigraphs,chudnovsky:these}]
\label{th.1}
  Let $G$ be a Berge graph. Then either $G$ is basic, or one of $G,
  \overline{G}$ has a proper 2-join or $G$ has \an\  \even\  skew partition.
\end{theorem}

\begin{theorem}[Chudnovsky, 2003, \cite{chudnovsky:these}]
\label{th.2}
  Let $G$ be a Berge graph. Then either $G$ is basic, or one of $G,
  \overline{G}$ is path-bipartite, or $G$ has a proper non-path
  2-join, or $\overline{G}$ has a proper 2-join, or $G$ has a
  homogeneous pair or $G$ has a skew partition.
\end{theorem}

\section{Algorithmic results and motivation}
\label{algomotiv}

Our main result is Theorem~\ref{th.th}, a new decomposition for Berge
graphs that is a generalization of Theorems~\ref{th.0},~\ref{th.1}
and~\ref{th.2}.  Note that our proof of Theorem~\ref{th.th} is not a
new proof of the previously known decomposition theorems for Berge
graphs, since it uses Theorem~\ref{th.1} at an essential step.  We
also give algorithmic applications. De Figueiredo, Klein, Kohayakawa
and Reed devised an algorithm that given a graph $G$ computes in
polynomial time a skew partition if $G$ has
one~\cite{figuereido.k.k.r:sp}. See also a recent work by Kennedy and
Reed~\cite{kennedyreed:skew}.  But the problem of detecting
\emph{\even} skew partitions has not been studied so far. Let us call
\ESPD\ the decision problem whose input is a graph and whose answer is
YES if the graph has \an\ \even\ skew partition and NO
otherwise. Using a construction due to
Bienstock~\cite{bienstock:evenpair}, we prove in Section~\ref{NPC}
that \ESPD\ is NP-hard (we are not able to prove that \ESPD\ is in NP
or in coNP). Using Theorem~\ref{th.th} we give an $O(n^9)$-time
algorithm for \ESPD\ restricted to Berge graphs.

In 2002, Chudnovsky, Cornu\'ejols, Liu, Seymour and Vu\v
skovi\'c~\cite{chudnovsky.c.l.s.v:reco} gave an algorithm that
recognizes Berge graphs in time $O(n^9)$.  This algorithm may be used
to prove that, when restricted to Berge graphs, \ESPD\ is in NP.
Indeed, \an\  \even\  skew partition is a good certificate for \ESPD:
given a Berge graph and a partition $(A,B)$ of its vertices, one can
easily check that $(A,B)$ is a skew partition; to check that it is
\even, it suffices to add a vertex adjacent to every vertex of $B$, to
no vertex of $A$, and to check that this new graph is still Berge.

Proving that \ESPD\ is actually in P by a decomposition theorem uses a
classical idea, used for instance in~\cite{conforti.c.k.v:eh2} to
check whether a given graph has or not an even hole. First, solve
\ESPD\ for each class of basic graph. This is done in
Section~\ref{algos} in time $O(n^5)$. Note that bipartite graphs are
the most difficult to handle efficiently. For them, we use an
algorithm due to Reed~\cite{reed:skewhist}. For a graph $G$ such that
one of $G, \overline{G}$ has a 2-join, we try to break $G$ into
smaller blocks in such a way that $G$ has \an\ \even\ skew partition
if and only if one of the blocks has one, allowing us to run
recursively the algorithm. And when a graph is not basic and has no
2-join, we simply answer ``the graph has \an\ \even\ skew partition'',
which is the correct answer because of the Decomposition
Theorem~\ref{th.1}. This blind use of decomposition is not safe from
criticism, but this will be discussed later.

Unfortunately, with the usual notions of 2-join and blocks, this
approach fails to solve \ESPD.  Building the blocks of a 2-join
preserves existing \even\ skew partitions, but some 2-joins can create
\even\ skew partitions when building the blocks carelessly. In the
graph represented in Fig.~\ref{figP3loose} on the left, we have to
simplify somehow the left part of the obvious 2-join to build one of
the blocks.  The most reasonable way to do so seems to be replacing
$X_1$ by a path of length~1.  But this creates a skew cutset: the
black vertices on the right. Of course, this graph is bipartite but
one can find more complicated examples based on the same template, and
another template exists.  These bad 2-joins will be described in more
details in Section~\ref{decth} and called~\emph{cutting 2-joins}. All
of them are path 2-joins.

\begin{figure}[h]
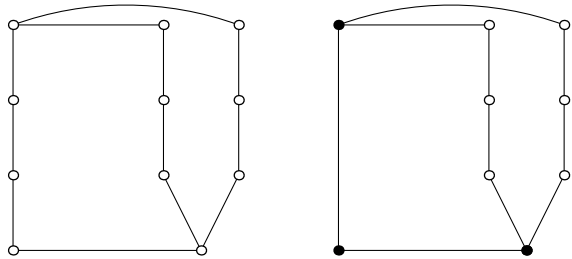
  
  \center
  \includegraphics{evenskew.13}\hspace{3em}\includegraphics{evenskew.14}
  \caption{Contracting a path creates a skew cutset\label{figP3loose}}
\end{figure}

Theorem~\ref{th.th} shows that cutting 2-joins are not necessary to
decompose Berge graphs.  A more general statement is proved, that
makes use of a new basic class and of a new kind of decomposition that
are quite long to describe. But  an interesting corollary 
can be stated with no new notions. By \emph{contracting a path $P$}
that is the side of a proper path 2-join of a graph we mean delete the
interior vertices of $P$, and link the ends of $P$ with a path of
length~1 or~2 according to the original parity of the length of $P$.

\begin{theorem}
  \label{th.case}
  Let $G$ be a Berge graph. Then either:
  \begin{itemize}
    \item $G$ is basic;
    \item one of $G, \overline{G}$ has a non-path proper 2-join;
    \item $G$ has no \even\ skew partition and exactly one of $G,
      \overline{G}$ (say $G$) has a proper path 2-join. Moreover, for
      every proper path 2-join of $G$, the graph obtained by
      contracting its path-side has no \even\ skew partition;
    \item $G$ has \an\  \even\  skew partition.
  \end{itemize}
\end{theorem}

The algorithm for detecting \even\ skew partitions is now easy to
sketch. Since the \even\ skew partition is a self-complementary
notion, we may switch from the graph to its complement as often as
needed. First check whether the input graph is basic, and if so look
directly for \an\ \even\ skew partition. Else, try to decompose along
non-path 2-joins (they preserve the existence of \even\ skew
partitions).  If there are none of them, try to decompose along path
2-joins (possibly, this creates \even\ skew partitions but does not
destroy them). At the end of this process, one of the leaves of the
decomposition tree has \an\ \even\ skew partition if and only if the
root has one. Note that \an\ \even\ skew partition in a leaf may have
been created by the contraction of a cutting 2-join since such 2-joins
do exist (we are not able to recognize all of them, it seems to be a
difficult task). But Theorem~\ref{th.case} shows that when such a bad
contraction occurs, the graph has anyway \an\ \even\ skew cutset
somewhere. The proof of correctness and complexity analysis are given
in Section~\ref{algos}.

Theorem~\ref{th.case} gives a structural description of Berge graphs
that have no \even\ skew partitions: these graphs can be decomposed
along 2-joins till reaching basic graphs.  This could be used to solve
algorithmic problems for the class of Berge graphs with no \even\ skew
partitions (together with the Berge graphs recognition
algorithm~\cite{chudnovsky.c.l.s.v:reco}, our work solves the
recognition in $O(n^9)$).  This class has an unusual feature
in the field of perfect graphs: it is not closed under taking induced
subgraphs. Theorem~\ref{th.case} also gives a structural information on
every Berge graph: it can be decomposed in a first step by using only
\even\ skew partitions, and in a second step by using only 2-joins,
possibly in the complement.

Let us come back to the weak point of our recognition algorithm, which
is when it answers ``the graph has \an\ \even\ skew-partition'' using
blindly some decomposition theorem. This weakness is the reason why we
are not able to find explicitly \an\ \even\ skew partition when there
is one.  However, our result suggests that an explicit algorithm might
exist. The proof of Theorem~\ref{th.0} or Theorem~\ref{th.1} might
contain its main steps. However, we would like to point out that if
someone manage to read algorithmically the proof of Theorem~\ref{th.0}
or of Theorem~\ref{th.1}, (s)he will probably end up with an algorithm
that given a graph, either finds an odd hole/antihole, or certifies
that the graph is basic, or finds some decomposition. If the
decomposition found is not \an\ \even\ skew partition, the algorithm
will probably not certify that there is no \even\ skew partition in
the graph, and thus \ESPD\ will not be solved entirely. To solve it,
one will have to think about the detection of \even\ skew partitions
in basic graphs, and in graphs having a 2-join: this is what we are
doing here.  Thus an effective algorithm might have to use much of the
present work.

This paper answers in some respect questions asked by several authors,
for instance the problem of how 2-joins and \even\ skew partitions
interact in Berge graphs.  See~\cite{problemperfect} where a section
is devoted to open problems about skew partitions. One of them is the
fast detection of general skew partitions in Berge graphs. This has
been solved for basic graphs by Reed~\cite{reed:skewhist}, so a
decomposition based approach might work. Moreover, at first glance,
general skew partitions seem easier than \even\ skew partitions: in
general graphs the first ones are
polynomial~\cite{figuereido.k.k.r:sp} to detect while the second ones
are NP-hard.  However, in Section~\ref{why} we explain why our work
does not improve the general skew partition detection in Berge graphs,
why we are not able to prove Theorem~\ref{th.case} with ``skew
partition'' instead of ``\even\ skew partition''. Rather than a
failure, we consider this as a further indication that \even\ skew
partition is the relevant decomposition for Berge graphs.

\section{The decomposition theorem and a sketch of its proof}
\label{decth}

As stated in Section~\ref{algomotiv}, our main problem for the
detection of \even\ skew partitions is the possibility of path 2-joins
in Berge graphs. One could hope that these 2-joins are actually not
necessary to decompose Berge graphs. Theorem~\ref{th.2} indicates that
such a hope is realistic, but this theorem allows non-\even\ skew
partitions, so it is useless for our purpose. What we would like is to
prove something like Theorem~\ref{th.1} with ``non-path 2-join''
instead of ``2-join''. Let us call this statement our
\emph{conjecture}. A simple idea to prove the conjecture would be to
consider a minimum counter-example $G$, that is: a Berge graph,
non-basic, with no \even\ skew partition, and no non-path 2-join. Such
a graph must have a path 2-join by Theorem~\ref{th.1} (possibly after
taking the complement). Here is why we need Theorem~\ref{th.1} in our
proof.  The idea is now to use this path 2-join to build a smaller
graph $G'$ that is also a counter-example, and this is a contradiction
which proves the conjecture.

So, given $G$ with its path 2-join, how can we build a smaller graph
that will have ``almost'' the same structure as $G$ ?  Obviously,
this can be done by contracting the path-side of the 2-join. Let us
call $G_c$ the graph that we obtain. It has to be proved that $G_c$ is
still a counter-example to the conjecture. But we know that this can
be false. Indeed, if the path 2-join of $G$ is cutting,
\an\ \even\ skew partition can be created in $G_c$, so $G_c$ is not a
counter-example. We need now to be more specific and to define cutting
2-joins.

\begin{figure}[ht]
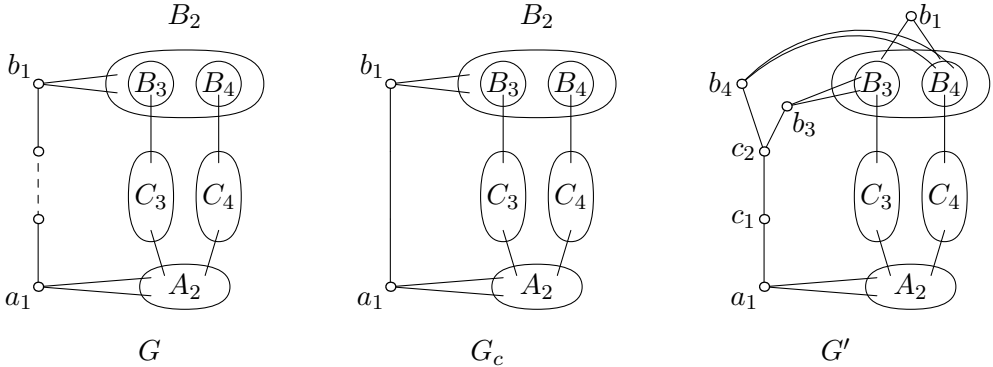

  \begin{center}
    \begin{tabular}{ccc}
      \includegraphics{evenskew.15}\rule{1em}{0ex}&
      \rule{1em}{0ex}\includegraphics{evenskew.19}\rule{1em}{0ex}&
      \rule{1em}{0ex}\includegraphics{evenskew.16}\\
      \rule{0em}{3ex}$G$&\rule{0em}{3ex}$G_c$&\rule{0em}{3ex}$G'$
    \end{tabular}
  \end{center}
  \caption{A graph $G$ with a cutting 2-join of type~1 and the
    associated graph $G'$\label{fig:cut1}}
\end{figure}

A 2-join is said to be \emph{cutting of type~1} if it has a split
$(X_1,$ $X_2,$ $A_1,$ $B_1,$ $A_2,$ $B_2)$ such that:

  \begin{enumerate}
  \item
    $(X_1, X_2)$ is a path 2-join with path-side $X_1$; 

  \item
    $G[X_2 \setminus A_2]$ is disconnected.
  \end{enumerate}

In Fig.~\ref{fig:cut1} the structure of a graph $G$ with a cutting
2-join of type~1 is represented. Obviously, after contracting the
path-side into an edge $a_1b_1$, we obtain a graph $G_c$ with a
potentially-\even\ skew cutset $\{a_1, b_1\} \cup A_2$ that separates
$C_3\cup B_3$ from $C_4 \cup B_4$. So, how can we find a graph smaller
than $G$ that is still a counter-example to the conjecture ? Our idea
is to build the graph $G'$, also represented in Fig.~\ref{fig:cut1}. A
formal definition of $G'$ is given in
Subsection~\ref{mainproofCase1}. If we count vertices, $G'$ is not
``smaller'' than $G$, but in fact, by ``minimum counter-example'' we
mean counter-example with a minimum number of path 2-joins. We can
prove that $G'$ is smaller in this sense (this is not trivial because
we have to prove that path 2-joins cannot be created in $G'$, but
clearly, one path 2-join is destroyed in $G'$). We can also prove that
$G'$ is a counter-example which gives the desired contradiction. This
is the first case of the main proof, described in
Subsection~\ref{mainproofCase1}.

\begin{figure}
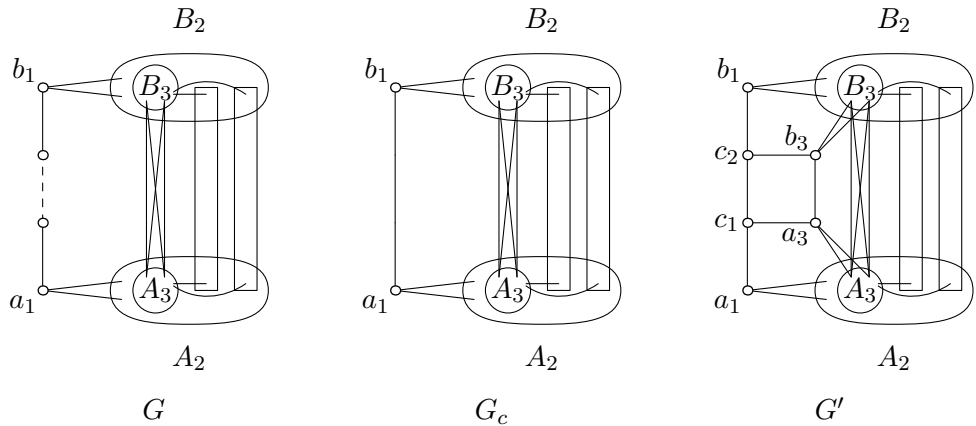

  \begin{center}
    \begin{tabular}{ccc}
      \includegraphics{evenskew.17}\rule{1em}{0ex}&
      \rule{1em}{0ex}\includegraphics{evenskew.20}\rule{1em}{0ex}&
      \rule{1em}{0ex}\includegraphics{evenskew.18}\\
      \rule{0em}{3ex}$G$&\rule{0em}{3ex}$G_c$&\rule{0em}{3ex}$G'$
    \end{tabular}
  \end{center}
  \caption{A graph $G$ with a cutting 2-join of type~2 and the
    associated graph $G'$\label{fig:cut2}}
\end{figure}

Unfortunately, there is another kind of path 2-join that can create
\even\ skew partitions when contracting the path-side.  A 2-join is
said to be \emph{cutting of type~2} if it has a split $(X_1,$ $X_2,$
$A_1,$ $B_1,$ $A_2,$ $B_2)$ such that there exist sets $A_3$, $B_3$
satisfying:

  \begin{enumerate}
  \item \label{cond.first}
    $(X_1, X_2)$ is a path 2-join with path-side $X_1$; 
  \item
    $A_3 \neq \emptyset$, $B_3 \neq \emptyset$, $A_3 \subset A_2$,
    $B_3 \subset B_2$;
  \item
    $A_3$ is complete to $B_3$;
  \item
    every outgoing path from $B_3\cup \{a_1\}$ to $B_3 \cup\{a_1\}$
    (resp. from $A_3\cup \{b_1\}$ to $A_3 \cup\{b_1\}$) has even
    length;
  \item \label{cond.penul} every antipath of length at least~2 with
    its ends outside  $B_3\cup \{a_1\}$ (resp. $A_3 \cup\{b_1\}$)
    and its interior in $B_3 \cup\{a_1\}$ (resp. $A_3\cup \{b_1\}$)
    has even length;
  \item \label{cond.disconnect}
    $G \setminus (X_1 \cup A_3 \cup B_3)$ is disconnected.
  \end{enumerate}

In Fig.~\ref{fig:cut2}, the structure of a graph $G$ with a cutting
2-join of type~2 is represented. After contracting the path-side into an
edge $a_1b_1$, we obtain a graph $G_c$ with \an\ \even\ skew cutset $\{a_1,
b_1\} \cup A_3 \cup B_3$. It is ``skew'' because $a_1\cup B_3$ is
complete to $b_1 \cup A_3$, and it is \even\ by the parity constraints
in the definition.  How can we find a graph smaller than $G$ that is
still a counter-example to the conjecture ?  Again, we find a graph
$G'$, also represented in Fig.~\ref{fig:cut2} and described formally in
Subsection~\ref{mainproofCase2}.  Again, we prove that $G'$ is a
smaller counter-example, a contradiction. This is the second case of
the main proof, described in Subsection~\ref{mainproofCase2}.

\label{why}

As mentioned in Section~\ref{algomotiv} we are not able to prove
something like Theorem~\ref{th.case} with ``skew partition'' instead
of ``\even\ skew partition''. Following our frame, we would have to
give up the conditions on the parity of paths in the definition of
cutting 2-joins of type~2. But then we would not be able to prove that
$G'$ is Berge, making the whole proof collapse. Also we would like to
explain a little twist in our proof. In fact Case~2 is not ``the
2-join is cutting of type 2'', but something slightly more general:
``the 2-join is such that there are sets $A_3$, $B_3$ satisfying the
items~\ref{cond.first}--\ref{cond.penul} of the definition of cutting
2-joins of type~2''. Indeed, in Case~2, we do not need to use the last
item. And this has to be done, since in Case~3, at some place where we
need a contradiction, we find a 2-join that is almost of type~2, that
satisfies items~\ref{cond.first}--\ref{cond.penul}, and not the last
one.

A 2-join is said to be \emph{cutting} if it is either cutting of
type~1 or cutting of type~2. So, in our main proof we can get rid of
cutting 2-joins as explained above.  In
Subsection~\ref{pathskewoverlap} we study how a 2-join and
\an\ \even\ skew partition can overlap in a Berge graph. The main
result of this subsection if Lemma~\ref{l.skew2join}. It says that
when contracting the path side of a non-cutting 2-join, no \even\ skew
partition is created. So if we come back to our main proof, we can at
last build $G'$ ``naturally'', that is by contracting the path-side of
the 2-join in $G$. This is the third case of the main proof, described in
Subsection~\ref{mainproofCase3}. Transforming $G$ into $G'$ will not
create a \even\ skew partition by Lemma~\ref{l.skew2join}. We need to
prove also that no 2-join is created. This might happen but then, an
analysis of the adjacencies in $G$ shows that $G$ has a 2-join that is
almost cutting of type~2 (``almost'' because the last item of the
definition of cutting 2-joins of type~2 does not hold). This is a
contradiction since we are in Case~3. But the contraction may create
other nasty things.

For instance suppose that $G$ is obtained by subdividing an edge of
the complement of a bipartite graph. Then, contracting the path-side
of the path 2-join of $G$ yields the complement of a bipartite
graph. This is why we have to view path-cobipartite graphs as basic in
our main theorem. Note that Chudnovsky also has to consider these
graphs as basic in her Theorem~\ref{th.2}.
 
Suppose now that $G$ is obtained from a double split graph $H$ by
subdividing matching edges of $H$ into paths of odd length. Such a
graph has a path 2-join whose contraction may yield a basic graph,
namely a double split graph. Let us define this more precisely.

We call \emph{flat path of a graph $H$} any path whose interior
vertices all have degree~2 in $H$ and whose ends have no common
neighbors outside  the path.  A \emph{path-double split graph} is
any graph $H$ that may be constructed as follows.  Let $m,n \geq 2$ be
integers. Let $A = \{a_1, \dots, a_m\}$, $B= \{b_1, \dots, b_m\}$, $C=
\{c_1, \dots, c_n\}$, $D= \{d_1, \dots, d_n\}$ be four disjoint
sets. Let $E$ be another possibly empty set disjoint from $A$, $B$,
$C$, $D$. Let $H$ have vertex set $A\cup B \cup C \cup D \cup E$ and
edges in such a way that:

\begin{itemize}
\item for every  vertex $v$ in $E$, $v$ has degree~2 and there exists
  $i \in \{1, \dots m\}$ such that $v$ lies on a
  path of odd length from $a_i$ to $b_i$; 
\item 
  for $1 \leq i \leq m$, there is a unique path of odd length
  (possibly~1) between $a_i$ and $b_i$ whose interior is in $E$.
  There are no edges between $\{a_i, b_i\}$ and $\{a_{i'}, b_{i'}\}$
  for $1\leq i < i' \leq m$;
\item 
  $c_j$ is non-adjacent to $d_j$ for $1 \leq j \leq n$. There are all
  four edges between $\{c_j, d_j\}$ and $\{c_{j'}, b_{j'}\}$ for
  $1\leq j < j' \leq n$;
\item
  there are exactly two edges between $\{a_i, b_i\}$ and $\{c_j,
  d_j\}$ for $1\leq i \leq m$, $1 \leq j \leq n$ and these two
  edges are disjoint.
\end{itemize}

Let us come back to our main proof. Adding path-cobipartite graphs and
path-double split graphs as basic graphs in our conjecture is not
enough. Because we need to prove that when contracting a path 2-join,
no 2-join in the complement is created, and that the counter-example
is not transformed into the complement of the line-graph of a
bipartite graph. And, unfortunately, both things may happen. But a
careful analysis of these phenomenons, done in the third case of the
main proof, Subsection~\ref{mainproofCase3}, shows that such graphs
have a special structure that we must add to our conjecture: a
\emph{homogeneous 2-join} is a partition of $V(G)$ into six non-empty
sets $(A,$ $B,$ $C,$ $D,$ $E,$ $F)$ such that:

\begin{itemize} 
\item 
  $(A, B, C, D, E, F)$ is a homogeneous pair;
\item 
  every vertex in $E$ has degree~2 and belongs to a flat path of odd
  length with an end in $C$, an end in $D$ and whose interior is in
  $E$;
 \item 
   every flat path outgoing from $C$ to $D$ and whose interior is in
   $E$ is the path-side of a non-cutting proper 2-join of $G$.
\end{itemize}

\noindent Now, we have defined all the new basic classes and
decompositions that we need.  Our main result is the following:

\begin{theorem}
  \label{th.th}
  Let $G$ be a Berge graph. Then either $G$ is basic, or one of $G,
  \overline{G}$ is a path-cobipartite graph,  or one of $G,
  \overline{G}$ is a path-double split graph, or one of $G,
  \overline{G}$ has a homogeneous 2-join, or one of $G, \overline{G}$
  has a non-path proper 2-join, or $G$ has \an\  \even\  skew partition.
\end{theorem}

Of course, in the proof sketched above, the graph $G$ is a
counter-example to Theorem~\ref{th.th}, not to the original
conjecture: ``Theorem~\ref{th.1} where path 2-joins are not
allowed''. So we need to be careful that our construction of graphs
$G'$ in cases~1, 2, 3 does not create a homogeneous 2-join and does not
yield a path-double split graph or a path-cobipartite graph. This
might have happened, and we would then have had to classify the
exceptions by defining new basic classes and decompositions, and this
would have lead us to a perhaps endless process. Luckily this process
ends up after just one step.

\begin{figure}[p]
  \begin{center}
    \includegraphics{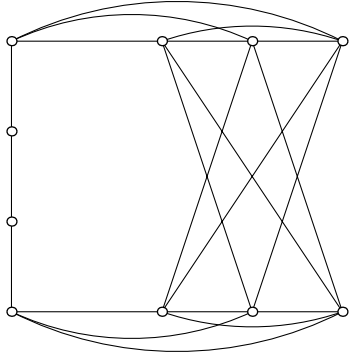}
    \caption{A path-cobipartite graph\label{fig:contrex1}}
  \end{center}
\end{figure}

\begin{figure}[p]
  \begin{center}
    \includegraphics{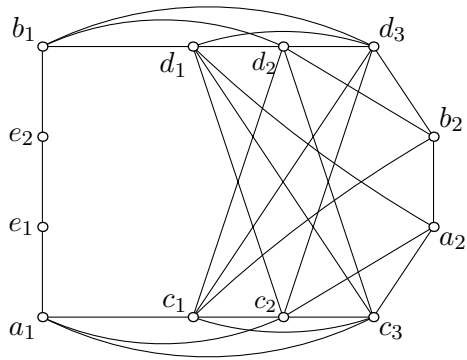}
    \caption{A path-double split graph\label{fig:contrex2}}
  \end{center}
\end{figure}

\begin{figure}[p]
  \begin{center}
    \includegraphics{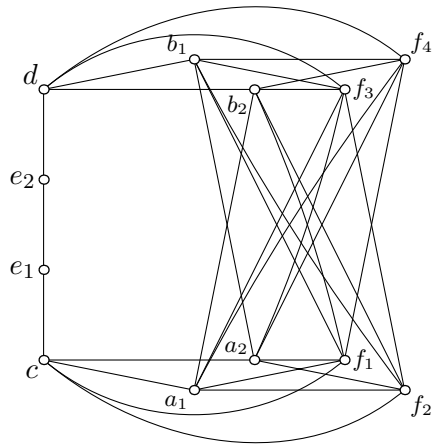}
    \caption{A graph that has a homogeneous 2-join $(\{a_1, a_2\},$
      $\{b_1, b_2\},$ $\{c\}, \{d\},$ $\{e_1, e_2\},$ $\{f_1, f_2,$ $f_3,
      f_4\})$\label{fig:contrex3}}
  \end{center}
\end{figure}

Theorem~\ref{th.th} generalizes Theorems~\ref{th.0},~\ref{th.1}
and~\ref{th.2}: path-cobipartite graphs may be seen either as graphs
having a proper path 2-join (Theorems~\ref{th.0} and~\ref{th.1}) or as
a new basic class (Theorem~\ref{th.2}). Path-double split graphs may
be seen as graphs having a proper path 2-join (Theorems~\ref{th.0}
and~\ref{th.1}) or as graphs having a non-\even\ skew partition
(Theorem~\ref{th.2}). And graphs having a homogeneous 2-join may be
seen as graphs having a homogeneous pair (Theorems~\ref{th.2} and
perhaps~\ref{th.0}) or as graphs having a proper path 2-join
(Theorems~\ref{th.1} and perhaps~\ref{th.0}). Formally all these
remarks are not always true: it may happen in special cases that
path-cobipartite graphs and path-double split graphs have no proper
2-join because the ``proper'' condition fails. But such graphs are
established in Lemma~\ref{l.thimplies} to be basic or to have
\an\ \even\ skew partition.

Note also that our new basic classes and decomposition yield
counter-examples to reckless extensions of Theorems~\ref{th.1}
and~\ref{th.2}. This needs a careful checking not worth doing here,
but let us mention it. The three graphs represented in
Fig.~\ref{fig:contrex1}, \ref{fig:contrex2}, \ref{fig:contrex3} are
counter-examples to our original conjecture, that is the extension of
Theorem~\ref{th.1} where path 2-joins are not allowed. Path-double
split graphs yield counter-examples to Theorem~\ref{th.2} with
``\even\ skew partition'' instead of ``skew partition'' (see
Fig.~\ref{fig:contrex2}). Graphs with a homogeneous 2-join yield
counter-examples to Theorem~\ref{th.2} where homogeneous pairs are not
allowed (see Fig.~\ref{fig:contrex3}). This shows that
Theorems~\ref{th.1},~\ref{th.2} are in a sense best possible, and that
to improve them, we need to do what we have done: add more basic
classes and decomposition. The three graphs represented in
Fig.~\ref{fig:contrex1}, \ref{fig:contrex2}, \ref{fig:contrex3} also
show that path cobipartite graphs, path-double split graphs and
homogeneous 2-join must somehow appear in our theorem, that is also in
a sense best possible.

This work suggests an algorithm for \ESPD\ with no reference to a new
decomposition theorem. Indeed, the graph $G'$ represented in
Fig.~\ref{fig:cut1} (resp. in Fig.~\ref{fig:cut2}) is a good candidate
to serve as a block of a cutting 2-joins of type~1 (resp. of
type~2). The fact that $G'$ is bigger than $G$ is not really a
problem, since the number of path 2-joins in a graph may be an
ingredient of a good notion of size. So, an algorithm might try to
deal with path 2-joins by constructing the appropriate block when the
2-join is recognized to be cutting. In fact this was our original idea
but it fails: we are not able to recognize cutting 2-joins of type~2.
To do this, we would have to guess somehow the sets $A_3, B_3$. But
this seems to be exactly the problem of detecting \even\ skew
partitions, so we are sent back to our original question. Perhaps an
astute recursive call to the algorithm would finally bypass this
difficulty, at the possible expense of a worse running time. Anyway,
we prefer to proceed as we have done, since a new decomposition for
Berge graphs is valuable in itself.

\section{Lemmas}
\label{lemmas}

The following is a useful characterization of line-graphs of bipartite
graphs:

\begin{theorem}[Harary and Holzmann \cite{harary.holzmann:lgbip}]
  \label{th.lgbg}
  $G$ is the line-graph of a bipartite graph if and only if $G$
  contains no odd hole, no claw and no diamond as induced subgraphs. 
\end{theorem}

\begin{figure}[h]
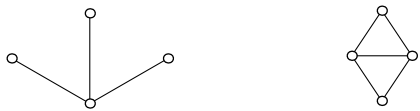

  \center
  \includegraphics{evenskew.6} \hspace{2cm}
  \includegraphics{evenskew.3}
  \caption{A claw  and a diamond\label{fig.cd}}
\end{figure}

The following fact is clear and useful:

\begin{lemma}
  \label{espcomp}
  If $(A, B)$ is \an\  \even\  skew partition of a graph $G$ then $(B, A)$
  is \an\  \even\  skew partition of $\overline{G}$. In particular, a graph
  $G$ has \an\  \even\  skew partition if and only if $\overline{G}$ has \an\ 
  \even\  skew partition.
\end{lemma}

A \emph{star} in a graph is a set of vertices $B$ such that there is a
vertex $x$ in $B$, called a \emph{center} of the star, seeing every
vertex of $B \setminus x$. Note that a star cutset of size at
least~2 is a skew cutset.

\begin{lemma}
  \label{l.starcutset}
  Let $G$ be a Berge graph of size at least~4, with at least one edge
  and that is not the complement of $C_4$. If $G$ has a star cutset
  then $G$ has \an\ \even\ skew partition.
\end{lemma}

\begin{proof}
  Let $B$ be a star cutset of $G$. Let us suppose $|B|$ being maximum
  with that property.  Let $A_1, A_2$ be such that $A_1, A_2, B$ are
  pairwise disjoint, there are no edges between $A_1, A_2$, and $A_1
  \cup A_2 \cup B = V(G)$.

  Suppose first that $B$ has size~1. Thus up to  symmetry $|A_1| \geq
  2$ since $G$ has at least~4 vertices. There is no edge between $B$
  and $A_1$ for otherwise such an edge would be a cutset which contradicts
  $|B|$ being maximum. There is no edge in $A_2$ since such an edge
  would be a cutset of $G$.  If there is no edge in $A_1$, any edge of
  $G$ is a cutset of $G$.  So, there is an edge $e$ in $A_1$. So,
  $|A_1|=2$ and $B$ is complete to $A_2$ for otherwise, $e$ would be a
  cutset of $G$. So, $|A_2| =1$ for otherwise, any edge between $B$
  and $A_2$ would be a cutset edge of $G$. Now, we observe that $G$ is the
  complement of $C_4$.

  If $B$ has size at least~2 then $B$ is a skew cutset of $G$.  Let
  $x$ be a center of $B$. By maximality of $B$, every component of $G
  \setminus B$ has either size~1 or contains no neighbor of $x$. Thus,
  if $P$ is a path that makes the skew cutset $B$ non-\even, then $P
  \cup x$ induces an odd hole of $G$. If $Q$ is an antipath that makes
  the skew cutset $B$ non-\even, then $Q \cup x$ induces an odd
  antihole of $G$.
\end{proof}

The following lemma is useful to establish formally that
Theorem~\ref{th.th} really implies Theorems~\ref{th.0},~\ref{th.1}
and~\ref{th.2}. But we also need it at several places in the next
section.

\begin{lemma}
  \label{l.thimplies}
  Let $G$ be a Berge graph. Then: 
  
  \begin{itemize}
  \item
    If $G$ has a flat path $P$ of length at least~3 then either $G$ is
    bipartite, or $G$ has \an\  \even\  skew partition or $P$ is the
    path-side of a proper path 2-join of $G$.
  \item
    If $G$ is a path-cobipartite graph, a path-double split graph or
    has a homogeneous 2-join, then either $G$ has a proper 2-join or
    $G$ has \an\ \even\ skew partition or $G$ is a bipartite graph, or
    the complement of a bipartite graph, or a double split graph.
  \end{itemize}
\end{lemma}

\begin{proof}
  Let us prove the first item. Let $P$ be a flat path of $G$ of length
  at least~3. So $(P, V(G)\setminus P)$ is a path 2-join of $G$. Let
  $(P, X_2, \{a_1\}, \{b_1\}, A_2, B_2)$ be a split of this 2-join. If
  $(P, X_2)$ is not proper, then either there is a component of $X_2$
  that does not meet any of $A_2$, $B_2$, or $X_2$ induces a path of
  length~1 or~2. In the last case, $G$ is bipartite, and in the first
  one, we may assume that there is a component $C$ of $X_2$ that does
  not meet $B_2$.  But then, $\{a_1\} \cup (A_2 \setminus C)$ is a
  star cutset of $G$ that separates $C$ from $B_2$, and so by
  Lemma~\ref{l.starcutset}, $G$ has \an\  \even\  skew partition.

  The second item follows easily: if $G$ is a path-cobipartite graph,
  then we may assume that $G$ is not the complement of a bipartite
  graph.  If $G$ is a path-double split graph then we may assume that
  $G$ is not a double split graph. In both cases, $G$ has a flat path
  of length at least~3.  If $G$ has a homogeneous 2-join then it also
  has a flat path of length at least~3. In every case, the conclusion
  follows from the first item.
\end{proof}

The following is well known for double split graphs (mentioned
in~\cite{chudnovsky.r.s.t:spgt}):

\begin{lemma} \label{l.dsg}
  A path-double split graph $G$ has exactly one skew partition and
  this skew partition is not \even. 
\end{lemma}

\begin{proof}
  Let $V(G)$ be partitioned into sets $A, B, C, D, E$ like in the
  definition of path-double split graphs. Obviously, $(A \cup B \cup
  E, C \cup D)$ is a non-\even\ skew partition of $G$. Every vertex of
  $A \cup B \cup E$ has a non-neighbor in every anticomponent of $C
  \cup D$. Hence, every subset of $V(G)$ strictly containing $C\cup
  D$ is anticonnected. So, if $X \neq C \cup D$ is a skew cutset of
  $G$, we may assume that $X$ does not contain $c_1$. So, $c_1$ is in
  a component of $G \setminus X$, and there is a vertex $y$ of $G$
  that is in another component. Up to  symmetry, we have two cases to
  consider:

  First case: $y = d_1$. Hence, every vertex of $C\cup D \setminus
  \{c_1, d_1\}$ must be in $X$. Every vertex in $A \cup B \cup E$ has
  a non-neighbor in every anticomponent of $C\cup D \setminus \{c_1,
  d_1\}$. So, since $X$ is not anticonnected, we have $X = C\cup D
  \setminus \{c_1, d_1\}$. This contradicts $G\setminus X$ being
  disconnected.

  Second case: $y$ is on a path $P$ from $a_1$ to $b_1$ whose interior
  is in $E$. Since $P$ has a vertex adjacent to $c_1$, at least one
  vertex of $P$ must be in $X$. If this vertex $u$ is in $E$ then we
  may assume up to  symmetry $b_1\in X$ since $u$ and $c_1$ must have
  a common neighbor in $X$ because $X$ is not anticonnected.  Else we
  may also assume $b_1 \in X$. Note that $a_1 \notin X$, because
  either $a_1$ and $b_1$ are not adjacent, and then cannot be both in
  $X$ because they have no common neighbor; or $a_1$ and $b_1$ are
  adjacent and then $y = a_1$ is the only possibility left for
  $y$. Hence, $X$ is a skew cutset that separates $a_1$ from
  $c_1$. Now, for every $2 \leq j \leq n$, one of $c_j$, $d_j$ is a
  common neighbor of $a_1, c_1$. Hence, up to  symmetry, we may
  assume $\{c_2, \dots, c_n\} \subset X$. Every vertex of $V(G)
  \setminus \{b_1, c_2, \dots, c_n\}$ has a non-neighbor in the unique
  anticomponent of $\{b_1, c_2, \dots, c_n\}$. Hence, $X = \{b_1, c_2,
  \dots, c_n\}$. So, $X$ is anticonnected. This contradicts $X$ being
  a skew cutset.
\end{proof}

\subsection{Paths and antipaths overlapping 2-joins}

Here, we state easy facts about parity of paths and antipaths
overlapping 2-joins. We need them to prove that when we build blocks of
a 2-join, the property of being balanced is preserved for every skew
cutset.  Some of the lemmas below are well known but they need to be
stated and proved clearly, especially because most of them are needed
for possibly non-proper 2-joins.

\begin{lemma}\label{l.2jAiBi}
  Let $G$ be a Berge graph with a connected 2-join $(X_1, X_2)$. Then
  all the paths with an end $A_1$, an end in $B_1$, no interior vertex
  in $A_1\cup B_1$, and all the paths with an end $A_2$, an end in $B_2$,
  no interior vertex in $A_2\cup B_2$ have the same parity.
\end{lemma}

\begin{proof}
  Note that since $(X_1, X_2)$ is connected there actually exists in
  $G[X_1]$ a path $P_1$ with an end in $A_1$, an end in $B_1$ and
  interior in $C_1$.  There exists a similar path $P_2$ in $G[X_2]$
  from $A_2$ to $B_2$. The paths $P_1, P_2$ have the same parity
  because $P_1 \cup P_2$ induces a hole. Let $P$ be a path from $A_1$
  to $B_1$ with no interior vertex in $A_1 \cup B_1$ (the proof is the
  same for a path from $A_2$ to $B_2$). Let $P^*$ be the interior
  of $P$. Then one of $P\cup P_2$, $P^* \cup P_1$ induces a
  hole. Hence, $P, P_1, P_2$ have the same parity.
\end{proof}

\begin{lemma}\label{l.2jAiAi}
  Let $G$ be a Berge graph with a 2-join $(X_1, X_2)$.  Let $i$ be in
  $\{1, 2\}$. Then every outgoing path from $A_i$ to $A_i$ (resp. from
  $B_i$ to $B_i$) has even length. Every antipath of length at least~2
  whose interior is in $A_i$ (resp. $B_i$) and whose ends are outside
  $A_i$ (resp. $B_i$) has even length.
\end{lemma}

\begin{proof}
  Note that we do not suppose $(X_1, X_2)$ being connected, so
  Lemma~\ref{l.2jAiBi} does not apply.  Let $P$ be an outgoing path
  from $A_1$ to $A_1$ (the other cases are similar).  If $P$ has a
  vertex in $A_2$, then $P$ has length~2.  Else, $P$ must lie entirely
  in $X_1$ except possibly for one vertex in $B_2$.  If $P$ lies
  entirely in $X_1$, then $P \cup \{a_2\}$ where $a_2$ is any vertex
  in $ A_2$ induces a hole, so $P$ has even length.  If $P$ has a
  vertex $b_2 \in B_2$, then we must have $P = \bp a \tp \cdots \tp b
  \tp b_2 \tp b' \tp \cdots \tp a'$ where $a \tp P \tp b $ and $b'\tp
  P \tp a'$ are paths with an end in $A_1$, an end in $B_1$ and
  interior in $C_1$.  Suppose that $P$ has odd length. Let $a_2$ be a
  vertex of $A_2$. Then $V(P) \cup \{a_2\}$ induces an odd cycle of
  $G$ whose only chord is $a_2 b_2$. So one of $V(a \tp P \tp b_2)
  \cup \{a_2\}$, $V(a' \tp P \tp b_2) \cup \{a_2\}$ induces an odd
  hole of $G$, a contradiction.

  Let $Q$ be an antipath of length at least~2 whose interior is in
  $A_1$ and whose ends are outside $A_1$ (the other cases are
  similar). If $Q$ has length at least~3, then the ends of $Q$ must
  have a neighbor in $A_1$ and a non-neighbor in $A_1$. Hence
  these ends are in $X_1$. Thus, $Q \cup \{a\}$, where $a$ is any
  vertex of $A_2$ is an antihole of $G$. Thus, $Q$ has even length.
\end{proof}


\begin{lemma}
  \label{PgivePX2}
  Let $G$ be a graph with a 2-join $(X_1, X_2)$.  Let $P$ be a path of
  $G$ whose end-vertices are in $X_2$.  Then either:
  \begin{enumerate}
  \item
    There are vertices $a\in A_1$, $b \in B_1$ such that $V(P)
    \subseteq X_2 \cup \{ a, b \}$. Moreover, if $a, b$ are both in
    $V(P)$, then they are non-adjacent.
  \item
    $P = c \tp \cdots \tp a_2 \tp a \tp \cdots \tp b \tp b_2 \tp
    \cdots \tp c'$ where: $a \in A_1$, $b \in B_1$, $a_2 \in A_2$,
    $b_2 \in B_2$. Moreover $V(c \tp P \tp a_2) \subset X_2$, $V(b_2
    \tp P \tp c') \subset X_2$, $V(a \tp P \tp b) \subset X_1$.
  \end{enumerate}
\end{lemma}

\begin{proof}
  If $P$ has no vertex in $X_1$, then for any $a\in A_1, b\in B_1$,
  the first outcome holds.  Else let $c, c'$ be the end-vertices of
  $P$.  Starting from $c$, we may assume that the first vertex of $P$
  in $X_1$ is $a \in A_1$. Note that $a$ is the only vertex of $P$ in
  $A_1$.  If $a$ has its two neighbors on $P$ in $X_2$, then $P$ has
  no other vertex in $X_1$, except possibly a single vertex $b \in
  B_1$ and the first outcome holds.  If $a$ has only one neighbor on
  $P$ in $X_2$, then let $a_2$ be this neighbor. Note that $P$ must
  have a single vertex $b$ in $B_1$. Let $b_2$ be the neighbor of $b$
  in $X_2$ along $P$. Vertices $a_2, a, b, b_2$ show that the second
  outcome holds.
\end{proof}

\begin{lemma}
  \label{PgivePA1}
  Let $G$ be a Berge graph with a 2-join $(X_1, X_2)$.  Let $P$ be a
  path of $G$ whose end-vertices are in $A_1 \cup X_2$ (resp.  $B_1
  \cup X_2$) and whose interior vertices are not in $A_1$
  (resp. $B_1$). Then either:
  \begin{enumerate}
  \item
    $P$ has even length.
  \item
    There are vertices $a\in A_1$, $b \in B_1$ such that $V(P)
    \subseteq X_2 \cup \{ a, b \}$. Moreover, if $a, b$ are both in
    $V(P)$, then they are non-adjacent.
  \item
    $P = a \tp \cdots \tp b \tp b_2 \tp \cdots \tp c$ where: $a \in
    A_1$, $b \in B_1$, $b_2 \in B_2$, $c \in X_2$.

    Moreover $V(a \tp P \tp b) \subset X_1$ and $V(b_2 \tp P \tp c)
    \subset X_2$.

    (resp. $P = b \tp \cdots \tp a \tp a_2 \tp \cdots \tp c$ where: $b
    \in B_1$, $a \in A_1$, $a_2 \in A_2$, $c \in X_2$.

    Moreover $V(b \tp P \tp a) \subset X_1$ and $V(a_2 \tp P \tp c)
    \subset X_2$.)
  \end{enumerate}
\end{lemma}

\begin{proof}
  Note that we do not suppose $(X_1, X_2)$ being proper.  Suppose that
  the end-vertices of $P$ are in $A_1 \cup X_2$ (the case when the
  end-vertices of $P$ are all in $B_1 \cup X_2$ is similar).

  If $P$ has its two end-vertices in $A_1$, then by
  Lemma~\ref{l.2jAiAi}, $P$ has even length and Output~1 of the lemma
  holds.

  If $P$ has exactly one end-vertex in $A_1$, let $a$ be this
  vertex. Let $c \in X_2$ be the other end-vertex of $P$.  Let $a'$ be
  the neighbor of $a$ along $P$.  If $a'$ is in $A_2$, then we may
  apply Lemma~\ref{PgivePX2} to $a' \tp P \tp c$: Outcome~2 is
  impossible and Outcome~1 yields Outcome~2 of the lemma we are
  proving now since $P$ has exactly one vertex in $A_1$. If $a'$ is
  not in $A_2$, then let $b$ be the last vertex of $X_1$ along $P$ and
  $b_2$ the first vertex of $X_2$ along $P$.  Outcome~3 of the lemma
  holds.
  
  If $P$ has no end-vertex in $A_1$ then Lemma~\ref{PgivePX2} applies
  to $P$.  The second outcome is impossible. The first outcome implies
  that there is a vertex $b \in B_1$ such that $V(P) \subseteq X_2
  \cup \{ b \}$ since no interior vertex of $P$ is in $A_1$. So,
  Outcome~2 of the lemma we are proving now holds.
\end{proof}


\begin{lemma}
  \label{antiPgivePX2}
  Let $G$ be a graph with a 2-join $(X_1, X_2)$.  Let $Q$ be an
  antipath of $G$ of length at least~4 whose interior vertices are all
  in $X_2$. Then there is a vertex $a$ in $A_1 \cup B_1$ such that
  $V(Q) \subseteq X_2 \cup \{a\}$.
\end{lemma}

\begin{proof}
  Let $c, c'$ be the end-vertices of $Q$. Note that $N(c) \cap N(c')
  \cap X_2$ have to be non-empty and that $N(c) \cap X_2$ must be
  different from $N(c') \cap X_2$, because $c, c'$ are the end-vertices
  of an antipath of length at least~4.  No pair of vertices in $X_1$
  satisfies these two properties, so at most one of $c,c'$ is in $V(Q)
  \cap X_1$.  If none of $c, c'$ are in $X_1$, then let $a$ be any
  vertex in $A_1$, else let $a$ be the unique vertex in $X_1$ among
  $c, c'$.  Since $c, c'$ must have a neighbor in $X_2$, $a\in A_1
  \cup B_1$ and clearly $V(Q) \subseteq X_2 \cup \{a\}$.
\end{proof}

\begin{lemma}
  \label{antiPgivePA1}
  Let $G$ be a Berge graph with a 2-join $(X_1, X_2)$.  Let $Q$ be an
  antipath of $G$ of length at least~5 whose interior vertices are all
  in $A_1 \cup X_2$ (resp.  $B_1 \cup X_2$) and whose end-vertices are
  not in $A_1$ (resp. $B_1$). Then either:
  \begin{enumerate}
  \item
    $Q$ has even length.
  \item
    There is a vertex $a \in A_1 \cup B_1$ such that $V(Q) \subseteq
    X_2\cup\{a\}$.
  \end{enumerate}
\end{lemma}

\begin{proof}
  We suppose that the interior vertices of $Q$ are all in $A_1 \cup
  X_2$.  The case when the interior vertices of $Q$ are all in $B_1
  \cup X_2$ is similar.
  
  If $Q$ has at least~2 vertices in $A_1$, then let $a\neq a'$ be two
  of these vertices. Since the end-vertices of $Q$ are not in $A_1$,
  $a, a'$ may be chosen in such a way that there are vertices $c, c'
  \notin A_1$ such that $\overline{c \tp a \tp \overline{Q} \tp a' \tp
    c'}$ is an antipath of ${G}$.  Since $c$ must miss $a$ while
  seeing $a'$, $c$ must be in $X_1 \setminus A_1$, and so is $c'$. But
  the interior vertices of $Q$ cannot be in $X_1 \setminus A_1$, so
  $c, c'$ are in fact the end-vertices of $Q$. Also, every interior
  vertex of $Q$ must be adjacent to at least one of $c, c'$.  If all
  the interior vertices of $Q$ are in $A_1$ then by
  Lemma~\ref{l.2jAiAi}, $Q$ has even length. Else, $Q$ must have at
  least one interior vertex $b \in X_2$. Since $b$ must see at least
  one of $c, c'$ we have $b\in B_2$, so $b$ misses both $a, a'$. Hence
  $\overline{a \tp b \tp a'}$ is an induced subgraph of $Q$ and $b$
  must see both $c, c'$, so $c, c' \in B_1$. Now we observe that $Q =
  \overline{c \tp a \tp b \tp a' \tp c'}$, which contradicts $Q$
  having a length of at least~5.

  If $Q$ has exactly one vertex $a$ in $A_1$ then by assumption, $a$
  is an interior vertex of $Q$.  Let $c, c'$ be the ends of
  $Q$. Suppose $c \in X_1$. Since $Q$ has length at least~5, $c$ must
  have a neighbor in the interior $Q$ that is different from $a$, hence
  $c \in B_1$. Since $Q$ has length at least~5, $a$ and $c$ must have
  a common neighbor, that must be $c'$ since it must be in
  $X_1$. Hence $c'\in X_1$, which implies $c' \in B_1$ since $c'$ must have
  a neighbor in $X_2$.  Now the non-neighbor of $c'$ along $Q$ is not
  $a$, so it must be a vertex of $X_2$ while seeing $c$ and missing
  $c'$, a contradiction. We proved $c \in X_2$, and similarly $c' \in
  X_2$. Hence $V(Q) \subset X_2 \cup \{a\}$.

  If $Q$ has no vertex in $A_1$ then Lemma~\ref{antiPgivePX2} applies:
  there is a vertex $a \in A_1 \cup B_1$ such that $V(Q) \subseteq
  X_2\cup\{a\}$.
\end{proof}

\subsection{\Even\ skew partitions overlapping 2-joins}
\label{pathskewoverlap}

\label{sec.defpiece}
Let $G$ be a Berge graph and $(X_1, X_2, A_1, B_1, A_2, B_2)$ be a
split of a proper 2-join of $G$. The \emph{blocks} of $G$ with respect
to $(X_1, X_2)$ are the two graphs $G_1, G_2$ that we describe now. We
obtain $G_1$ by replacing $X_2$ by a flat path $P_2$ from a vertex
$a_2$ complete to $A_1$, to a vertex $b_2$ complete to $B_1$.  This
path has the same parity than a path from $A_1$ to $B_1$ whose
interior is in $C_1$. There is such a path since $(X_1, X_2)$ is
proper and all such paths have the same parity by Lemma~\ref{l.2jAiBi}.
The length of $P$ is decided as follow: if $(X_1, X_2)$ is a path
2-join with path-side $X_2$ then $P$ has length~1 or~2, else it has
length~3 or~4.  The block $G_2$ is obtained similarly by replacing
$X_1$ by a flat path. The following lemma shows that blocks are
relevant for inductive proofs and recursive algorithms.

\begin{lemma}
  \label{l.blockberge}
  Let $G$ be a Berge graph and $(X_1, X_2)$ be a proper 2-join of
  $G$. Then the blocks $G_1, G_2$ of $G$ with respect to $(X_1, X_2)$
  are both Berge graphs.
\end{lemma}

\begin{proof}
  Note that $(V(G) \setminus X_2, X_2)$ is a connected 2-join of $G_2$
  (possibly non-substantial).  Let $a_1, b_1$ be the vertices of
  $G_2\setminus X_2$ respectively complete to $A_2, B_2$.  

  Let $H$ be a hole of $G_2$. Either $V(H) \subset V(X_2) \cup \{u\}$
  where $u \in \{a_1, b_1\}$ or $H$ is edge-wise partitioned into two
  paths from $a_1$ to $b_1$. In either cases, $H$ is even because it
  may be viewed as a hole of $G$ or by Lemma~\ref{l.2jAiBi} applied to
  $G$ and by definition of blocks. 

  Let $H'$ be an antihole of $G_2$ of length at least~7.  No vertex of
  $G_2\setminus (X_2 \cup \{a_1, b_1\})$ can be in $H'$ since these
  vertices all have degree~2. Also, $a_1, b_1$ cannot be both in $H'$
  because in $H'$, any pair of vertices has a common neighbor. Thus
  $H'$ may be viewed as an antihole of $G$ and is even.
\end{proof}

It is convenient to consider a degenerated kind of 2-join that implies
the existence of \an\  \even\  skew partition.  A 2-join $(X_1, X_2)$ is
said to be \emph{degenerate} if either:

\begin{itemize}
\item
  there exists $i \in \{1, 2\}$ and a vertex $v$ in $A_i$
  (resp. $B_i$) that has no neighbor in $X_i \setminus A_i$ (resp. in
  $X_i \setminus B_i$);
\item
  one of $A_1 \cup A_2$, $B_1 \cup B_2$ is a skew cutset of $G$;
\item 
  the 2-join $(X_1, X_2)$ is not connected (ie, there exists $i \in
  \{1, 2\}$ and a component of $X_i$ that does not meet both $A_i,
  B_i$);
\item
  there exists $i \in \{1, 2\}$ and a vertex in $A_i$ that is complete
  to $B_i$ or a vertex in $B_i$ that is complete to $A_i$;
\item
  there exists $i \in \{1, 2\}$ and a vertex in $C_i$ that is complete
  to $A_i \cup B_i$.
\end{itemize}

\begin{lemma}
  \label{l.degenerate}
  Let $G$ be a Berge graph and $(X_1, X_2)$ be a degenerate
  substantial 2-join of~$G$.  Then $G$ has \an\ \even\ skew
  partition. Moreover, if $(X_1, X_2)$ is proper then and at least one
  of the blocks $G_1, G_2$ of $G$ has \an\ \even\ skew partition.
\end{lemma}

\begin{proof}
  Let us look at the possible reasons why $(X_1, X_2)$ is
  degenerate. The following five paragraphs correspond to the five
  items of the definition of degenerate 2-joins.

  If there is a vertex $v$ in $A_1$ that has no neighbor in $X_1
  \setminus A_1$ then suppose first $|A_1|>1$.  So $(A_1 \setminus
  \{v\}) \cup A_2 $ is a skew cutset separating $v$ from the rest of
  the graph. Hence, in $\overline{G}$ there is a star cutset of center
  $v$, and by Lemmas~\ref{l.starcutset} and~\ref{espcomp}, $G$ has
  \an\ \even\ skew partition. Hence we may assume $A_1 = \{v\}$. Since
  $(X_1, X_2)$ is substantial, $|X_1| \geq 3$. Thus, for any $b\in
  B_1$, $\{b\} \cup B_2$ is a star cutset that separates $v$ from $X_1
  \setminus \{b, v\}$ and $G$ has \an\ \even\ skew partition by
  Lemma~\ref{l.starcutset}.  By the same way, the block $G_1$ has
  \an\ \even\ skew partition. The cases with $A_2, B_1, B_2$ are
  similar.

  If $A_1\cup A_2$ is a skew cutset of $G$ then let us check that this
  skew cutset is \even\ (the case when $B_1 \cup B_2$ is a skew cutset
  is similar). Since $A_1$ is complete to $A_2$, any outgoing path
  from $A_1 \cup A_2$ to $A_1 \cup A_2$ is either outgoing from $A_1$
  to $A_1$ or outgoing from $A_2$ to $A_2$. Thus, such a path has even
  length by Lemma~\ref{l.2jAiAi}.  If there is an antipath $Q$ of
  length at least~5 with its interior in $A_1 \cup A_2$ and its ends
  in the rest of the graph, then it must lie entirely in $X_1$ or
  $X_2$, say $X_1$ up to symmetry. Thus, such an antipath has even
  length by Lemma~\ref{l.2jAiAi}. By the same way $A_1 \cup \{a_2\}$,
  where $a_2$ is the vertex of $G_1$ that represents $A_2$, is
  \an\ \even\ skew cutset of $G_1$.

  If $(X_1, X_2)$ is not connected, then let for instance $Y$ be a
  component of $X_1$ that does not meet $B_1$. If $Y \cap C_1 \neq
  \emptyset$ then $A_1 \cup A_2$ is a skew cutset of $G$ that
  separates $Y \cap C_1$ from $B_1$. So, by the preceding paragraph,
  $G$ and $G_1$ have \an\ \even\ skew partition and we may assume that
  $Y \subset A_1$. Hence, every vertex in $Y$ has no neighbor in $X_1
  \setminus A_1$. So, by the penultimate paragraph, $G$ and $G_1$ have
  \an\ \even\ skew partition.

  If there is a vertex $a \in A_1$ that is complete to $B_1$ (the
  other cases are symmetric) then suppose first $|A_1| > 1$.  Consider
  $a'\neq a$ in $A_1$. Hence $(\{a\} \cup N(a)) \setminus a'$ is a
  star cutset of $G$ separating $a'$ from $B_2$.  So, by
  Lemma~\ref{l.starcutset}, we may assume $A_1 = \{a\}$.  If $|B_1| >
  1$, consider $b\neq b'$ in $B_1$.  Hence, $(\{b\} \cup N(b))
  \setminus b'$ is a star cutset of $G$ separating $b'$ from $A_2$. So
  we may assume $B_1 = \{b\}$.  Since $(X_1, X_2)$ is substantial,
  $|X_1|\geq 3$, and there is a vertex $c$ in $V(G) \setminus (A_1
  \cup B_1)$. Now, $\{a,b\}$ is a star cutset separating $c$ from
  $X_2$. By the same way, $G_1$ has \an\ \even\ skew partition.

  If there is a vertex $c$ complete to $A_i \cup B_i$ then we may
  assume $C_i = \{c\}$ for otherwise there would be another vertex
  $c'$ in $C_i$ and $\{c\} \cup A_i \cup B_i$ would be a star cutset
  separating $c'$ from the rest of the graph. By the preceding
  paragraph, we may assume that there is a vertex $a \in A_1$ and a
  vertex $b \in B_1$ missing $a$. Then $a \tp c \tp b$ is an outgoing
  path of even length from $A_i$ to $B_i$. By the penultimate
  paragraph, we may assume $(X_1, X_2)$ being connected. Thus by
  Lemma~\ref{l.2jAiBi}, there is no edge between $A_i$ and $B_i$. If
  there are two vertices $a \neq a' \in A_i$ then $\{a\} \cup N(a)
  \setminus \{a'\}$ is a star cutset of $G$ separating $a'$ from
  $B_{3-i}$. Thus may assume $|A_i| = 1$, and similarly $|B_i|=1$.
  Thus, $X_i$ is an outgoing path of length~2 from $A_i$ to $B_i$,
  which contradicts $(X_1, X_2)$ being substantial.  By the same way,
  one of $G_1$, $G_2$ has \an\ \even\ skew partition.
\end{proof}

\begin{lemma}
  \label{l.connect}
  Let $G$ be a graph with a non-degenerate 2-join $(X_1, X_2)$. Let
  $i$ be in $\{1, 2\}$. Then for every vertex $v \in X_i$ there is a
  path $P_a = \bp a \tp \cdots \tp v$ and a path $P_b = \bp b \tp
  \cdots \tp v$ such that:
  \begin{itemize}
  \item $a \in A_i$, $b \in B_i$;
  \item Every interior vertex of $P_a, P_b$ is in $X_i \setminus (A_i
    \cup B_i)$.
  \end{itemize}
\end{lemma}

\begin{proof}
  Note that $(X_1, X_2)$ is connected since it is not
  degenerate. Suppose first $v \in X_i \setminus (A_i \cup B_i)$.  By
  the definition of connected 2-joins, the connected component $X_v$
  of $v$ in $G[X_i]$ meets both $A_i$, $B_i$ and there is at least one
  path from $v$ to a vertex of $B_i$ in $G[X_i]$.  If every such path
  of $G[X_i]$ from $v$ to $B_i$ goes through $A_i$, then $A_i$ is a
  cutset of $G[X_i]$ that separates $v$ from $B_i$. Thus $A_1\cup A_2$
  is a skew cutset of $G$, so $(X_1, X_2)$ is degenerate, a
  contradiction. So there is a path $P_b$ as desired, and by the same
  way, $P_a$ exists.

  If $v\in A_i$, then $P_a$ exists and has length~0: put $P_a = v$.
  The vertex $v$ has a neighbor $w$ in $X_i \setminus A_i$ otherwise
  $(X_1, X_2)$ is degenerate.  By the preceding paragraph, there is a
  path $Q$ from $w$ to $b \in B_i$ whose interior vertices lie in $X_i
  \setminus (A_i \cup B_i)$.  So $P_b$ exists: consider a shortest
  path from $v$ to $b$ in $G[V(Q) \cup \{ b \}]$.
\end{proof}

\begin{lemma}
  \label{l.overlap}
  Let $G$ be a Berge graph with a non-degenerate 2-join $(X_1, X_2)$.
  Let $F$ be \an\ \even\ skew cutset of $G$. Then for some $i \in \{1,
  2\}$ either:
  \begin{itemize}
  \item
    $F \subsetneq X_i$;
  \item 
    $F \cap X_i \subsetneq X_i$ and one of $(F \cap X_i) \cup
    A_{3-i}$, $(F \cap X_i) \cup B_{3-i}$ is \an\  \even\  skew cutset of
    $G$.
  \end{itemize}
\end{lemma}

\begin{proof}
  We consider three cases:
  
  \noindent{\bf Case 1:} $F \cap A_1$, $F \cap A_2$, $F \cap B_1$, $F
  \cap B_2$ are all non-empty.

  If there is a vertex $a \in A_1\cap F$ non-adjacent to a vertex $b
  \in B_1 \cap F$ then there is an antipath of length at most~3
  between any vertex of $F$ and $a$, which contradicts $\overline{G}[F]$
  being disconnected. Thus $A_1 \cap F$ is complete to $B_1 \cap F$,
  and similarly $A_2 \cap F$ is complete to $B_2 \cap F$. It can be
  shown by similar techniques that $F \cap C_1 = F \cap C_2 =
  \emptyset$.  If $A_1 \subset F$ then there is a vertex in $B_1$ that
  is complete to $A_1$, which contradicts $(X_1, X_2)$ being
  non-degenerate. Thus $A_1 \setminus F \neq \emptyset$, and
  similarly $A_2 \setminus F \neq \emptyset$, $B_1 \setminus F \neq
  \emptyset$, $B_2 \setminus F \neq \emptyset$.

  Let $E_1$ be the component of $G \setminus F$ that contains
  $(A_1\setminus F) \cup (A_2 \setminus F)$. Let $E_2$ be another
  component of $G\setminus F$.  Up to  symmetry we assume $E_2 \cap
  X_2 \neq \emptyset$.  We claim that $F' = (F \cap X_2) \cup A_1$ is
  a skew cutset of $G$ that separates $E_1 \cap X_2$ from $E_2 \cap
  X_2$. For suppose not. This means that there is a path $P$ of $G
  \setminus F'$ with an end in $E_1\cap X_2$ and an end in $E_2 \cap
  X_2$.  If $P$ has no vertex in $X_1$ then $P \subset G \setminus F$
  and $P$ contradicts $E_1, E_2$ being components of $G \setminus
  F$. If $P$ has a vertex in $X_1$ then this vertex $b$ is unique and
  is in $B_1$ because $A_1 \subset F'$. By replacing $b$ by any vertex
  of $B_1 \setminus F$, we obtain again a path that contradicts $E_1,
  E_2$ being components of $G \setminus F$.  Thus $F'$ is a skew
  cutset of $G$. Note that this skew cutset is included in $A_1 \cup
  A_2 \cup B_2$. Let us prove that this skew cutset is \even.

  Let $P$ be an outgoing path from $F'$ to $F'$. Let us apply
  Lemma~\ref{PgivePA1} to $P$. If Outcome~1 of the lemma holds then
  $P$ has even length. If Outcome~2 of the lemma holds then $V(P)
  \subset X_2 \cup \{a, b\}$. Let $a_1$ be a vertex of $A_1 \cap F$
  and $b_1$ be a vertex of $B_1 \setminus F$ such that $a_1$ misses
  $b_1$. Note that $b_1$ exists for otherwise $(X_1, X_2)$ would be a
  degenerate 2-join of $G$.  After possibly replacing $a$ by $a_1$ and
  $b$ by $b_1$, we obtain an outgoing path from $F$ to $F$ that has
  the same length as $P$. Thus, $P$ has even length since $F$ is
  \an\ \even\ skew cutset. If Outcome~3 of the lemma holds then $P$
  has one end in $A_1$ and one end in $B_2$ and $P$ is a path from
  $A_1$ to $B_1$ whose interior is in $C_1$, plus one edge. Note that
  there is an edge between $A_2$ and $B_2$ so by Lemma~\ref{l.2jAiBi}
  every path from $A_1$ to $B_1$ whose interior is in $C_1$ has odd
  length. Hence in every case $P$ has even length.

  Let $Q$ be an antipath with both ends in $G \setminus F'$ and 
  interior in $F'$. If $Q$ has length~3 then $Q$ may be seen as an
  outgoing path from $F'$ to $F'$, so we may assume that $Q$ has
  length at least~5. By Lemma~\ref{antiPgivePA1} applied to $Q$,
  either $Q$ has even length or $V(Q) \subset X_2 \cup \{a\}$. If
  $a\in A_1$ let us replace $a$ by a vertex of $F \cap A_1$ and if $a
  \in B_1$ let us replace $a$ by a vertex of $B_1 \setminus F$. We
  obtain an antipath that has the same length as $Q$, that has both
  ends outside  $F$ and  interior in $F$. Thus $Q$ has even
  length because $F$ is \an\  \even\  skew cutset.

  \noindent{\bf Case 2:} one of $F \cap A_1$, $F \cap A_2$, $F \cap
  B_1$, $F \cap B_2$ is empty and $F \cap X_1$, $F \cap X_2$ are both
  non-empty. 

  We assume up to  symmetry that one of $B_1\cap F$, $B_2 \cap F$ is
  empty.  Since $F \cap X_1$ and $F \cap X_2$ are both non-empty,
  there is a least one edge between $F \cap X_1$ and $F \cap X_2$
  because $\overline{G}[F]$ is disconnected. Thus we know that $F \cap
  A_1$ and $F\cap A_2$ are both non-empty.  If $(F \cap X_1) \setminus
  A_1$ and $(F \cap X_2) \setminus A_2$ are both non-empty then there
  is a vertex of $F$ in one of $C_1, C_2$ since one of $B_1\cap F$,
  $B_2 \cap F$ is empty. Up to  symmetry, suppose $C_1 \cap F \neq
  \emptyset$. Then $\overline{G}[F]$ is connected since every vertex
  in it can be linked to a vertex of $C_1$ by an antipath of length at
  most~2, a contradiction. Hence one of $(F \cap X_1) \setminus A_1$
  and $(F \cap X_2) \setminus A_2$ is empty.  Thus we may assume $F
  \subset X_2 \cup A_1$.  Suppose $B_2 \subset F$. Then $B_2$ and $F
  \cap A_1$ are in the same component of $\overline{G}[F]$, thus there
  must be a vertex $v$ in $F$ that is complete to $B_2 \cup (F \cap
  A_1)$. So, $v$ is in $A_2$, and $v$ is complete to $B_2$,
  which contradicts $(X_1, X_2)$ being non-degenerate. We proved that
  there is at least one vertex $u$ in $B_2 \setminus F$. In particular,
  $F \cap X_2 \subsetneq X_2$. By Lemma~\ref{l.connect} there is a
  path from every vertex of $X_1 \setminus F$ to $u$ whose interior is
  in $X_1 \setminus A_1$, thus there is a component $E_1$ of $G
  \setminus F$ that contains $X_1 \setminus F$ and $u$. There is
  another component $E_2$ included in $X_2$. Thus $(F \cap X_2) \cup
  A_1$ is a skew cutset of $G$ that separates $B_1$ from $E_2$.  We
  still have to prove that the skew cutset $(F \cap X_2) \cup A_1$ is
  \even.

  Let $P$ be an outgoing path from $(F \cap X_2) \cup A_{1}$ to $(F
  \cap X_2) \cup A_{1}$. Let us apply Lemma~\ref{PgivePA1} to $P$. If
  Outcome~1 of the lemma holds then $P$ has even length. If Outcome~2
  of the lemma holds then $V(P) \subset X_2 \cup \{a, b\}$. Let $a_1$
  be a vertex of $A_1 \cap F$ and $b_1$ be a vertex of $B_1$ such that
  $a_1$ misses $b_1$. Note that $b_1$ exists for otherwise $(X_1,
  X_2)$ would be a degenerate 2-join of $G$.  After possibly replacing in
  $P$ $a$ by $a_1$ and $b$ by $b_1$, we obtain an outgoing path from
  $F$ to $F$ that has the same length as $P$. Thus, $P$ has even
  length since $F$ is \an\ \even\ skew cutset. If Outcome~3 of the
  lemma holds then $P = a \tp \cdots \tp b \tp b_2 \tp \cdots \tp
  c$. Let $a_1$ be in $A_1 \cap F$. By Lemma~\ref{l.connect} there is
  a path $P_1$ of $G[X_1]$ from $a_1$ to a vertex $b_1 \in
  B_1$. Moreover, $P_1$ has an end in $A_1$, an end in $B_1$ and
  interior in $C_1$. Note that by Lemma~\ref{l.2jAiBi}, $P_1$ and $a
  \tp P \tp b$ have the same parity. Thus $a_1 \tp P_1 \tp b_1 \tp b_2 \tp
  P \tp c$ is an outgoing path from $F$ to $F$ that has the same
  parity as $P$. Thus $P$ has even length.

  If $Q$ is an antipath with both ends in $G \setminus ((F \cap X_2)
  \cup A_1)$ and its interior in $(F \cap X_2) \cup A_1$, we prove
  that $Q$ has even length like in Case~1.
    
  \noindent{\bf Case 3:} One of $F \cap X_1, F \cap X_2$ is empty.
 
  Since $F \subsetneq X_2$ is an output of the lemma, we may assume up
  to  symmetry $F = X_2$ an look for a contradiction. If there is a
  path of odd length from $A_2$ to $B_2$ whose interior is in $C_2$,
  then there is by Lemma~\ref{l.2jAiBi} a similar path $P$ from $A_1$
  to $B_1$ of odd length. Hence $A_2$ is complete to $B_2$ because a
  pair of non-adjacent vertices yields together with $P$ an outgoing
  path of odd length from $F$ to $F$, which contradicts $F$ being
  \an\ \even\ skew cutset. In particular, there is a vertex of $A_2$
  that is complete to $B_2$, which implies $(X_1, X_2)$ being degenerate, a
  contradiction. If there is a path of even length from $A_2$ to $B_2$
  whose interior is in $C_2$ then by Lemma~\ref{l.2jAiBi} there are no
  edges between $A_2$ and $B_2$. Since $X_2=F$ is not anticonnected,
  there is a vertex in $C_2$ that is complete to $A_2 \cup B_2$,
  which implies again $(X_1, X_2)$ being degenerate, a contradiction.
\end{proof}

\begin{lemma}
  \label{l.skew2joind}
  Let $G$ be a Berge graph and $(X_1, X_2)$ be a proper 2-join of
  $G$. If $G$ has \an\ \even\ skew partition then at least one of the
  blocks of $G$ has \an\ \even\ skew partition.
\end{lemma}

\begin{proof}
  If $(X_1, X_2)$ is degenerate, then the conclusion holds by
  Lemma~\ref{l.degenerate}. From now on, we assume that $(X_1, X_2)$
  is non-degenerate.  Suppose that $G$ has \an\ \even\ skew partition
  $(E, F)$.  By Lemma~\ref{l.overlap} and up to  symmetry either $F
  \subsetneq X_2$, or $(F\cap X_2) \subsetneq X_2$ and $A_1 \subset F$,
  after possibly replacing $F$ by $(F \cap X_2) \cup A_1$.

  If $F \subsetneq X_2$ then we claim that $F$ is \an\ \even\ skew
  cutset of $G_2$. Note that there is at least one component $E$ of
  $G\setminus F$ that has some vertex in $X_2$ but no vertex in $A_2
  \cup B_2$. Else every component of $G\setminus F$ has neighbors in
  $A_1$ or $B_1$, and therefore contains $A_1 \cup B_1$ because $(X_1,
  X_2)$ is connected. This implies $G\setminus F$ being connected, a
  contradiction. Thus, $F$ is a skew cutset of $G_2$ that separates
  $E$ from $V(G_2) \setminus X_2$.  Let $P$ be an outgoing path of
  $G_2$ from $F$ to $F$. Note that $G_2$ has an obvious 2-join,
  $(V(G_2) \setminus X_2, X_2)$, possibly non-substantial.  Let us
  apply Lemma~\ref{PgivePX2} to $P$. If Outcome~1 of the Lemma holds
  then after possibly replacing $a$ by any $a_1 \in A_1$ and $b$ by
  any $b_1 \in B_1$ non-adjacent to $a_1$, $P$ may be viewed as an
  outgoing of $G$ from $F$ to $F$, thus $P$ has even length. Note that
  $b_1$ may be chosen non-adjacent to $a_1$ because $(X_1, X_2)$ is
  non-degenerate. If Outcome~2 of the lemma holds, then $P = c \tp
  \cdots \tp a_2 \tp a_1 \tp \cdots \tp b_1 \tp b_2 \tp \cdots \tp
  c'$. Let $P'$ be any path from $A_1$ to $B_1$ whose interior is in
  $C_1$.  Then $c \tp \cdots \tp a_2 \tp P' \tp b_2 \tp \cdots \tp c'$
  is an outgoing path of $G$ from $F$ to $F$ that has the same parity as
  $P$ by Lemma~\ref{l.2jAiBi}. Thus $P$ has even length.  Let $Q$ be
  an antipath of $G_2$ with its ends out of $F$ and its interior in
  $F$. Let us apply Lemma~\ref{antiPgivePX2} to $Q$: $V(Q) \subseteq
  X_2 \cup \{a\}$. Thus, after possibly replacing $a$ by a vertex in
  $A_1 \cup B_1$, $Q$ may be seen as an antipath of $G$ that has the same
  length as $Q$. Thus $Q$ has even length.

  If $(F\cap X_2) \subsetneq X_2$ and $A_1 \subset F$ then we put $F'
  = (F\cap X_2) \cup \{a_1\}$. We claim that $F'$ is \an\ \even\ skew
  cutset of $G_2$. Exactly as above, we prove that $F'$ is a skew
  cutset of $G_2$ that separates $b_1$ from a component of $G\setminus
  F$ that has vertices in $X_2$ but no vertex in $B_2$.  Let $P$ be an
  outgoing path from $F'$ to $F'$. By similar techniques it can be
  shown that $P$ has even length by Lemma~\ref{PgivePA1}.  Let $Q$ be
  an antipath of $G_2$ with its ends out of $F'$ and its interior in
  $F'$.  As above, we prove that $Q$ has even length by
  Lemma~\ref{antiPgivePA1}.
\end{proof}

\begin{lemma}
  \label{l.skew2join}
  Let $G$ be a Berge graph and $(X_1, X_2)$ be a non-cutting
  substantial 2-join of $G$. Then $G$ has \an\ \even\ skew partition
  if and only if one of the blocks of $G$ has \an\ \even\ skew
  partition.
\end{lemma}

\begin{proof}
  If $G$ has \an\ \even\ skew partition then by
  Lemma~\ref{l.skew2joind} one of the blocks of $G$ has
  \an\ \even\ skew partition.  If $(X_1, X_2)$ is degenerate, then the
  conclusion holds by Lemma~\ref{l.degenerate}. From now on, we assume
  that $(X_1, X_2)$ is non-degenerate. In particular, it is connected
  and proper. Let us suppose that one of $G_1, G_2$ (say $G_2$ up to 
  symmetry) has \an\ \even\ skew cutset $F'$. We denote by $P_1 = a_1
  \tp \cdots \tp b_1$ the path induced by $V(G_2) \setminus X_2$.  Note
  that $G_2$ has an obvious connected path 2-join: $(P_1, X_2)$,
  possibly non-substantial.

  \begin{claim}
    \label{cclem}
    Either: 
    \begin{itemize}
    \item
      $F' \subsetneq X_2$;
    \item 
      $F' \cap X_2 \subsetneq X_2$ and one of $(F' \cap X_2) \cup
      \{a_1\}$, $(F \cap X_2) \cup \{b_1\}$ is \an\ \even\ skew cutset
      of $G_2$.
    \end{itemize}
  \end{claim}

  \begin{proofclaim}
    If $P_1$ has length~3 or~4, then $(P_1, X_2)$ is proper. It is
    non-degenerate because $(X_1, X_2)$ is non-degenerate.  Let us
    apply Lemma~\ref{l.overlap}. The conclusion $F' \subsetneq X_1$,
    is impossible since then by Lemma~\ref{l.connect}, $G_2 \setminus
    F'$ is connected.  Also $(F'\cap P_1) \cup A_2$ and $(F'\cap P_1)
    \cup B_2$ cannot be skew cutsets of $G_2$, because $a_1, b_1$
    cannot be both in a skew cutset of $G_2$ since they are non
    adjacent with no common neighbors. Hence, Lemma~\ref{l.connect}
    proves that $(F'\cap P_1) \cup A_2$ and $(F'\cap P_1) \cup B_2$
    are not cutsets of $G_2$.  Thus~(\ref{cclem}) is simply the only
    possible conclusion of Lemma~\ref{l.overlap}.

    If $P_1$ has length~2 then $P_1 = a_1 \tp c_1 \tp b_1$. If $a_1,
    b_1$ are both in $F'$, then $F' = \{a_1, c_1, b_1\}$ because $c_1$
    is the only common neighbor of $a_1, b_1$ in $G_2$. This means
    that $G_2[X_2] = G[X_2]$ is disconnected, which implies that $(X_1,
    X_2)$ is a cutting 2-join of type~1, a contradiction. By
    Lemma~\ref{l.connect} applied to $G_2[X_2] = G[X_2]$, none of
    $a_1, b_1$ can be the center of a star cutset of $G$. Hence, $c_1
    \notin F'$. Thus, $F \cap X_2 \subsetneq X_2$ because any induced
    subgraph of $P_1$ containing $c_1$ is connected. We
    proved~(\ref{cclem}) when $P_1$ has length~2.

    We are left with the case when $P_1 = a_1 \tp b_1$. If $a_1, b_1$
    are both in $F'$ then $F' \subset \{a_1, b_1\} \cup A_2 \cup
    B_2$. If $F'\cap A_2 \neq \emptyset$ and $F'\cap B_2 \neq
    \emptyset$ then putting $A_3 = F' \cap A_2$ and $B_3 = F' \cap
    B_2$ we see that $(X_1, X_2)$ is a cutting 2-join of type~2 of
    $G$. Indeed, $A_3$ is complete to $B_3$ for otherwise, $F'$ would be
    anticonnected.  The requirements on the parity of paths and
    antipaths are satisfied because $F'$ is \an\ \even\ skew cutset.
    If at least one of $F' \cap A_2$ and $F' \cap B_2$ is empty then
    we see that $(X_1, X_2)$ is a cutting 2-join of type~1.  Both
    cases contradict $(X_1, X_2)$ being non-cutting. Thus we know that
    at most one of $a_1, b_1$ is in $F$. Also $F' \cap X_2 \subsetneq
    X_2$ because every induced subgraph of $P_1$ is connected.
  \end{proofclaim}

  By~(\ref{cclem}), we may assume that not both $a_1, b_1$ are in
  $F'$. Up to  symmetry, we assume $b_1 \notin F'$.  If $a_1 \in F'$,
  put $A'_1= A_1$, else put $A'_1 = \emptyset$.  Now $F = (F' \cap
  X_2) \cup A'_1$ is a skew cutset of $G$ that separates a vertex of
  $X_2$ from $X_1 \setminus A'_1$.  The proof that $F$ is
  \an\ \even\ skew cutset of $G$ is entirely similar to the similar
  proofs above: we consider an outgoing path of $G$ from $F$ to $F$.
  Lemma~\ref{PgivePX2} or Lemma~\ref{PgivePA1} shows that $P$ has the
  the same parity as an outgoing path of $G_2$ from $F'$ to $F'$.  We
  consider an antipath $Q$ of $G$ of length at least~2 with all its
  interior vertices in $F$ and with its end-vertices outside  $F$.
  Lemma~\ref{antiPgivePX2} or Lemma~\ref{antiPgivePA1} shows that $Q$
  has the the same parity as a similar antipath with respect to $F'$ in
  $G_2$.
\end{proof}

\begin{lemma}
  \label{l.skew2joinnonpath}
  Let $G$ be a Berge graph and $(X_1, X_2)$ be a non-path proper
  2-join of $G$. Then $G$ has \an\ \even\ skew partition if and only
  if one of the blocks of $G$ has \an\ \even\ skew partition.
\end{lemma}

\begin{proof}
  Clear by Lemma~\ref{l.skew2join} since a non-path 2-join is a
  non-cutting 2-join.
\end{proof}

\subsection{\Even\ skew partitions overlapping homogeneous 2-joins}

\noindent A homogeneous 2-join $(A, B, C, D, E, F)$ is said to be
\emph{degenerate} if either:

\begin{itemize}
  \item
    there is a vertex $x \in C$ with no neighbor in $E \cup D$ or a
    vertex $y \in D$ with no neighbor in $E \cup C$;
  \item
    there is a vertex $x\in C$ such that $N(x) \subset A \cup D \cup
    E$ or a vertex $y\in D$ such that $N(y) \subset B \cup C \cup E$.
\end{itemize}

\begin{lemma}
  \label{l.homod}
  Let $G$ be a Berge graph with a degenerate homogeneous 2-join. Then
  $G$ has \an\ \even\ skew partition.
\end{lemma}

\begin{proof}
  Suppose first that there exists a vertex $x \in C$ with no neighbor
  in $E \cup D$ (the case with $y\in D$ is similar). Then, $(A \cup C
  \cup F) \setminus \{x\}$ is a skew cutset that separates $x$ from
  the rest of the graph. Thus, $\overline{G}$ has a star cutset
  centered on $x$. By Lemma~\ref{l.starcutset}, $\overline{G}$ has
  \an\ \even\ skew partition and by Lemma~\ref{espcomp} so is $G$.
 
  Suppose now that there exists $x\in C$ such that $N(x) \subset A
  \cup D \cup E$ (the case with $y\in D$ is similar). Let $D_x$ be the
  set of those vertices of $D$ that are the ends of a path from $C$ to
  $D$ whose interior is in $E$ and starting from $x$.  Note that all
  such paths have odd length (possibly~1). If a vertex $f\in F$ misses
  $d \in D_x$, then consider a pair $a \in A, b \in B$ of non-adjacent
  vertices. Then $\{a, b, f\} \cup P$, where $P$ is a path from $x$ to
  $d$ whose interior is in $E$, induces an odd hole. Thus $F$ is
  complete to $D_x$.  Thus, for any $f \in F$, $\{f\} \cup N(F)
  \setminus B$ is a star cutset of $G$ that separates $x$ from
  $B$. Thus, by Lemma~\ref{l.starcutset}, $G$ has \an\ \even\ skew
  partition.
\end{proof}

The following little fact is needed twice in the proof of
Theorem~\ref{th.th}:

\begin{lemma}
  \label{l.twice}
  Let $G$ be a Berge graph. Suppose that $G$ has a vertex $u$ of
  degree~3 whose neighborhood induces a stable set. Moreover, $G$ has
  a stable set $\{x, y, z\}$ such that $x, y, z$ all have degree at
  least~3. Then $G$ is not a path-cobipartite graph, not a path-double
  split graph and $G$ has no non-degenerate homogeneous 2-join.
\end{lemma}

\begin{proof}
  In a path-cobipartite graph the vertices of degree at least~3
  partition into 2 cliques. Since $\{x, y, z\}$ contradicts this
  property, $G$ is not a path-cobipartite graph.

  In a path-double split graph, every vertex of degree exactly~3
   must have an edge in his neighborhood. Since $u$ contradicts
  this property, $G$ is not a path-double split graph.

  If $G$ has a non-degenerate homogeneous 2-join $(A,$ $B,$ $C,$ $D,$
  $E,$ $F)$, then every vertex in $F$ has degree at least~4. Every
  vertex in $A, B$ has an edge in his neighborhood. Every vertex in
  $C$ has a neighbor in $C$ or $F$ for otherwise, $(A,$ $B,$ $C,$ $D,$
  $E,$ $F)$ would be degenerate. Thus, every vertex in $C$, and by the same
  way every vertex in $D$, has an edge in his neighborhood. Every
  vertex in $E$ has degree~2. Hence, $u$ is in none of $A,$ $B,$ $C,$
  $D,$ $E,$ $F$, a contradiction.
\end{proof}

\section{Proof of Theorem~\ref{th.th}}
\label{proof}

For any graph $G$, let $f(G)$ be the number of maximal flat paths of
length at least~3 in $G$.  Let us consider $G$, a counter-example to
Theorem~\ref{th.th} such that $f(G) + f(\overline{G})$ is minimal.
Since $G$ is a counter-example and since $G$ is Berge, by
Theorem~\ref{th.1} and up to a complementation of~$G$, we may assume
that:

\renewcommand{\theenumi}{\alph{enumi}}

\begin{enumerate}
\item 
  $G$ is not basic, none of $G, \overline{G}$ is a path-cobipartite
  graph, none of $G, \overline{G}$ is a path-double split graph, $G$
  has no \even\ skew partition, none of $G, \overline{G}$ has a
  non-path proper 2-join, none of $G, \overline{G}$ has a homogeneous
  2-join;
  
\item
  $G$ has a path proper 2-join.
\end{enumerate}

\noindent 
Since $G$ has a path proper 2-join, $G$ has a flat path of length at
least~3, so $f(G) \geq 1$. We choose such a flat path $X_1$
inclusion-wise maximal. Note that by Lemma~\ref{l.thimplies}, $(X_1,
V(G) \setminus X_1)$ is a proper 2-join of $G$ since $G$ is not basic
and has no \even\ skew partition.  Let us consider $(X_1, X_2, A_1,
B_1, A_2, B_2)$ a split of this 2-join. Note that $G[X_2]$ is not a
path since $G$ is not bipartite. We denote by $a_1$ the only vertex in
$A_1$ and by $b_1$ the only vertex in $B_1$. We put $C_1 = X_1
\setminus \{a_1, b_1\}$, and $C_2 = X_2 \setminus (A_2 \cup B_2)$.

\noindent If one of $G$, $\overline{G}$ has a degenerate proper
2-join, a degenerate homogeneous 2-join or a star cutset then one of
$G, \overline{G}$ has \an\  \even\  skew partition by
Lemma~\ref{l.degenerate}, Lemma~\ref{l.homod} or
Lemma~\ref{l.starcutset}. So $G$ has \an\  \even\  skew partition by
Lemma~\ref{espcomp}.  This contradicts $G$ being a
counter-example. Thus:

\begin{enumerate}
\setcounter{enumi}{2}
\item \label{prop.degenerate}
  $G$ and $\overline{G}$ have no degenerate proper 2-join, no
  degenerate homogeneous 2-join and no star cutset.
\end{enumerate}

\noindent Suppose that $a_1$ has degree~2 in $G$. Since $X_1$ is the
path-side of a path 2-join, this means that the unique neighbor $a$ of
$a_1$ in $X_2$ sees at least one neighbor $b\in X_2$ of
$b_1$. Otherwise, $X_1 \cup \{a\}$ is flat path which contradicts $X_1$
being maximal. Hence, $b$ is a vertex of $B_2$ complete to $A_2 =
\{a\}$, which implies $(X_1, X_2)$ being degenerate, a contradiction.
Hence:

\begin{enumerate}
\setcounter{enumi}{3}
\item   \label{prop.deg3} 
  $a_1, b_1$ both have degree at least~3 in $G$. 
\end{enumerate}

\noindent Let us study the connectivity of $G$.  If $G[X_2]$ is
disconnected, then let $X'_2$ be any component of $G[X_2]$.  Since
$(X_1,X_2)$ is proper, the sets $A_2 \cap X'_2$ and $B_2 \cap X'_2$
are not empty.  So $(V(G) \setminus X'_2,X'_2)$ is a 2-join of $G$.
Let us suppose that $X'_2$ is not a path of length~1 or~2 from $A_2$
to $B_2$ whose interior is in $C_2$.  This implies that $(V(G)
\setminus X'_2,X'_2)$ is a proper 2-join.  So since $G$ is a
counter-example, we know that $(V(G) \setminus X'_2,X'_2)$ is a path
2-join of $G$. Since $X_1$ is a maximal flat path of $G$, $V(G)
\setminus X'_2$ cannot be the path side of this 2-join. Thus $G[X'_2]$
is the path side of this 2-join.  Hence we know that every component
of $X_2$ is a path from $A_2$ to $B_2$ whose interior is in $C_2$.
This implies that $G$ is bipartite, which contradicts $G$ being a
counter-example. Hence:

\begin{enumerate}
\setcounter{enumi}{4}
\item \label{prop.X2conn}
   $G[X_2]$ is connected.
\end{enumerate}

\noindent 
Since by Property~\ref{prop.degenerate}, $(X_1, X_2)$ is
non-degenerate, the following is a direct consequence of
Lemma~\ref{l.connect}:

\begin{enumerate}
\setcounter{enumi}{5}
\item \label{prop.conn} In $G[X_2]$, there exists a path from $A_2$ to
  $B_2$ whose interior is in $C_2$. Moreover, for every $A'_2
  \subseteq A_2$, $B'_2 \subseteq B_2$ the graphs $G[A'_2 \cup C_2
    \cup B_2 \cup \{b_1\}]$ and $G[B'_2 \cup C_2 \cup A_2 \cup
    \{a_1\}]$ are connected.
\end{enumerate}
  
\noindent
The six properties listed above will be referred as the
\emph{properties of $G$} in the rest of proof.  We denote by
$\varepsilon \in \{0, 1\}$ the parity of the length of the path
$X_1$. We now consider three cases according to the structure of the
2-join $(X_1, X_2)$. In each case, we will consider a graph $G'$
obtained from $G$ by destroying the path 2-join $(X_1, X_2)$, and we
will show that $G'$ is a counter-example that contradicts $f(G) +
f(\overline{G})$ being minimal.

\medskip
\subsection{Case 1: $X_1$ may be chosen in such a way that
$(X_1, X_2)$ is cutting of type~1.}  
\label{mainproofCase1}

  \setcounter{claim}{0}

  Up to  symmetry we assume that $G[X_2 \setminus A_2]$ is
  disconnected. Let $X$ be a component of $G[X_2 \setminus A_2]$. If
  $X$ is disjoint from $B_2$ then $\{a_1\} \cup A_2$ is a star cutset
  of $G$ separating $X$ from $X_2 \setminus X$, which contradicts the
  properties of $G$. Thus $X$ intersects $B_2$, and by the same proof
  so is any component of $X_2 \setminus X$. Hence, there are two
  non-empty sets $B_3 = B_2 \cap X$ and $B_4 = B_2 \setminus X$. Also
  we put $C_3 = C_2 \cap X$, $C_4 = C_2 \setminus X$. Possibly, $C_3$,
  $C_4$ are empty. There are no edges between $B_3 \cup C_3$ and $B_4
  \cup C_4$.

  We consider the graph $G'$ obtained from $G$ (see
  Fig.~\ref{fig:cut1}) by deleting $X_1 \setminus \{a_1, b_1\}$.
  Moreover, we add new vertices: $c_1, c_2, b_3, b_4$.  Then we add
  every possible edge between $b_3$ and $B_3$, between $b_4$ and
  $B_4$.  We also add edges $a_1 c_1$, $c_2 b_3$, $c_2 b_4$. If
  $\varepsilon = 0$, we consider for convenience $c_1 = c_2$, so that
  $c_1$ is always a vertex of $G'$. Else we consider $c_1 \neq c_2$
  and we add an edge between $c_1$ and $c_2$. Note that in $G'$,
  $N(b_1) = B_2$.  

  \begin{lemmaS}
    \label{cberge}
    $G'$ is Berge.
  \end{lemmaS}

  \begin{proof}

  \begin{claim}
    \label{clodd}
    Every path of $G'$ from $B_2$ to $A_2$ with no interior vertex in
    $A_2 \cup B_2$ has length of parity $\varepsilon$.
  \end{claim}

  \begin{proofclaim}
    If such a path contains one of $a_1, b_3, b_4, c_1, c_2$ then it
    has length~$4+\varepsilon$.  Else such a path may be viewed as a
    path of $G$ from $B_2$ to $A_2$. By Lemma~\ref{l.2jAiBi} it has
    parity~$\varepsilon$.
  \end{proofclaim}

  \begin{claim}
    \label{clevenB}
    Every outgoing path of $G'$ from $B_2$ to $B_2$ has even length.
  \end{claim}
  
  \begin{proofclaim}
    For suppose there is such a path $P = b \tp \cdots \tp b'$, $b, b'
    \in B_2$. If $P$ goes through $b_1$ then it has length~2.  If $P$
    goes through $b_3$ and $b_4$ it has length~4.  If $P$ goes through
    only one of $b_3, b_4$ then either $P$ has length~2 or we may
    assume up to  symmetry that $P= b \tp b_3 \tp c_2 \tp c_1 \tp a_1
    \tp a \tp \cdots \tp b'$ where $a \in A_2$.  So, $a \tp P \tp b'$
    is a path from $A_2$ to $B_2$ whose interior is in $C_2$ and
    by~(\ref{clodd}) it has parity~$\varepsilon$.  So, $P$ has even
    length.  If $P$ goes through $c_2$ or $c_1$ then it must goes
    through at least one of $b_3, b_4$, and by the discussion above it
    must have even length. So we may assume that $P$ goes through none
    of $c_1, c_2, b_1, b_3, b_4$.  Hence $P$ may be viewed as a path
    of $G$. Thus, $P$ has even length by Lemma~\ref{l.2jAiAi}.  In
    every case, $P$ has even length.
  \end{proofclaim}

  \begin{claim}
    \label{clevenA}
    Every outgoing path of $G'$ from $A_2$ to $A_2$ has even length.
  \end{claim}

  \begin{proofclaim}
    For suppose there is such a path $P = a \tp \cdots \tp a'$, where
    $a, a' \in A_2$.  If $P$ goes through $a_1$ then it has length~2.
    So we may assume that $P$ does not go through $a_1$. Note that if
    $c_1 \neq c_2$ then $P$ does not go through $c_1$.

    If $P$ goes through $c_2$ or through both $b_3, b_4$ then we may
    assume $P = a \tp \cdots \tp b \tp b_3 \tp c_2 \tp b_4 \tp b' \tp
    \cdots \tp a'$ where $b\in B_3$ and $b'\in B_4$.  By~(\ref{clodd})
    $b \tp P \tp a$ and $a' \tp P \tp b'$ have both
    parity~$\varepsilon$.  Thus, $P$ has even length. If $P$ goes
    through $B_3$, $b_1$ and $B_4$ then we prove that it has even
    length by the same way. So we may assume that $P$ neither goes
    through $c_2$ nor through both $b_3, b_4$ nor through $B_3$, $b_1$
    and $B_4$.

    If $P$ goes through exactly one of $b_3, b_4$, say $b_3$ up to 
    symmetry, then just like above $P = a \tp \cdots \tp b \tp b_3 \tp
    b' \tp \cdots \tp a'$, where both $b \tp P \tp a$ and $a' \tp P
    \tp b'$ are paths from $B_2$ to $A_2$.  So by~(\ref{clodd}), they
    both have parity~$\varepsilon$. Thus, $P$ has even length. If $P$
    goes through $b_1$ and exactly one of $B_3, B_4$, then we prove
    that it has even length by the same way. So we may assume that $P$
    goes though none of $b_1, b_3, b_4$.

    Now $P$ goes through none of $a_1, c_1, c_2, b_1, b_3, b_4$, so
    $P$ may be viewed as an outgoing path of $G$ from $A_2$ to
    $A_2$. It has even length by Lemma~\ref{l.2jAiAi}.

    In every case, $P$ has even length.
  \end{proofclaim}

  \begin{claim}
    \label{clevenB3}
    Every outgoing path of $G'$ from $B_3$ to $B_3$ (resp. $B_4$
    to $B_4$) has even length.
  \end{claim}

  \begin{proofclaim}  
    Suppose that there is an outgoing path $P = b \tp \cdots \tp b'$
    from $B_3$ to $B_3$ (the case with $B_4$ is similar).  Note that
    $P$ may have interior vertices in $B_4$, so~(\ref{clevenB}) does
    not apply to $P$.  If $P$ goes through $b_1$ or $b_3$ it has
    length~2.  So we may assume that $P$ does not go through $\{b_1,
    b_3\}$.  If $P$ has no vertex in $A_2$, then $P$ has no interior
    vertices in $B_4$ since $B_3$ and $B_4$ are in distinct components
    of $G\setminus (\{b_1, b_3\} \cup A_2)$.  So~(\ref{clevenB})
    applies and $P$ has even length.

    So we may assume that $P$ has at least one vertex in $A_2$.  Let us
    then call \emph{$B$-segment of $P$} every subpath of $P$ whose end
    vertices are in $B_2$ and whose interior vertices are not in
    $B_2$.  Note that $P$ is edgewise partitioned into its
    $B$-segment.  Similarly, let us call \emph{$A$-segment of $P$}
    every subpath of $P$ whose end-vertices are in $A_2$ and whose
    interior vertices are not in $A_2$.  By~(\ref{clevenA}), every
    $A$-segment has even length or has length~1.  An $A$-segment of
    length~1 is called an \emph{$A$-edge}.  Suppose that $P$ has odd
    length.  Let $b, b' \in B_2$ be the end-vertices of $P$.  Along
    $P$ from $b$ to $b'$, let us call $a$ the first vertex in $A_2$
    after $b$, and $a'$ the last vertex in $A_2$ before $b'$.  So $b
    \tp P \tp a$ and $a' \tp P \tp b'$ are both paths from $B_2$ to
    $A_2$, and by~(\ref{clodd}) they have the same parity.  So $a \tp P
    \tp a'$ is a path of odd length that is edgewise partitioned into
    its $A$-segment, and that contains all the $A$-segments of $P$.
    Thus $P$ has an odd number of $A$-edges.  Since $P$ is edgewise
    partitioned into into its $B$-segments, there is a $B$-segment
    $P'$ of $P$ with an odd number of $A$-edges. Let $\beta, \beta'$
    be the end-vertices of $P'$. Along $P'$ from $\beta$ to $\beta'$,
    let us call $\alpha$ the first vertex in $A_2$ after $\beta$, and
    $\alpha'$ the last vertex in $A_2$ before $\beta'$.  So $P'' =
    \alpha \tp P' \tp \alpha'$ is a path that is edgewise partitioned
    into its $A$-segment with an odd number of $A$-edge.  Thus $P''$
    has odd length.  Since $\beta \tp P \tp \alpha$ and $\alpha' \tp P
    \tp \beta'$ are both paths from $B_2$ to $A_2$, they have the same
    parity by~(\ref{clodd}).  Finally, $P'$ is of odd length, outgoing
    from $B_2$ to $B_2$, and contradicts~(\ref{clevenB}).  Thus $P$
    has even length.
  \end{proofclaim}


  \begin{claim}
    \label{clantiA2}
    Every  antipath of  $G'$  with  length at  least~2,  with its  end
    vertices in $V(G')\setminus A_2$, and all its interior vertices in
    $A_2$ has even length.
  \end{claim}

  \begin{proofclaim}
    Let $Q$ be such an antipath.  We may assume that $Q$ has length at
    least~3. So each end-vertex of $Q$ must have a neighbor in $A_2$
    and a non-neighbor in $A_2$.  So none of $a_1, c_1, c_2, b_1, b_3,
    b_4$ can be an end-vertex of $Q$, and $Q$ may be viewed as an
    antipath of $G$. So $Q$ has even length by Lemma~\ref{l.2jAiAi}.
  \end{proofclaim}

  \begin{claim}
    \label{clantiB2}
    Every  antipath of  $G'$  with  length at  least~2,  with its  end
    vertices in $V(G')\setminus B_2$, and all its interior vertices in
    $B_2$ has even length.
  \end{claim}

  \begin{proofclaim}
    Let $Q$ be such an antipath.  We may assume that $Q$ has length at
    least~3. So each end-vertex of $Q$ must have a neighbor in $B_2$
    and a non-neighbor in $B_2$.  So none of $a_1, b_1, c_1, c_2$ can
    be an end-vertex of $Q$. If $b_3$ is an end-vertex of $Q$, then
    the other end-vertex must be adjacent to $b_3$ while not being in
    $B_2 \cup \{a_1, b_1, c_1, c_2\}$, a contradiction.  So $b_3$ is
    not an end-vertex of $Q$ and by a similar proof, neither is
    $b_4$. So none of $a_1, c_1, c_2, b_1, b_3, b_4$ is in $Q$ and $Q$
    may be viewed as an antipath of $G$. So $Q$ has even length by
    Lemma~\ref{l.2jAiAi}.
  \end{proofclaim}

  \begin{claim}
    \label{clantiB3}
    Every antipath of $G'$ with length at least~2, with its end
    vertices in $V(G')\setminus B_3$ (resp. $V(G')\setminus B_4$), and
    all its interior vertices in $B_3$ (resp. $B_4$) has even length.   
  \end{claim}
  
  \begin{proofclaim}
    Let $Q$ be such an antipath whose interior is in $B_3$ (the case
    with $B_4$ is similar).  We may assume that $Q$ has length at
    least~3. So each end-vertex of $Q$ must have a neighbor in
    $B_3$. So no vertex of $B_4$ can be an end-vertex of $Q$.
    Thus~(\ref{clantiB2}) applies and $Q$ has even length.
  \end{proofclaim}

  \begin{claim}
    \label{clantip}
    Let $Q$ be an antipath of $G'$ of length at least~4. Then $Q$ does
    not go through  $c_1, c_2$. Moreover $Q$ goes  through at most one
    of $a_1, b_1, b_3, b_4$. 
  \end{claim}

  \begin{proofclaim}
    In an antipath of length at least~4, each vertex either is in a
    square of the antipath or in a triangle of the antipath.  So,
    $c_1, c_2$ are not in $Q$ since they are not in any triangle or
    square of $G'$.  In an antipath of length at least 4, for any pair
    $x,y$ of non-adjacent vertices, there must be a third vertex
    adjacent to both $x,y$.  Thus, $Q$ goes through at most one vertex
    among $a_1, b_3, b_4$. Suppose now that $Q$ also goes through
    $b_1$. Then it does not go through $a_1$ since $a_1, b_1$ have no
    common neighbors. So, up to  symmetry we may assume that $Q$
    goes through $b_3$ and $b_1$. There is no vertex in $G'\setminus c_2$
    seeing $b_3$ and missing $b_1$. So $b_1$ is an end of $Q$. Along
    $Q$, after $b_1$ we meet $b_3$. The next vertex along $Q$ must be
    in $B_4$. The next one, in $B_3$. The next one must see $b_3$ and
    must have a neighbor in $B_4$, a contradiction.
  \end{proofclaim}


    Let us now finish the proof of the lemma.  Let $H$ be a hole of
    $G'$.  Suppose first that $H$ goes through $a_1$.  If $H$ does not
    go through $c_1$, then $H \setminus a_1$ is a path of even length
    by~(\ref{clevenA}), so $H$ has even length.  If $H$ goes through
    $c_1$ then $H$ goes though exactly one of $b_3, b_4$, say $b_3$ up
    to symmetry, and $H\setminus \{a_1, c_1, c_2, b_3\}$ is a path
    $P$.  If $P$ does not go through $b_1$ then it has
    parity~$\varepsilon$ by~(\ref{clodd}). If $P$ goes through $b_1$,
    then $P = b \tp b_1 \tp b' \tp \dots \tp a$ where $b' \tp P \tp a$
    is from $B_4$ to $A_2$. So, again $P$ has parity~$\varepsilon$
    by~(\ref{clodd}).  So $H$ has even length and we may assume that
    $H$ does not go through $a_1$.  If $c_1 \neq c_2$ then $H$ does
    not go through $c_1$.  If $H$ goes through $c_2$ then the path $H
    \setminus \{b_3, c_2, b_4\}$ has even length by~(\ref{clevenB}),
    so $H$ is even. If $H$ goes through $b_1$ then the path $H
    \setminus \{b_1\}$ has even length by~(\ref{clevenB}), so $H$ is
    even. So we may assume that $H$ does not go through $b_1, c_2$.
    If $H$ goes through both $b_3, b_4$ then $H \setminus \{b_3,
    b_4\}$ is partitioned into two outgoing paths from $B_2$ to $B_2$
    that both have even length by~(\ref{clevenB}). Thus $H$ has even
    length. If $H$ goes through $b_3$ and not through $b_4$, then
    $H\setminus b_3$ is an outgoing path from $B_3$ to
    $B_3$. By~(\ref{clevenB3}) it has even length, so $H$ is even.  If
    $H$ goes through $b_4$ and not through $b_3$ then $H$ is even by a
    similar proof. So we may assume that $H$ goes through none of
    $b_3, b_4$.  Now, $H$ goes through none of $a_1, c_1, c_2, b_1,
    b_3, b_4$.  So $H$ may be viewed as a hole of $G$, and so it is
    even. So every hole of $G'$ is even.

    Let us now consider an antihole $H$ of $G'$.  Since the antihole
    on 5~vertices is isomorphic to $C_5$, we may assume that $H$ has
    at least~7 vertices.  Let $v$ be a vertex of $H$ that is not in
    $\{a_1, c_1, c_2, b_1, b_3, b_4\}$.  By~(\ref{clantip}) applied to
    $H\setminus \{v\}$, $H$ does not go through $c_1, c_2$ and goes
    through at most one vertex of $\{a_1, b_1, b_3, b_4\}$.  If $H$
    goes through $a_1$, the antipath $H\setminus a_1$ has all its
    interior vertices in $A_2$ and by~(\ref{clantiA2}),
    $H\setminus a_1$ has even length, thus $H$ is even.  If $H$ goes
    through $b_1$ then the antipath $H\setminus b_1$ has all its
    interior vertices in $B_2$ and by~(\ref{clantiB2}),
    $H\setminus b_1$ has even length, thus $H$ is even.  If $H$ goes
    through one of $b_3, b_4$, say $b_3$ up to  symmetry, the
    antipath $H\setminus b_3$ has all its interior vertices in $B_3$
    and by~(\ref{clantiB3}), $H\setminus b_3$ has even length,
    thus $H$ is even.  If $H$ goes through none of $a_1, c_1, c_2,
    b_1, b_3, b_4$ then $H$ may be viewed as an antihole of $G$. So
    every antihole of $G'$ has even length. Hence, $G'$ is Berge.
  \end{proof} 



  \begin{lemmaS}
    \label{cesp}
    \label{cd2join}
    $G'$ has no \even\ skew partition. Moreover, $G'$ and
    $\overline{G'}$ have no degenerate substantial 2-join, no
    degenerate homogeneous 2-join and no star cutset.
  \end{lemmaS}

  \begin{proof}
    Let $(F', E')$ be \an\  \even\  skew partition of $G'$ with a split
    $(E'_1, E'_2, F'_1, F'_2)$. Starting from $F'$, we shall build \an\ 
    \even\  skew cutset $F$ of $G$, which contradicts the properties of
    $G$.
    
    Let us first suppose $c_1\neq c_2$ and $c_1 \in F'$. Then, $F'$
    must contains at least one neighbor of $c_1$. If $F'$ contains $a_1$
    and not $c_2$, then $F'$ is a star cutset of $G'$ centered on
    $a_1$.  But this contradicts Property~\ref{prop.conn}
    of~$G$. If $F'$ contains $c_2$ and not $a_1$, then $F'$ is a star
    cutset of $G'$ centered on $c_2$.  But this again contradicts 
    Property~\ref{prop.conn} of~$G$. So, $F'$ must contain $a_1$ and
    $c_2$.  Since $a_1, c_2$ have no common neighbors we have $F' =
    \{a_1, c_1, c_2\}$. This is a contradiction since $G' \setminus
    \{a_1, c_1, c_2\}$ is connected by  Property~\ref{prop.conn}
    of~$G$.  So if $c_1 \neq c_2$ then $c_1 \notin F'$.
  
    Suppose $c_2 \in F'$.  By  Property~\ref{prop.conn} of~$G$, no
    subset of $\{c_2, b_3, b_4\}$ can be a cutset of $G$. So, $F'$
    must be a star cutset centered on one of $b_3, b_4$.  This again
    contradicts  Property~\ref{prop.conn} of~$G$.  So $c_2 \notin
    F'$.  Not both $b_3, b_4$ can be in $F'$ since they have no common
    neighbors in $F'$. So we assume $b_4 \notin F'$

    Up to  symmetry, we may assume $\{c_1, c_2, b_4\} \subset
    E'_1$. Also, $\{a_1, b_3\} \cap E' \subset E'_1$. We claim that
    $\{b_1\} \cap E' \subset E'_1$. Else, $F'$ separates $b_1$ from
    $c_2$. Hence we must have $B_4\subset F'$. Now $b_3 \in F'$ is
    impossible since there is no vertex seeing $b_3$ and having a
    neighbor in $B_4$. So, $B_3 \subset F'$. Since there is no edge
    between $B_3$ and $B_4$, there must be a vertex in $F'$ that is
    complete to $B_3 \cup B_4 = B_2$. The only place to find such a
    vertex is in $A_2$. But this implies $(X_1, X_2)$ being
    degenerate, which contradicts Property~\ref{prop.degenerate} of
    $G$.

    We proved $\{c_1, c_2, b_4\} \subset E'_1$ and $\{a_1, b_1, b_3\}
    \cap E' \subset E'_1$.  Let $v$ be any vertex of $E'_2$.  Since
    $\{a_1, c_1, c_2, b_1, b_3, b_4\} \cap E' \subset E'_1$, we have
    $v \in X_2$.  If $b_3$ is in $F$, put $B'_1= \{b_1\}$, else put
    $B'_1 = \emptyset$.  Now $F = (F' \setminus \{b_3\}) \cup B'_1$ is
    a skew cutset of $G$ that separates $v$ from the interior vertices
    of the path induced by $X_1$. Indeed, either $F=F'$, or $F'$ is
    obtained by deleting $b_3$ and adding $b_1$. Since $N(b_3) \cap
    X_2 \subset N(b_1) \cap X_2$, $F$ is not anticonnected and is a
    cutset.  It suffices now to prove that $F$ is \an\ \even\ skew
    cutset of $G$.

    Let $P$  be an  outgoing path of  $G$ from  $F$ to $F$.   We shall
    prove that $P$ has even length.

    If $a_1, b_1 \notin F$, then $F \subset X_2$ and the end-vertices
    of $P$ are both in $X_2$.  So Lemma~\ref{PgivePX2} applies to
    $P$. Suppose that the first outcome of Lemma~\ref{PgivePX2} is
    satisfied: $V(P) \subseteq X_2 \cup \{a_1, b_1 \}$.  Note that by
    the definition of $F$, $b_1 \notin F$ implies $b_1 \notin
    F'$. Hence, $P$ may be viewed as an outgoing path from $F'$ to
    $F'$, so $P$ has even length since $F'$ is \an\  \even\  skew cutset of
    $G'$.  Suppose now that the second outcome of Lemma~\ref{PgivePX2}
    is satisfied: $P = c \tp \cdots \tp a_2 \tp a_1 \tp X_1 \tp b_1
    \tp b_2 \tp \cdots \tp c'$. Put $i=3$ if $b_2 \in B_3$ and $i = 4$
    if $b_2 \in B_4$.  Put $P' = c \tp P \tp a_2 \tp a_1 \tp c_1 \tp
    c_2 \tp b_i \tp b_2 \tp P \tp c'$. Note that by the definition of
    $F$, $b_1 \notin F$ implies $b_3 \notin F'$. The paths $P$ and
    $P'$ have the same parity and $P'$ is an outgoing path of $G'$
    from $F'$ to $F'$.  So $P'$ and $P$ have even length since $F'$ is
    \an\  \even\  skew cutset of $G'$.

    If $a_1 \in F$, note that $b_1 \notin F$ since $a_1, b_1$ are
    non-adjacent with no common neighbors (in both $G, G'$).  We have
    $F' = F \subset X_2 \cup \{a_1\}$, the end-vertices of $P$ are
    both in $X_2 \cup \{a_1\}$ and no interior vertex of $P$ is in
    $\{a_1\}$ since $a_1 \in F$.  So Lemma~\ref{PgivePA1} applies. If
    Outcome~1 of the lemma holds, then $P$ has even length. If
    Outcome~2 of the lemma holds, then just like in the preceding
    paragraph, we can build a path $P'$ of $G'$ that is outgoing from
    $F$ to $F$ and that has a length with the same parity as $P$. So
    $P$ has even length.  If Outcome~3 of the lemma holds, the proof
    is again similar to the preceding paragraph.

    If $b_1 \in F$ then $a_1 \notin F$, $F \subset X_2 \cup \{b_1\}$,
    and Lemma~\ref{PgivePA1} applies. If Outcome~1 of the lemma holds,
    then $P$ has even length.  If Outcome~2 of the lemma holds, we may
    assume that $b_1$ is in $F \setminus F'$ and that $b_1$ is an
    end of $P$, for otherwise the proof would work like in the paragraph
    above. Then we build a path $P'$ of $G'$ that is outgoing from
    $F'$ to $F'$ and that has a length with the same parity as $P$, by
    replacing $\{b_1\}$ by $\{b_3\}$ (if $P$ goes through $B_3$) or by
    $\{b_3, c_2, b_4\}$ (if $P$ goes through $b_4$).  So $P$ has even
    length.  If Outcome~3 of the lemma holds then $P = b_1 \tp X_1 \tp
    a_1 \tp a_2 \tp \cdots \tp c$ where $a_2 \in A_2$, $c \in
    X_2$. Note that one of $b_1, b_3$ is in $F'$.  If $b_3 \in F'$,
    then we put $P' = b_3 \tp c_2 \tp c_1 \tp a_1 \tp a_2 \tp P \tp
    c$. If $b_3\notin F'$ then up to  symmetry, we assume $V(a_2 \tp
    P \tp c) \subset A_2 \cup C_3$. Note that $b_1 \in F'$.  We put
    $P' = b_1 \tp b \tp b_4 \tp c_2 \tp c_1 \tp a_1 \tp a_2 \tp P \tp
    c$ where $b$ is any vertex in $B_4$. It may happen that $P'$ is
    not a path of $G'$ because of the chord $a_2 b$.  But then we put
    $P'= b_1 \tp b \tp a_2 \tp P \tp c$. In every case, $P'$ is
    outgoing from $F'$ to $F'$, and has the same parity as $P$. Hence,
    $P$ has even length.

    Now, let $Q$ be an antipath of $G$ of length at least~2 with all
    its interior vertices in $F$ and with its end-vertices outside 
    $F$.  We shall prove that $Q$ has even length. Note that we may
    assume that $Q$ has length at least~5, because if $Q$ has
    length~3, it may be viewed as an outgoing path from $F$ to $F$,
    that has even length by the discussion above on paths.

    If both $a_1, b_1 \notin F$, then $F \subset X_2$ and the interior
    vertices of $Q$ are all in $X_2$.  So Lemma~\ref{antiPgivePX2}
    applies: $V(Q) \subseteq X_2 \cup \{a\}$ where $a \in \{a_1,
    b_1\}$.  So $Q$ may be viewed as an antipath of $G'$ that has even
    length because $F'$ is \an\  \even\  skew cutset of $G'$.
    
    If $a_1 \in F$, let us remind that $b_1 \notin F$.  We have $F
    \subset X_2 \cup \{a_1\}$, the interior vertices of $Q$ are in
    $X_2 \cup \{a_1\}$ and the end-vertices of $Q$ are not in
    $\{a_1\}$ since $a_1 \in F$.  So Lemma~\ref{antiPgivePA1}
    applies. We may assume that Outcome~2 holds. Once again, $Q$ may
    be viewed as an outgoing path of $G'$ that has even length because
    $F'$ is \even. 

    If $b_1 \in F$, we have to consider the case when $b_1 \notin F'$
    (else the proof is like in the paragraph above). Since $b_1 \notin
    F'$, we have $b_3 \in F'$. Note that $B_4 \cap F' = B_4 \cap F =
    \emptyset$ since there are no edges between $b_3, B_4$ and no
    vertex seeing $b_3$ while having a neighbor in $B_4$. So, if $Q$
    is an antipath whose interior is in $F$, then $Q$ does not go
    through $B_4$. Hence, if we replace $b_1$ by $b_3$, we obtain an
    antipath $Q'$ whose interior is in $F'$ and whose ends are
    not. Hence, $Q$ has even length.

    In every case, $Q$ has even length. We proved that $G'$ has no
    \even\ skew partition.  If one of $G'$, $\overline{G'}$ has a
    degenerate substantial 2-join, a degenerate homogeneous 2-join or
    a star cutset then $G'$ has \an\ \even\ skew partition by
    Lemma~\ref{l.degenerate}, \ref{l.homod} or~\ref{l.starcutset}, a
    contradiction.
  \end{proof}



  \begin{lemmaS}
    \label{c2jfinal}
    $G'$ has no proper non-path 2-join.
  \end{lemmaS}

  \begin{proof}

  \begin{claim}
    \label{cYZ}
    There exist no sets $Y_1, Z_1, Y_2, Z_2$ such that:
    \begin{itemize}
    \item
      $Y_1, Z_1, Y_2, Z_2$ are pairwise disjoint and $Y_1 \cup Z_1
      \cup Y_2 \cup Z_2 = X_2$;
    \item
      there are every possible edges between $Y_1$ and $Y_2$, and
      these edges are the only edges between $Y_1 \cup Z_1$ and $Y_2
      \cup Z_2$;
    \item
      $A_2 \subset Y_1 \cup Z_1$ and $B_2 \subset Y_2 \cup Z_2$. 
    \end{itemize}
  \end{claim}

  \begin{proofclaim}
    Suppose such sets exist. Note that $Y_1 \neq \emptyset$ and $Y_2
    \neq \emptyset$ since by Property~\ref{prop.X2conn} of $G$,
    $G[X_2]$ is connected. Note that $Z_1, Z_2$ can be empty. Suppose
    $Y_2 \cap B_2 \neq \emptyset$ and pick a vertex $b \in Y_2 \cap
    B_2$. Up to  symmetry we assume $b\in B_3$ and we pick a vertex
    $b' \in B_4$. Since $B_2 \subset Y_2 \cup Z_2$ we have $b'\in Y_2
    \cup Z_2$. Now $\{b\} \cup N(b)$ is a star cutset of $G$ that
    separates $a_1$ from $b'$, which contradicts the properties of $G$. Thus
    $Y_2 \cap B_2 = \emptyset$. Hence $(Y_2 \cup Z_2, V(G)\setminus
    (Y_2 \cup Z_2))$ is a 2-join of $G$. This 2-join is proper (the
    check of connectivity relies on the fact that $(X_1, X_2)$ is
    connected and on Lemma~\ref{l.connect}). By the properties of $G$,
    this 2-join has to be a path 2-join. Since $X_1$ is a maximal flat
    path of $G$, $Y_2 \cup Z_2$ is the path-side of the 2-join. This
    is impossible because $|B_2| \geq 2$.
  \end{proofclaim}

  Implicitly, when $(X'_1, X'_2)$ is a 2-join, we consider a split
  $(X'_1, X'_2, A'_1, B'_1, A'_2, B'_2)$.  We also put $C'_1 = X'_1
  \setminus (A'_1 \cup B'_1)$ and $C'_2 = X'_2 \setminus (A'_2 \cup
  B'_2)$.

  \begin{claim}
     \label{cc1c2}
     If $G'$ has a proper 2-join $(X'_1, X'_2)$ then either $\{c_1,
     c_2\} \subset X'_1$ or $\{c_1, c_2\} \subset X'_2$.
  \end{claim}

  \begin{proofclaim}
    Suppose not. We may assume that there is a 2-join $(X'_1, X'_2)$
    such that $c_1\in X'_2$ and $c_2\in X'_1$.  In particular, $c_1
    \neq c_2$. Up to  symmetry, we assume $c_1 \in A'_2$ and $c_2 \in
    A'_1$.  Then, $a_1\in X'_2$ for otherwise $c_1$ is isolated in
    $X'_2$, which contradicts $(X'_1, X'_2)$ being proper.  Also one of
    $b_3, b_4$ must be in $X'_1$ for otherwise $c_2$ would be isolated in
    $X'_1$. Up to symmetry we assume $b_3 \in X'_1$.

    By Property~\ref{prop.conn} of~$G$ there is a path $P = h_1
    \tp \cdots \tp h_k$ from $A_2$ to $B_3$ whose interior is in
    $C_2$, with $h_1\in A_2$, $h_k \in B_3$. We denote by $H$ the hole
    induced by $V(P) \cup \{a_1, c_1, c_2, b_3\}$. Note that $H$ has
    an edge whose ends are both in $X'_1$ (it is $c_2 b_3$) and an
    edge whose ends are both in $X'_2$ (it is $a_1 c_1$). So $H$ is
    vertex-wise partitioned into a path from $A'_1$ to $B'_1$ whose
    interior is in $X'_1$ and a path from $B'_2$ to $A'_2$ whose
    interior is in $X'_2$. Hence, starting from $c_1$, then going to
    $a_1$ and continuing along $H$, one will first stay in $X'_2$,
    will meet a vertex in $B'_2$, immediately after that, a vertex in
    $B'_1$, and after that will stay in $X'_1$ and reach $c_2$. We now
    discuss several cases according to the unique vertex $x$ in $H
    \cap B'_2$.

    If $x = a_1$ then $a_1 \in B'_2$. So $b_3 \in C'_1$. This implies
    step by step $B_3 \subset X'_1$, $B_3 \subset C'_1$, $b_1 \in
    X'_1$, $b_1 \in C'_1$, $B_4 \subset X'_1$, $B_4 \subset C'_1$,
    $b_4 \in X'_1$. Let $v$ be a vertex in $C_2$ (if any). Then by 
    Property~\ref{prop.conn} of~$G$ there is a path $Q$ from $v$ to
    $B_2$ with no vertex in $A_2$. If $v\in X'_2$, then $Q$ must
    contain a vertex in $A'_1 \cup B'_1$. This is impossible since no
    vertex in $C_2 \cup B_2$ sees $a_1$ or $c_1$. So, $C_2 \subset
    C'_1$. Let $v$ be a vertex in $A_2$. Note that by 
    Property~\ref{prop.conn} of~$G$, $v$ must have a neighbor in
    $C_2\cup B_2$. So, $v\in X'_1$ since $C_2 \cup B_2 \subset
    C'_1$. Finally, we proved $X'_2 = \{a_1, c_1\}$. This is
    impossible since $(X'_1, X'_2)$ is proper.

    If $x = h_i$ with $1 \leq i < k$, then $h_i \in B'_2\cap(A_2 \cup
    C_2)$ and $h_{i+1} \in B'_1$.  Note that $b_3 \in C'_1$ since
    $b_3$ misses $c_1$ and $h_1$. So, $B_3 \subset X'_1$. By the
    definition of $x$, we know that $a_1 \in C'_2$. So, $A_2 \subset
    X'_2$. We consider now two cases.

    First case: $b_4 \in X'_1$. Since there are no edges between
    $\{b_3, b_4\}$ and $\{c_1, h_i\}$ we know that $\{b_3, b_4\}
    \subset C'_1$. This implies $B_3\cup B_4 \subset X'_1$. Also, $b_1
    \in X'_1$ for otherwise $b_1$ would be isolated in $X'_2$. Now, $A'_1
    \cup B'_1 \subset (B_2 \cup C_2)$. Let us put: $Y_1 = B'_2$, $Z_1
    = (X'_2 \cap X_2) \setminus Y_1$, $Y_2 = B'_1$, $Z_2 = (X'_1 \cap
    X_2) \setminus Y_2$. These four sets yield a contradiction
    to~(\ref{cYZ}).

    Second case: $b_4 \in X'_2$. Then $b_4 \in A'_2$ and $A'_1 =
    \{c_2\}$. If there is a vertex $v$ of $X'_1$ in $B_4$ then $v\in
    A'_1$. This is impossible since $v$ misses $c_1 \in A'_2$. So,
    $B_4 \subset X'_2$. Hence, if $b_1\in X'_1$ then $b_1\in A'_1\cup
    B'_1$. But this is impossible since $b_1$ misses $c_1$ and
    $h_i$. So, $b_1 \in X'_2$. Since $B_3 \subset X'_1$, we know $B_3
    = B'_1$ and $b_1 \in B'_2$.  So $b_3$ is a vertex of $C'_1$
    complete to $A'_1 \cup B'_1$, which implies $(X'_1, X'_2)$ being
    degenerate, a contradiction. 

    If $x = h_k$ then $a_1 \in C'_2$ and $A_2 \subset X'_2$. Let $v$
    be a vertex of $C_2 \cup B_3 \cup B_4 \cup \{b_1, b_4\}$.  By 
    Property~\ref{prop.conn} of~$G$ there is a path $Q$ from $v$ to
    $A_2$ with no interior vertex in $B_3 \cup A_2$. If $v\in X'_1$,
    then $Q$ must have a vertex $u\neq v$ in $A'_2 \cup B'_2$. Note $u
    \notin B_3$. This is impossible because $u$ misses $c_2$ and
    $b_3$. So, $v\in X'_2$. Hence, $X'_1 = \{c_2, b_3\}$, which contradicts
    $(X'_1, X'_2)$ being proper.
  \end{proofclaim}

  \begin{claim}
     \label{ccb34}
     If $G'$ has a proper 2-join $(X'_1, X'_2)$ then either $\{c_1,
     c_2, b_3, b_4\} \subset X'_1$ or $\{c_1, c_2, b_3, b_4\} \subset
     X'_2$.
  \end{claim}

  \begin{proofclaim}
    Suppose not. By~(\ref{cc1c2}), we may assume that there is a
    2-join $(X'_1, X'_2)$ such that $c_1, c_2\in X'_1$ and $b_3\in
    X'_2$. Up to  symmetry, we assume $c_2 \in A'_1$ and $b_3 \in
    A'_2$.  At least one vertex of $B_3$ is in $X'_2$ for otherwise
    $b_3$ would be isolated in $X'_2$. So let $b$ be a vertex of $X'_2 \cap
    B_3$. We claim that there is a hole $H$ that goes through $b_3$,
    $c_2$, $c_1$, $a_1$, $h_1 \in A_2$, \dots $h_k=b$, with at least
    one edge in $X'_1$ and at least one edge in $X'_2$. If $c_1 \neq
    c_2$ then our claim hold trivially: $c_1c_2 \in X'_1$ and $b_3 b
    \in X'_2$. If $c_1 = c_2$, suppose that our claim fails. Then $a_1
    \in X'_2$, which implies $A'_1 = \{c_2\}$ and $a_1 \in A'_2$. We have
    $b_4 \in X'_1$ for otherwise $c_2$ would be isolated in $X'_1$. If $b_4
    \in B'_1$ then $(X'_1, X'_2)$ is degenerate since $b_4$ is
    complete to $A'_1$. So, $b_4 \in C'_1$, which implies $B_4 \subset
    X'_1$.  If $b_1 \in X'_1$ then $b \in B'_1$ since $b \in X'_2$. So
    $B'_2 \subset B_3$ and $b_3$ is a vertex of $A'_2$ that is
    complete to $B'_2$, which implies $(X'_1, X'_2)$ being degenerate, a
    contradiction.  So $b_1 \in X'_2$. Hence $B'_1 = B_4$ because no
    vertex of $B'_1$ can be in $B_3$ since $b_3 \in A'_2$. So $b_4 \in
    C'_1$ is complete to $A'_1 \cup B'_1$, which implies $(X'_1, X'_2)$
    being degenerate, a contradiction.  Thus our claim holds: $H$ has
    an edge in $X'_1$ and an edge in $X'_2$. So there is a unique
    vertex $x$ in $H \cap B'_2$. We now discuss according to the place
    of $x$.

    If $x = a_1$ then by the discussion above $c_1 \neq c_2$. Also,
    $a_1 \in B'_2$ and $c_1 \in B'_1$. Suppose that $X'_1 \cap X_2$
    and $X'_2 \cap X_2$ are both non-empty. The vertices of $A'_2 \cup
    B'_2$ are not in $X_2$ because they have to see either $c_1$ or
    $c_2$. So there are no edges between $X'_1 \cap X_2$ and $X'_2
    \cap X_2$. Hence, $G'[X_2]$ is not connected, which contradicts 
    Property~\ref{prop.X2conn} of $G$. So either $X_2 \subset X'_1$ or
    $X_2 \subset X'_2$. If $X_2 \subset X'_1$ then $X'_2 \subset
    \{a_1, b_1, b_3, b_4\}$, so $X'_2$ is a stable set, which contradicts
    $(X'_1, X'_2)$ being proper. If $X_2 \subset X'_2$ then $b_1$ is
    in $X'_2$ for otherwise it would be isolated in $X'_1$. So, $X'_1
    \subset \{c_1, c_2, b_4\}$. This is a contradiction since no
    subset of $\{c_1, c_2, b_4\}$ can be a side of a proper 2-join of
    $G'$.

    If $x=h_1$ then $h_1 \in B'_2$ and $a_1\in B'_1$.  If $b_4 \in
    X'_1$ then $b_4 \in C'_1$ because of $b_3$ and $h_1$. So, $B_4
    \subset X'_1$. But in fact, by the same way, $B_4 \subset C'_1$,
    and $b_1 \in C'_1$. So, $B_3 \subset X'_1$, which contradicts $h_k \in
    X'_2$.  We proved $b_4 \in X'_2$, which implies $A'_1 = \{c_2\}$. If a
    vertex $v$ of $X_2 \cup \{b_1\}$ is in $X'_1$, then by
    Lemma~\ref{l.connect} applied to $(X'_1, X'_2)$ there is a path of
    $X'_1$ from $v$ to $A'_1 = \{c_2\}$ with no interior vertex in
    $B'_1$, a contradiction. So $X_2 \cup \{b_1\} \subset X'_2$. We
    proved $X'_1 = \{a_1, c_1, c_2\}$, which contradicts $(X'_1, X'_2)$
    being proper.

    If $x=h_i$, $2 \leq i \leq k$ then $h_i\in B'_2$, $h_{i-1} \in
    B'_1$. Since $a_1\in C'_1$ we have $A_2 \subset X'_1$. If $b_4 \in
    X'_1$ then $b_4 \in C'_1$, which implies $B_4 \subset X'_1$. If $b_1 \in
    X'_2$ then $b_1$ must be in $A'_2 \cup B'_2$, a contradiction
    since $b_1$ misses $c_2$ and $h_{i-1}$. So, $b_1 \in X'_1$. Since
    $h_k\in X'_2$, we know $b_1 \in B'_1$. Thus $B'_2 \subset
    B_3$. Hence $b_3$ is a vertex of $A'_2$ that is complete to
    $B'_2$, which implies $(X'_1, X'_2)$ being degenerate, a
    contradiction. We proved $b_4 \in X'_2$. Now $A'_2 = \{b_3,
    b_4\}$. Suppose that there is a vertex $v$ of $X'_1$ in $B_3 \cup
    B_4$.  Then $v$ must be in $A'_1$ since $v$ sees one of $b_3,
    b_4$. But this is a contradiction since $v$ misses one of $b_3,
    b_4$. We proved $B_3 \cup B_4 \subset X'_2$. Also, $b_1 \in X'_2$
    for otherwise, $b_1$ would be isolated in $X'_1$.  Let us put: $Y_1 =
    B'_1$, $Z_1 = (X'_1 \cap X_2) \setminus Y_1$, $Y_2 = B'_2$, $Z_2 =
    (X'_2 \cap X_2) \setminus Y_2$. These four sets yield a
    contradiction to~(\ref{cYZ}).
  \end{proofclaim}

  \begin{claim}
     \label{c2jb1}
     If $G'$ has a proper 2-join $(X'_1, X'_2)$ then either $\{c_1, c_2, b_1,
     b_3, b_4\} \subset X'_1$ or $\{c_1, c_2, b_1, b_3, b_4\} \subset
     X'_2$. 
  \end{claim}

  \begin{proofclaim}
    Suppose not. By~(\ref{ccb34}), we may assume that there is a
    2-join $(X'_1, X'_2)$ of $G'$ such that $c_1, c_2, b_3, b_4 \in
    X'_1$ and $b_1 \in X'_2$. If $\{b_3, b_4\} \cap (A'_1 \cup B'_1) =
    \emptyset$ then $\{b_3, b_4\} \subset C'_1$, so $B_3 \cup B_4
    \subset X'_1$. Hence $b_1$ is isolated in $X'_2$, a contradiction.

    If $|\{b_3, b_4\} \cap (A'_1 \cup B'_1)| = 1$, then up to 
    symmetry we may assume $b_3 \in A'_1$ and $b_4\in C'_1$. Thus $B_4
    \subset X'_1$. Since $b_1 \in X'_2$, we have $B_4 \subset A'_1
    \cup B'_1$. But no vertex $x$ of $B_4$ can be in $A'_1$ because
    $x$ and $b_3$ have no common neighbors, so $B_4 \subset
    B'_1$. Thus $b_1 \in B'_2$. Because of $b_3$, $A'_2 \subset
    B_3$. So $b_1$ is a vertex of $B'_2$ that is complete to $A'_2$,
    which implies $(X'_1, X'_2)$ being degenerate, a contradiction. We proved
    $\{b_3, b_4\} \subset (A'_1 \cup B'_1)$. 

    Since $b_3, b_4$ have no common neighbors in $X'_2$, we may assume
    up to  symmetry that $b_3 \in A'_1$ and $b_4 \in B'_1$. So $b_1$
    have non-neighbors in both $A'_1, B'_1$. This implies $b_1 \in
    C'_2$, and $B_3 \cup B_4 \subset X'_2$. Hence $A'_2 = B_3$ and
    $B'_2 = B_4$. Now, $b_1 \in C'_2$ is complete to $A'_2 \cup B'_2$,
    which implies $(X'_1, X'_2)$ being degenerate, a contradiction. 
  \end{proofclaim}

  Let us now finish the proof.

    Let $(X'_1, X'_2)$ be a proper 2-join of $G'$. By~(\ref{c2jb1}),
    we may assume $\{c_1, c_2, b_1, b_3, b_4\} \subset X'_2$.  If $b_3
    \notin C'_2$ and $b_4 \notin C'_2$ then up to  symmetry we may
    assume $b_3 \in A'_2$, $b_4 \in B'_2$ since $b_3, b_4$ have no
    common neighbors in $X'_1$. So, there is a vertex of $A'_1$ in
    $B_3$ and a vertex of $B'_1$ in $B_4$, which implies $b_1 \in A'_2 \cap
    B'_2$, a contradiction. We proved $b_3 \in C'_2$ or $b_4 \in
    C'_2$. Up to  symmetry we assume $b_3 \in C'_2$, so $B_3
    \subset X'_2$.  Note that $X'_1$ is a subset of $V(G)$. If $A'_1
    \cap B_4, B'_1 \cap B_4$ are both non-empty then $b_1$ must be in
    $A'_2 \cap B'_2$, a contradiction. Thus we may assume $A'_1 \cap
    B_4 = \emptyset$. If $a_1\in X'_1$ and $B'_1 \cap B_4 \neq
    \emptyset$ then $a_1 \notin B'_1$ since $a_1$ misses $b_1$. Thus
    we may assume $B'_1 \cap \{a_1\} = \emptyset$.

    Let us now put: $X''_1 = X'_1$, $X''_2 = V(G) \setminus X''_1$,
    $A''_1 = A'_1$, $B''_1 = B'_1$, $B''_2 = B'_2 \setminus
    \{b_4\}$. If $a_1 \in A'_1$ then $A''_2 = (A'_2 \cap X_2) \cup
    (N_G(a_1) \cap X_1)$ else $A''_2 = A'_2$. Note that $A''_2 \cap
    B''_2 = \emptyset$. Also, if $b_4 \in B'_2$ then $b_1 \in B'_2$
    and $b_1 \in B''_2$. From the definitions it follows that $(X''_1,
    X''_2)$ is a partition of $V(G)$, that $A''_1, B''_1 \subset
    X''_1$, $A''_2, B''_2 \subset X''_2$, that $A''_1$ is complete to
    $A''_2$, that $B''_1$ is complete to $B''_2$ and that there are no
    other edges between $X''_1$ and $X''_2$.  So, $(X''_1, X''_2)$ is
    a 2-join of $G$.


    We claim that $(X''_1, X''_2)$ is a proper 2-join of $G$.  Note
    that $G[X''_1]$ is not a path of length~1 or~2 from $A''_1$ to
    $B''_1$ whose interior is in $C''_1$, because $X''_1 = X'_1$ and
    because $(X'_1, X'_2)$ is a proper 2-join of $G'$. Also $G[X''_2]$
    is not a path from $A''_2$ to $B''_2$ whose interior is in $C''_2$
    because $b_1$ has at least 2 neighbors in $X''_2$ (one in $X_1$,
    one in $B_3$) while having degree at least~3 because of
    $B_4$. Hence $(X''_1, X''_2)$ is substantial. So it is connected
    and proper for otherwise it would be degenerate, which contradicts the
    properties of $G$. This proves our claim.

    Since $(X''_1, X''_2)$ is proper, we know by the properties of $G$
    that $(X''_1, X''_2)$ is a path 2-join of $G$.  If $X''_2$ is the
    path-side of $(X''_1, X''_2)$ then $b_1$ is an interior vertex of
    this path while having degree at least~3, a contradiction. Hence,
    $X''_1$ is the path-side of $(X''_1, X''_2)$.  Since $X''_1=X'_1$,
    $(X'_1, X'_2)$ is a path 2-join of $G'$.
  \end{proof}


  \begin{lemmaS}
    \label{cgbar1}
    $\overline{G'}$ has no proper 2-join.
  \end{lemmaS}

  \begin{proof}
    Here the word ``neighbor'' refers to
    the neighborhood in $\overline{G'}$. Let $(X'_1, X'_2)$ be a
    proper 2-join of $\overline{G'}$.
    
    Suppose $c_1\neq c_2$.  In $\overline{G'}$, $c_1$ has degree
    $n-3$, so up to  symmetry we may assume $c_1\in A'_1$. In $B'_2$
    there must be a non-neighbor of $c_1$. Also, since $(X'_1, X'_2)$
    cannot be a degenerate 2-join of $\overline{G'}$, vertex $c_1$
    must have a non-neighbor in $B'_1$. So we have two cases to
    consider.  Case~1: $a_1 \in B'_1$, $c_2 \in B'_2$.  Then $c_2$
    must have a non-neighbor in $A'_2$ for otherwise $(X'_1, X'_2)$ would be
    degenerate. This non-neighbor must be one of $b_3, b_4$. But this
    is impossible since $b_3, b_4$ both see $a_1$ in $\overline{G'}$.
    Case~2: $a_1 \in B'_2$, $c_2 \in B'_1$. Then  $A'_2 \subset
    \{b_3, b_4\}$. So, $a_1 \in B'_2$ is complete to $A'_2$. Again,
    $(X'_1, X'_2)$ is degenerate.

    Suppose $c_1 = c_2$.  Up to  symmetry we assume $c_1\in X'_1$. If
    $c_1\in C'_1$ then the only possible vertices in $X'_2$ are $a_1,
    b_3, b_4$, so $\overline{G'}[X'_2]$ induces a triangle. So, any
    vertex of $A'_2$ is complete to $B'_2$ and $(X'_1, X'_2)$ is
    degenerate, a contradiction. So, $c_1 \notin C'_1$. Up to 
    symmetry, we assume $c_1 \in A'_1$. So, $B'_2\subset \{a_1, b_3,
    b_4\}$ and at least one of $a_1, b_3, b_4$ (say $x$) must be
    in $B'_2$. Since $(X'_1, X'_2)$ is not degenerate, $c_1$ must have
    a non-neighbor in $B'_1$. So, one of $a_1, b_3, b_4$ (say $y$)
    must be in $B'_1$. Since $(X'_1, X'_2)$ is not degenerate, $x$
    must have a non-neighbor $z$ in $A'_2$. But $z$ must also be a
    non-neighbor of~$y$. This is impossible because in $G'\setminus
    c_1$, $N(a_1), N(b_3), N(b_4)$ are disjoint.
  \end{proof}

  \begin{lemmaS}
    \label{clbip}
    \label{ccompbas1}
    $G'$ is not basic. None of $G, \overline{G}$ is a path-cobipartite
    graph, a path-double split graph; none of $G, \overline{G}$ has a
    homogeneous 2-join. Moreover, $\overline{G'}$ has no flat path of
    length at least~3.
  \end{lemmaS}

  \begin{proof}
    If $G'$ is bipartite then all the vertices of $A_2$ are of the
    same color because of $a_1$. Because of $b_1$ all the vertices of
    $B_2$ have the same color. By Property~\ref{prop.conn} of $G$,
    there is a path from $A_2$ to $B_2$ whose interior is in $C_2$
    that has parity~$\varepsilon$ since $G$ is Berge. So, the number
    of colors in $A_2 \cup B_2$ is equal to $1+\varepsilon$, which
    implies that $G$ is bipartite and this contradicts the properties
    of $G$. Hence $G'$ is not bipartite.
 
    One of the graphs $G'[c_2, c_1, b_3, b_4]$, $G'[a_1, c_1, b_3,
      b_4]$ is a claw, so $G'$ is not the line-graph of a bipartite
    graph by Theorem~\ref{th.lgbg}.  Let us choose $b\in B_3, b'\in
    B_4$.  The graph $\overline{G'} [a_1, c_1, b, b']$ is a diamond,
    so $\overline{G'}$ is not the line-graph of a bipartite graph by
    Theorem~\ref{th.lgbg}.

    Note that $b, b'$ both have degree at least~3 in $G'$ because
    since $(X_1, X_2)$ is not degenerate, $b, b'$ have neighbors in
    $A_2 \cup C_2$. Also $a_1$ has degree at least~3 in $G'$ by
    Property~\ref{prop.deg3} of~$G$.  So, there exist in $G'$ a stable
    set of size 3 containing vertices of degree at least~3 ($\{a_1, b,
    b'\}$), and a vertex of degree~3 whose neighborhood induces a
    stable set ($c_1$).  Hence, by Lemma~\ref{l.twice}, $G'$ is not a
    path-cobipartite graph (and in particular, it is not the
    complement of a bipartite graph), not a path-double split graph
    (and in particular, it is not a double split graph) and $G'$ has
    no non-degenerate homogeneous 2-join.  By Lemma~\ref{cd2join},
    $G'$ has no degenerate homogeneous 2-join, so it has no
    homogeneous 2-join.
  
    If $\overline{G'}$ has a flat path of length at least~3, then by
    Lemma~\ref{l.thimplies} there is a contradiction with the fact
    that $\overline{G'}$ is not bipartite, or with Lemma~\ref{cesp}
    or~\ref{cgbar1}.
    \end{proof}

  \begin{lemmaS}
    \label{cfgfgp}
     $f(G') + f(\overline{G'}) < f(G) + f(\overline{G})$.  
  \end{lemmaS}


  \begin{proof}
    Every vertex in $\{a_1\} \cup B_3 \cup B_4$ has degree at least~3
    in $G'$. For $a_1$, this is Property~\ref{prop.deg3} of~$G$
    and for vertices in $B_3 \cup B_4$, this is because $(X_1, X_2)$
    is not degenerate. Hence no vertex in $\{a_1\} \cup B_3 \cup B_4$
    can be an interior vertex of a flat path of $G'$, and no vertex in
    $\{c_1, c_2, b_3, b_4, b_1\}$ can be in a maximal flat path of
    $G'$ of length at least~3. Hence, every maximal flat path of $G'$
    of length at least~3 is a maximal flat path of $G$, so
    $f(G') \leq f(G)$. But in fact $f(G') < f(G)$ because $X_1$ is a
    flat path of $G$ that is no more a flat path in $G'$.
    By Lemma~\ref{ccompbas1}, we know $0 = f(\overline{G'}) \leq
    f(\overline{G})$.  We add these two inequalities.
  \end{proof}



  Let us now finish the proof in Case~1.  By
  Lemmas~\ref{cberge}---\ref{clbip},  $G'$ is a counter-example to the
  theorem we are proving now. Hence, Lemma~\ref{cfgfgp} contradicts the
  minimality of $G$.  This completes the proof in Case~1.

\medskip
\subsection{Case 2:
$X_1$ may be chosen in such a way that there are sets $A_3$, $B_3$
satisfying the items~\ref{cond.first}--\ref{cond.penul} of the
definition of cutting 2-joins of type~2.}
\label{mainproofCase2}

The frame of the proof is very much like in Case~1, but the details
differ and are simpler.  We consider the graph $G'$ obtained from $G$
by deleting $X_1 \setminus \{a_1, b_1\}$ (see Fig.~\ref{fig:cut2}).
Moreover, we add new vertices: $c_1, c_2, a_3, b_3$.  Then we add
every possible edge between $a_3$ and $A_3$, between $b_3$ and $B_3$.
We also add edges $a_1 c_1$, $c_1c_2$, $c_2b_1$, $a_3b_3$, $c_1a_3$,
$c_2b_3$.  
  
\begin{lemmaS}
  \label{cberge2}
  $G'$ is Berge.
\end{lemmaS}

\begin{proof}

  \begin{claim}
    \label{clodd2}
    Every path of $G'$ from $B_2$ to $A_2$ with no interior vertex in
    $A_2 \cup B_2$ has odd length.
  \end{claim}\vspace{-2ex}

  \begin{proofclaim}
    If such a path contains one of $a_1, b_1, a_3, b_3,
    c_1, c_2$ then it has length~$3$ or~$5$.  Else such a path may be
    viewed as a path of $G$ from $B_2$ to $A_2$. By
    Lemma~\ref{l.2jAiBi} it has odd length.
  \end{proofclaim}

  \begin{claim}
    \label{clevenAB2}
    Every outgoing path of $G'$ from $A_2$ to $A_2$ (resp.  $B_2$
    to $B_2$) has even length.
  \end{claim}
  
  \begin{proofclaim}
    For suppose there is such a path $P$ from $A_2$ to $A_2$ (the case
    with $B_2$ is similar). If $P$ goes through $a_1$ then it has
    length~2.  If $P$ goes through at least one of $c_1, c_2, a_3,
    b_3, b_1$ then $P$ is the union of two edge-wise-disjoint paths
    from $A_2$ to $B_2$. Thus $P$ has even length
    by~(\ref{clodd2}). Else, $P$ may be viewed as an outgoing path of
    $G$ from $A_2$ to $A_2$, that has even length by
    Lemma~\ref{l.2jAiAi}.  In every case, $P$ has even length.
 \end{proofclaim}

  \begin{claim}
    \label{clevenAB3}
    Every outgoing path of $G'$ from $A_3$ to $A_3$ (resp. $B_3$
    to $B_3$) has even length.
  \end{claim}
  
  \begin{proofclaim}
    For suppose there is such a path $P$ from $A_3$ to $A_3$ (the case
    with $B_3$ is similar). If $P$ goes through $a_1$, $a_3$ or $B_3$
    then it has length~2.  From now on, we assume that $P$ goes
    through none of $a_1, a_3, B_3$. Hence $P$ cannot go through
    $b_3, c_1, c_2$.

    If $P$ goes through $b_1$ then $P$ is the edge-wise-disjoint union
    of two outgoing paths of $G$ from $A_3 \cup \{b_1\}$ to $A_3 \cup
    \{b_1\}$. Thus $P$ has even length by the definition of cutting
    2-joins of type~2. Thus we may assume that $P$ does not go through
    $b_1$.

    Now $P$ may be viewed as an outgoing path of $G$ from $A_3$ to
    $A_3$, that does not go through $b_1$. Thus $P$ is outgoing from
    $A_3 \cup \{b_1\}$ to $A_3 \cup \{b_1\}$, it has even length by
    the definition of cutting 2-joins of type~2.
 \end{proofclaim}


  \begin{claim}
    \label{clantiAB2}
    Every antipath of $G'$ with length at least~2, with its end
    vertices in $V(G')\setminus A_2$ (resp. $V(G')\setminus B_2$), and
    all its interior vertices in $A_2$ (resp. $B_2$) has even length.
  \end{claim}

  \begin{proofclaim}
    Let $Q$ be such an antipath whose interior is in $A_2$ (the case
    with $B_2$ is similar).  We may assume that $Q$ has length at
    least~3. So each end-vertex of $Q$ must have a neighbor in $A_2$
    and a non-neighbor in $A_2$.  So none of $a_1, c_1, c_2, b_1, b_3$
    can be an end-vertex of $Q$. If $a_3$ is an end of $Q$ then the
    other end of $Q$ must be a neighbor of $a_3$, a contradiction.
    Thus $Q$ may be viewed as an antipath of $G$. By
    Lemma~\ref{l.2jAiAi}, $Q$ has even length.
  \end{proofclaim}

  \begin{claim}
    \label{clantiAB3}
    Every antipath of $G'$ with length at least~2, with its end
    vertices in $V(G')\setminus A_3$ (resp. $V(G')\setminus B_3$), and
    all its interior vertices in $A_3$ (resp. $B_3$) has even length.
  \end{claim}

  \begin{proofclaim}
    Let $Q$ be such an antipath whose interior is in $A_3$ (the case
    with $B_3$ is similar).  We may assume that $Q$ has length at
    least~3. So each end-vertex of $Q$ must have a neighbor in $A_3$
    and a non-neighbor in $A_3$.  So none of $a_1, a_3, c_1, c_2, b_1,
    b_3$ can be an end-vertex of $Q$.  Thus $Q$ may be viewed as an
    antipath of $G$. It has even length by the definition of cutting
    2-joins of  type~2.
  \end{proofclaim}

  \begin{claim}
    \label{clantip2}
    Let $Q$ be an antipath of $G'$ of length at least~5. Then $Q$ does
    not go through $c_1, c_2$. Moreover one of $V(Q) \cap \{a_1,
    a_3\}$, $V(Q) \cap \{b_1, b_3\}$ is empty.
  \end{claim}

  \begin{proofclaim}
    Let $Q$ be such an antipath.  In an antipath of length at least~5,
    each vertex is in a triangle of the antipath. So, $c_1, c_2$ are
    not in $Q$ since they are not in any triangle of $G'$.
 
    Suppose $V(Q) \cap \{a_1, a_3\}$, $V(Q) \cap \{b_1, b_3\}$ are
    both non-empty.  In an antipath of length at least~6, for every
    pair $u,v$ of vertices, there is a vertex $x$ seeing both
    $u,v$. Thus $Q$ has length~5 because no vertex of $G'$ has
    neighbors in both $\{a_1, a_3\}$, $\{b_1, b_3\}$. Let $q_1, \dots,
    q_6$ be the vertices of $Q$ in their natural order. Since $V(Q)
    \cap \{a_1, a_3\}$, $V(Q) \cap \{b_1, b_3\}$ are both non-empty
    there are two vertices of $Q$ that have no common neighbors in
    $G'$. These vertices must be $q_2$ and $q_5$, and up to  symmetry
    we must have $q_2= a_3$, $q_5 = b_3$. Thus $q_3$ must be a vertex
    of $B_3$ and $q_4$ must be a vertex of $A_3$. There is a
    contradiction since by the definition of cutting 2-joins of
    type~2, $A_3$ is complete to $B_3$.
  \end{proofclaim}

    Let us now finish the proof. Let $H$ be a hole of $G'$. 

    If $H$ goes through both $c_1, c_2$ then $H$ has length~4 or it
    must contains one of $\{a_1, b_1\}$, $\{a_1, b_3\}$, $\{b_1,
    a_3\}$. In the first case, $H$ is edge-wise partitioned into two
    paths  from $A_2$ to $B_2$. Thus $H$ has even length
    by~(\ref{clodd2}). In the second case $H$ is edge-wise
    partitioned into two paths outgoing from $B_3 \cup \{a_1\}$ to
    $B_3 \cup \{a_1\}$, one of them of length~4, the other one
    included in $V(G)$. Thus $H$ has even length by the definition of
    cutting 2-joins of type~2. The third case is similar. From now on,
    we assume that $H$ goes through none of $c_1, c_2$.  If $H$ goes
    through both $a_1, a_3$ then it has length~4. If $H$ goes through
    $a_2$ and not through $a_3$ then $H$ has even length
    by~(\ref{clevenAB2}). If $H$ goes through $a_3$ and not through
    $a_2$ then $H$ has even length by~(\ref{clevenAB3}). Thus, we may
    assume that $H$ goes through none of $a_1, a_3$. Similarly, we may
    assume that $H$ goes through none of $b_1, b_3$.  Now $H$ may be
    viewed as a hole of $G$. In every case, $H$ has even length.

    Let us now consider an antihole $H$ of $G'$. We may assume that
    $H$ has length at least~7. Let $v$ be a vertex of $V(H) \setminus
    \{a_1, b_1, c_1, c_2, a_3, b_3\}$.  By~(\ref{clantip2}) the
    antipath $V(H) \setminus v$ does not go through $c_1, c_2$ and we
    may assume up to  symmetry that $V(Q) \cap \{b_1, b_3\}$ is
    empty. If $H$ goes through both $a_1, a_3$ then $H$ must contains
    a vertex that sees $a_3$ and misses $a_1$, a contradiction. If $H$
    goes through $a_1$ and not through $a_3$ then $H$ has even length
    by~(\ref{clantiAB2}). If $H$ goes through $a_3$ and not through
    $a_1$ then $H$ has even length by~(\ref{clantiAB3}). If $H$ goes
    through none of $a_1, a_3$ then $H$ may be viewed as an antihole
    of $G$. In every case, $H$ has even length.
  \end{proof}


  \begin{lemmaS}
    \label{cesp2}
     \label{cd2join2}
    $G'$ has no \even\ skew partition. Moreover, $G'$ and
     $\overline{G'}$ have no degenerate substantial 2-join, no
     degenerate homogeneous 2-join and no star cutset.
  \end{lemmaS}

  \begin{proof}
    Suppose that $G'$ has \an\  \even\  skew partition $(E', F')$ with a
    split $(E'_1,$ $E'_2,$ $F'_1,$ $F'_2)$. Starting from $F'$, we shall
    build \an\  \even\  skew cutset $F$ of $G$, which contradicts the
    properties of $G$.
    
    By Property~\ref{prop.conn} of~$G$, $F'$ cannot be a star
    cutset centered on any of $a_1,$ $b_1,$ $c_1,$ $c_2,$ $a_3,$
    $b_3$. For the same reason, $F'$ cannot be a subset of any of
    $\{c_1,$ $c_2,$ $a_3,$ $b_3\}$, $\{a_1,$ $c_1,$ $a_3\}$ $\cup
    A_3$, $\{b_1,$ $c_2,$ $b_3\}$ $\cup B_3$.  Thus, $c_1 \notin F'$
    and $c_2 \notin F'$.  Since $a_1, b_1$ are non-adjacent with no
    common neighbors, they are not together in $F'$ and we may assume $b_1
    \notin F'$.  Up to symmetry we may assume $\{c_1, c_2\} \subset
    E'_1$, so $\{a_1,$ $a_3,$ $c_1,$ $c_2,$ $b_1,$ $b_3\}$ $\cap
    E' \subset E'_1$.  Let $v$ be any vertex of $E'_2$.  Since $\{a_1,
    a_3, c_1, c_2, b_1, b_3\} \cap E' \subset E'_1$, we have $v \in
    X_2$.

    We claim that $F = F' \setminus \{a_3, b_3\}$ is a skew cutset of
    $G$ that separates $v$ from the interior vertices of the path
    induced by $X_1$. Since $F'$ is not a star cutset centered on any
    of $a_3, b_3$, we know that if $a_3 \in F'$ (resp $b_3 \in F'$)
    then $a_3$ (resp. $b_3$) is not the only vertex in its
    anticomponent of $F'$. Hence, $F$ is not anticonnected. If $P$ is
    a path of $G\setminus F$ from $v$ to a vertex $u$ in the interior
    of $X_1$ then up to  symmetry, $P = v \tp \cdots \tp a_1 \tp X_1
    \tp u$. Hence $v \tp P \tp a_1 \tp c_1$ is a path of $G' \setminus
    F'$, which contradicts $F'$ being a cutset of $G'$. We proved our
    claim. Let us prove that the skew cutset $F$ is \even.

    Let $P$ be an outgoing path of $G$ from $F$ to $F$.  We shall
    prove that $P$ has even length.  If $a_1 \notin F$, then $F
    \subset X_2$ and the end-vertices of $P$ are both in $X_2$.  So
    Lemma~\ref{PgivePX2} applies to $P$. Suppose that the first
    outcome of Lemma~\ref{PgivePX2} is satisfied: $V(P) \subseteq X_2
    \cup \{a_1, b_1 \}$.  Hence, $P$ may be viewed as an outgoing path
    from $F'$ to $F'$, so $P$ has even length since $F'$ is
    \an\ \even\ skew cutset of $G'$.  Suppose now that the second
    outcome of Lemma~\ref{PgivePX2} is satisfied: $P = c \tp \cdots
    \tp a_2 \tp a_1 \tp X_1 \tp b_1 \tp b_2 \tp \cdots \tp c'$. Put
    $P' = c \tp P \tp a_2 \tp a_1 \tp c_1 \tp c_2 \tp b_1 \tp b_2 \tp
    P \tp c'$.  The paths $P$ and $P'$ have the same parity and $P'$ is an
    outgoing path of $G'$ from $F'$ to $F'$.  So $P'$ and $P$ have even
    length since $F'$ is \an\ \even\ skew cutset of $G'$.  If $a_1 \in
    F$ then $F \subset X_2 \cup \{a_1\}$ and Lemma~\ref{PgivePA1}
    applies. If Outcome~1 of the lemma holds, then $P$ has even
    length.  If Outcome~2 of the lemma holds then $P$ may be viewed as
    an outgoing path of $G'$ from $F'$ to $F'$. Hence $P$ has even
    length.  If Outcome~3 of the lemma holds then $P = a_1 \tp X_1 \tp
    b_1 \tp b_2 \tp \cdots \tp c$ where $b_2 \in B_2$, $c \in X_2$.
    We put $P' = a_1 \tp c_1 \tp c_2 \tp b_1 \tp b_2 \tp P \tp c$. So
    $P'$ is outgoing from $F'$ to $F'$ in $G'$ while having the same
    parity as $P$. In every case $P$ has even length.

    Now, let $Q$ be an antipath of $G$ of length at least~5 with all
    its interior vertices in $F$ and with its end-vertices outside 
    $F$.  We shall prove that $Q$ has even length.  If $a_1 \notin F$,
    then $F \subset X_2$ and the interior vertices of $Q$ are all in
    $X_2$.  So Lemma~\ref{antiPgivePX2} applies: $V(Q) \subseteq X_2
    \cup \{a\}$ where $a \in \{a_1, b_1\}$.  So $Q$ may be viewed as
    an antipath of $G'$ that has even length because $F'$ is
    \an\ \even\ skew cutset of $G'$.  If $a_1 \in F$, the proof is
    similar. Hence, $Q$ has even length.

    We proved that $G'$ has no \even\ skew partition. If one of $G'$,
    $\overline{G'}$ has a degenerate substantial 2-join, a degenerate
    homogeneous 2-join or a star cutset, then $G'$ has
    \an\ \even\ skew partition by Lemma~\ref{l.degenerate},
    \ref{l.homod} or~\ref{l.starcutset}.  This is a contradiction.
  \end{proof}



  \begin{lemmaS}
    \label{c2jfinal2}
    $G'$ has no non-path proper 2-join. 
  \end{lemmaS}

  \begin{proof}

  \begin{claim}
     \label{cc1c2a3b3}
     If $G'$ has a proper 2-join $(X'_1, X'_2)$ then either $\{c_1,$
     $c_2,$ $a_3,$ $b_3\}$ $\subset X'_1$ or $\{c_1,$ $c_2,$ $a_3,$
     $b_3\}$ $\subset X'_2$.
  \end{claim}

  \begin{proofclaim}
    Suppose not. Up to  symmetry, we have five cases to consider
    according to $X'_1 \cap \{c_1,$ $c_2,$ $a_3,$ $b_3\}$.  Each of them
    leads to a contradiction:

    \vspace{1ex}

    \noindent Case $\{c_1\} \subset X'_1$ and $\{c_2, a_3, b_3\}
    \subset X'_2$:

      Up to  symmetry, we assume $c_1 \in A'_1$ and $c_2, a_3 \in
      A'_2$. Note that $A'_1 = \{c_1\}$ because $c_1$ is the only
      vertex in $X'_1$ that sees both $c_2, a_3$. Note that $a_1$ is
      in $X'_1$ for otherwise $c_1$ would be isolated in $X'_1$. Also if a
      vertex $x$ of $A_3$ is in $X'_1$ then $x$ must be in $A'_1$
      since it sees $a_3$. This is impossible since $x$ misses
      $c_2$. Thus $x\in X'_2$. Since $x$ sees $a_1 \in X'_1$, $x$ must
      be in $B'_2$ and $a_1$ must be in $B'_1$. So, $a_1$ is a vertex
      of $B'_1$ that is complete to $A'_1$, which implies $(X'_1, X'_2)$
      being degenerate, which contradicts Lemma~\ref{cd2join2}.

    \vspace{1ex}

    \noindent Case $\{a_3\} \subset X'_1$ and $\{c_1, c_2, b_3\}
    \subset X'_2$:

      This case is like the previous one, we just sketch it. We assume
      $a_3 \in A'_1$, which implies $c_1, b_3 \in A'_2$. Thus $A'_1 =
      \{a_3\}$. There is a vertex $x$ of $X'_1$ in $A_3$. Also, $a_1
      \in X'_2$ for otherwise $a_1 \in A'_1$ while missing $b_3$, a
      contradiction. Thus $x\in B'_1$, and $x$ is a vertex of $B'_1$
      that is complete to $A'_1$, a contradiction.
  
    \vspace{1ex}

    \noindent Case $\{c_1, c_2\} \subset X'_1$ and $\{a_3, b_3\}
    \subset X'_2$:
      
      Up to  symmetry, we assume $c_1 \in A'_1$, $a_3 \in A'_2$, $c_2
      \in B'_1$, $b_3 \in B'_2$.  Since by Lemma~\ref{cd2join2} $(X'_1,
      X'_2)$ is not degenerate, $a_3$ must have a non-neighbor $x$ in
      $B'_2$. Since $x$ must see $c_2$ we have $x=b_1$ and $b_1 \in
      B'_2$. Similarly, $b_3$ must have a non-neighbor in $A'_2$,
      which implies $a_1 \in A'_2$.  Now put $Y_1 = X_2 \cap X'_1$ and $Y_2
      = X_2 \cap X'_2$.  Note that $Y_1 \neq \emptyset$ for otherwise
      $X'_1 = \{c_1, c_2\}$ and $(X'_1, X'_2)$ is not proper. Also
      $Y_2 \neq \emptyset$ for otherwise, $a_1$ is isolated in
      $X'_2$. If there is an edge of $G'$ with an end in $Y_1$ and an
      end $y$ in $Y_2$, then $y$ must be in one of $A'_2,
      B'_2$. This is a contradiction since $y$ misses both $c_1,
      c_2$. Thus there is no edge with an end in $Y_1$ and an end
      $Y_2$. This contradicts $G[X_2]$ being connected
      (Property~\ref{prop.X2conn} of $G$).
  
    \vspace{1ex}
 
    \noindent Case $\{c_1, a_3\} \subset X'_1$ and $\{c_2, b_3\}
    \subset X'_2$:
      
      Up to  symmetry, we assume $c_1 \in A'_1$, $a_3 \in B'_1$, $c_2
      \in A'_2$, $b_3 \in B'_2$.  Since by Lemma~\ref{cd2join2} $(X'_1,
      X'_2)$ is not degenerate, $a_3$ must have a non-neighbor $x$ in
      $A'_1$. Since $x$ must see $c_2$ we have $x=b_1$ and $b_1 \in
      A'_1$. Similarly, $b_3$ must have a non-neighbor in $A'_2$,
      which implies $a_1 \in A'_2$.  So, $b_1 \in A'_1$, $a_1 \in A'_2$ and
      $a_1 b_1 \notin E(G')$, a contradiction.
        
    \vspace{1ex}

    \noindent Case $\{c_1, b_3\} \subset X'_1$ and $\{c_2, a_3\}
    \subset X'_2$:

      Up to  symmetry, we assume $c_1 \in A'_1$, $a_3 \in A'_2$, $c_2
      \in A'_2$, $b_3 \in A'_1$. There is a vertex $x$ of $X'_1$ in
      $B_3$ for otherwise $b_3$ would be isolated in $X'_1$. Also, $b_1 \in
      X'_2$ for otherwise $c_2$ would be isolated in $X'_2$. But $b$ sees
      $x$. Since $b_1 \in A'_2$ is impossible because $b_1$ misses
      $c_1$ we have $b_1 \in B'_2$. By similar techniques, it can be
      shown that $a_1 \in B'_1$. So, $b_1 \in B'_2$, $a_1 \in B'_1$
      and $a_1 b_1 \notin E(G')$, a contradiction.
  \end{proofclaim}

    Let us now finish the proof.  
    By~(\ref{cc1c2a3b3}), we may assume $\{c_1, c_2, a_3, b_3\}
    \subset X'_2$.  We claim that at most one of $c_1,$ $c_2,$ $a_3,$
    $b_3$ is in $A'_2 \cup B'_2$. For otherwise, up to  symmetry
    there are four cases. First case, $a_3 \in A'_2$, $b_3 \in B'_2$,
    so $A'_1 \subset A_3$ and $B'_1 \subset B_3$, which implies
    $(X'_1, X'_2)$ being degenerate because any vertex of $A'_1$ is
    complete to $B'_1$, contradictory to Lemma~\ref{cd2join2}.  Second case,
    $c_1 \in A'_2$, $c_2 \in B'_2$, which implies $A'_1 = \{a_1\}$, $B'_1 =
    \{b_1\}$, $a_3, b_3 \in C'_2$, $A_3 \cup B_3 \subset X'_2$. Hence,
    $X'_1 \cap X_2 \neq \emptyset$ and $A_3 \cup B_3$ are in different
    components of $G[X_2]$ contradictory to Property~\ref{prop.X2conn} of
    $G$.  Third case, $a_3 \in A'_2$, $c_1 \in B'_2$ so $A'_1
    \subset A_3$, $a_1 \in B'_1$, which implies $(X'_1, X'_2)$ being
    degenerate because $a_1 \in B'_1$ is to complete to $A'_1$,
    contradictory to Lemma~\ref{cd2join2}.  Fourth case, $a_3 \in A'_2$, $c_2
    \in B'_2$, which implies $b_1 \in B'_1$. Also $b_3 \in C'_2$ because
    $b_3, c_2$ (resp. $b_3, a_3$) have no common neighbors in
    $X'_1$. So $B_3 \subset X'_2$ and because of $b_1$, $B_3 \subset
    B'_2$. Because of $a_3$ there is a vertex $a$ of $A'_1$ in
    $A_3$. Hence $a$ is a vertex of $A'_1$ that has a neighbor in
    $B'_2$, a contradiction. All four cases yield a contradiction, so
    our claim is proved.

    Thus up to  symmetry we assume that we are in one of the three
    cases that we describe below:

    \begin{itemize}
    \item  
      $a_3 \in A'_2$. Moreover, $a_1 \in X'_2$ because $c_1 \in C'_2$.
      Because of $a_3$, there is a vertex of $X'_1$ in $A_3$, which implies
      $a_1 \in A'_2$ and $B_3 \subset A'_2$.
    \item  
      $c_1 \in A'_2$. This implies $a_1 \in A'_1$. Since $a_3 \in
      C'_2$, we have $A_3 \subset X'_2$ and $A_3 \subset A'_2$ because
      of $a_1$. Note that $A'_1 = \{a_1\}$ because $a_1$ is the only
      neighbor of $c_1$ in $X'_1$.
    \item  
      $a_2 \notin A'_2$ and $c_1 \notin A'_2$. Moreover, $a_1 \in
      X'_2$ and $A_3 \subset X'_2$.
    \end{itemize}

    In every case, $c_2, b_3 \in C'_2$, so $\{b_1\} \cup B_3
    \subset X'_2$. Note that $X'_1 \subset V(G)$.  Let us now put:
    $X''_1 = X'_1$, $X''_2 = V(G) \setminus X''_1$, $A''_1 = A'_1$,
    $B''_1 = B'_1$, $B''_2 = B'_2$. If $c_1 \in A'_2$ then put $A''_2
    = (A'_2 \cap X_2) \cup (N_G(a_1) \cap X_1)$.  If $c_1 \notin A'_2$
    then put $A''_2 = A'_2\setminus\{a_3\}$.  From the definitions it
    follows that $(X''_1, X''_2)$ is a partition of $V(G)$, that
    $A''_1, B''_1 \subset X''_1$, $A''_2, B''_2 \subset X''_2$, that
    $A''_1$ is complete to $A''_2$, that $B''_1$ is complete to
    $B''_2$ and that there are no other edges between $X''_1$ and
    $X''_2$.  So, $(X''_1, X''_2) = (X'_1, V(G) \setminus X'_1) $ is a
    2-join of $G$.


    Note that $G[X''_1]$ is not a path of length~1 or~2 from $A''_1$
    to $B''_1$ whose interior is in $C''_1$, because $(X'_1, X'_2)$ is
    a proper 2-join of $G'$ and because $X''_1 = X'_1$. Also
    $G[X''_2]$ is not an outgoing path from $A''_2$ to $B''_2$ whose
    interior is in $C''_2$ because $b_1$ has at least~2 neighbors in
    $X''_2$ ($c_2$ and one in $B_3$) while having degree at least~3 by
     Property~\ref{prop.deg3} of $G$. This proves that $(X''_1,
    X''_2)$ is substantial. It is connected for otherwise it would be
    degenerate, contradictory to Lemma~\ref{cd2join2}.  So $(X''_1, X''_2)$
    is proper and we know by the properties of $G$ that $(X''_1,
    X''_2)$ is a path 2-join of $G$.  If $X''_2$ is the path-side of
    $(X''_1, X''_2)$ then $b_1$ is an interior vertex of this path
    while having degree at least~3 by Property~\ref{prop.deg3} of $G$,
    a contradiction. Hence, $X''_1$ is the path-side of $(X''_1,
    X''_2)$. Thus $(X'_1, X'_2)$ is a path 2-join of $G'$ because
    $X''_1 = X'_1$.
  \end{proof}

  \begin{lemmaS}
    \label{cgbar2}
    $\overline{G'}$ has no proper 2-join.
  \end{lemmaS}

  \begin{proof}
    Here the word ``neighbor'' refers to
    the neighborhood in $\overline{G'}$. Let $(X'_1, X'_2)$ be a
    proper 2-join of $\overline{G'}$.
    
    If $c_1 \in C'_1$ then $X'_2 \subset \{a_1, a_3, c_2\}$, so
    $(X'_1, X'_2)$ is degenerate or non-proper,
    a contradiction to Lemma~\ref{cd2join2}. Thus, we may assume $c_1 \in
    A'_1$. Similarly $c_2$ must be in one of $A'_1$, $A'_2$, $B'_1$,
    $B'_2$. But $c_2 \in A'_2$ is impossible because $c_2$ is not a
    neighbor of $c_1$. Also $c_2 \in A'_1$ is impossible because
    otherwise $B'_2= \emptyset$ since no vertex of $\overline{G'}$ can
    be a non-neighbor of both $c_1, c_2$. Thus $c_2$ is in one of
    $B'_1, B'_2$.

    If $c_2 \in B'_1$ then $A'_2 \subset \{b_1, b_3\}$ because of
    $c_2$ and $B'_2 \subset \{a_1, a_3\}$ because of $c_1$. But $b_1$
    must be in $A'_2$ because it is a common neighbor of $c_1, a_1,
    a_3$. Thus $b_1$ is a vertex of $A'_2$ that is complete to $B'_2$,
    so $(X'_1, X'_2)$ is degenerate, a contradiction to
    Lemma~\ref{cd2join2}.

    If $c_2 \in B'_2$ then there is a non-neighbor of $c_2$ in $A'_2$
    for otherwise $(X'_1, X'_2)$ would be degenerate. Thus at least one of
    $b_1, b_3$ is in $A'_2$. Similarly, because of $c_1$, at least one
    of $a_1, a_3$ must be in $B'_1$. But since there is no edge of
    $\overline{G'}$ between $B'_1, A'_2$, we have $a_3 \in B'_1$, $b_3
    \in A'_2$. Since $a_3, b_3, c_2$ are neighbors of $a_1$, we know
    $a_1 \in B'_2$. Now $b_1$ is a neighbor of $c_1 \in A'_1$, $a_3
    \in B'_1$, $a_1 \in B'_2$, $b_3 \in A'_2$, a contradiction.
  \end{proof}


  \begin{lemmaS}
    \label{clbip2}
    \label{ccompbas2}
    $G'$ is not basic. None of $G, \overline{G}$ is a path-cobipartite
    graph, a path-double split graph; none of $G, \overline{G}$ has a
    homogeneous 2-join. Moreover, $\overline{G'}$ has no flat path of
    length at least~3.
  \end{lemmaS}

  \begin{proof}
    If $G'$ is bipartite then all the vertices of $A_2$ are of the
    same color because of $a_1$. Because of $b_1$ all the vertices of
    $B_2$ have the same color. By Property~\ref{prop.conn} of~$G$,
    there is a  path from $A_2$ to $B_2$ that has odd length
    since $G$ is Berge.  Thus $G$ is bipartite, which contradicts the
    properties of $G$. Hence $G'$ is not bipartite.
    
    The graph $G'[c_2, c_1, a_1, a_3]$ is a claw, so $G'$ is not the
    line-graph of a bipartite graph. Since $\overline{G'}[a_1, b_1,
      a_3, b_3]$ is a diamond, $\overline{G'}$ is not the line-graph
    of a bipartite graph.

     Note that $b_1$ has degree at least~3 in $G'$ by
     Property~\ref{prop.deg3} of~$G$. So, there exists in $G'$ a stable
     set of size 3 containing vertices of degree at least~3 ($\{b_1,
     b_3, c_1\}$), and a vertex of degree~3 whose neighborhood induces
     a stable set ($c_1$). Hence, by Lemma~\ref{l.twice}, $G'$ is not
     a path-cobipartite graph (and in particular, it is not the
     complement of a bipartite graph), not a path-double split graph
     (and in particular, it is not a double split graph) and $G'$ has
     no non-degenerate homogeneous 2-join. Hence by Lemma~\ref{cd2join2},
     $G'$ has no homogeneous 2-join.

     If $\overline{G'}$ has a flat path of length at least~3, then by
     Lemma~\ref{l.thimplies} there is a contradiction with the fact
     that $\overline{G'}$ is not bipartite, or with Lemma~\ref{cesp2}
     or~\ref{cgbar2}.
  \end{proof}


  \begin{lemmaS}
    \label{cfgfgp2}
     $f(G') + f(\overline{G'}) < f(G) + f(\overline{G})$.  
  \end{lemmaS}

  \begin{proof}
    Every vertex in $\{a_1, b_1\} \cup A_3 \cup B_3$ has degree at
    least~3 in $G'$. For $a_1$ and $b_1$, this is 
    Property~\ref{prop.deg3} of~$G$ and for vertices in $A_3 \cup
    B_3$, this is clear. Hence no vertex in $\{a_1, b_1\} \cup A_3
    \cup B_3$ can be an interior vertex of a flat path of $G'$, and no
    vertex in $\{c_1, c_2, a_3, b_3\}$ can be in a maximal flat path
    of $G'$ of length at least~3. Hence, every maximal flat path of
    $G'$ of length at least~3 is a maximal flat path of $G$, so $f(G')
    \leq f(G)$. But in fact $f(G') < f(G)$ because $X_1$ is a flat
    path of $G$ that is no more a flat path in $G'$.  By
    Lemma~\ref{ccompbas2}, we know $0 = f(\overline{G'}) \leq
    f(\overline{G})$.  We add these two inequalities.
  \end{proof}


  Let us now finish the proof in Case~2.  By
  Lemmas~\ref{cberge2}---\ref{clbip2},  $G'$ is a counter-example to the
  theorem we are proving now. Hence, Lemma~\ref{cfgfgp2} contradicts the
  minimality of $G$.  This completes the proof in Case~2.

\medskip
\subsection{Case 3: We are neither in Case~1 nor in Case~2.}
\label{mainproofCase3}

In Case~3, the proof is shorter than in the other cases, but is more
complicated in some respects. Indeed, homogeneous 2-joins will  be
found. Again, we shall build a graph $G'$ that is a counter-example.
Note that the following claim is about $G$ itself:

\begin{claim}
  \label{c.noncutting}
  $G$ has no cutting 2-join. 
\end{claim}

\begin{proofclaim}
  Follows directly from the fact that we are not in Case~1, 2.
\end{proofclaim}

We consider the graph $G'$ obtained from $G$ by replacing $X_1$ by a
path of length $2 - \varepsilon$ from $a_1$ to $b_1$. Possibly, this
path has length~2. In this case we denote by $c_1$ its unique interior
vertex. Else, this path has length~1, and for convenience we put $c_1
= a_1$ (thus $c_1$ is a vertex of $G'$ whatever $\varepsilon$). Note
that $(V(G') \setminus X_2, X_2)$ is not a proper 2-join of $G$ since
$V(G') \setminus X_2$ is a path of length~1 or~2 from $a_1$ to $b_1$.
Note that $a_1 \tp c_1 \tp b_1$ a flat path of $G'$ (possibly of
length~1 when $a_1 = c_1$) because if there is a common neighbor $c$
of $a_1, b_1$, then $(X_1, X_2)$ is not a 2-join of~$G$.  Note that
$G'$ is what we call in section~\ref{sec.defpiece} the block $G_2$ of
$G$ with respect to the 2-join $(X_1, X_2)$.

\begin{claim}
  \label{c.skew3}
  $G'$ has no \even\ skew partition, and none of $G$, $\overline{G'}$
  has a star cutset, a degenerate substantial 2-join or a degenerate
  homogeneous 2-join.
\end{claim}

\begin{proofclaim}
   Since $G'$ is a block of $G$, and since $(X_1, X_2)$ is not
   cutting, by Lemma~\ref{l.skew2join}, if $G'$ has \an\  \even\  skew
   partition then so is $G$, which contradicts the properties of $G$. By
   Lemma~\ref{l.starcutset}, \ref{l.degenerate} and~\ref{l.homod}, $G,
   \overline{G}$ have no star cutset, no degenerate 2-join and no
   degenerate homogeneous 2-join.
\end{proofclaim}

\begin{claim}
  \label{c.berge3}
  $G'$ is Berge.
\end{claim}

\begin{proofclaim}
  Clear by Lemma~\ref{l.blockberge}.
\end{proofclaim}

\begin{claim}
  \label{c.np2j3}
  $G'$ has no proper non-path 2-join. 
\end{claim}

\begin{proofclaim}
  Let $(X'_1, X'_2, A'_1, B'_1, A'_2, B'_2)$ be a split of a proper
  non-path 2-join of $G'$.  If $a_1 \in X'_1$, $b_1 \in X'_1$ then
  $c_1 \in X'_1$ since otherwise $c_1$ would be isolated in $X'_2$. If $c_1
  \neq a_1$ then $c_1 \in C'_1$ because $c_1$ has degree~2.  So, by
  subdividing $a_1 \tp c_1 \tp b_1$ we obtain a non-path proper 2-join
  of $G$, which contradicts the properties of $G$. Thus, since $a_1 \tp
  c_1 \tp b_1$ is a flat path of $G'$, up to  symmetry, we may assume
  $c_1 \in B'_1$, $b_1 \in B'_2$.

  Suppose $|B'_2| = 1$. Then no vertex of $A'_2$ has a neighbor in
  $B'_2$ for otherwise, $(X_1, X_2)$ would be degenerate. Thus, $(X'_1 \cup
  B'_2, X'_2 \setminus B'_2)$ is a non-path proper 2-join of $G'$, and
  by subdividing $a_1 \tp c_1 \tp b_1$, we obtain a non-path proper
  2-join of $G$, which contradicts the properties of $G$. Thus, $|B'_2|
  \geq 2$. In particular, $c_1 = a_1$, and similarly $|B'_1| \geq 2$.
 
  In $G$, $a_1$ is complete to $B'_2 \setminus \{b_1\}$, and $b_1$ is
  complete to $B'_1 \setminus \{a_1\}$.  We put $A_3 = B'_2 \setminus
  \{b_1\}$, $B_3 = B'_1 \setminus \{a_1\}$.  In $G$, $X_1$ is a flat
  path from $a_1$ to $b_1$, $A_3 \subset A_2$ and $B_3  \subset B_2$
  and $A_3$ is complete to $B_3$.  We claim that every path of $G$
  outgoing from $A_3 \cup \{b_1\}$ to $A_3 \cup \{b_1\}$ has even
  length. Note that after possibly deleting the interior of $X_1$,
  such a path $P$ may be viewed as a path $P'$ of $G'$ that has same
  parity as $P$.  In $G'$, $P'$ is an outgoing path from $B'_1$ to
  $B'_1$ and by Lemma~\ref{l.2jAiAi}, $P$ has even length as claimed.
  We claim that every  antipath of $G$ whose interior is in
  $A_3 \cup \{b_1\}$ and whose ends are outside  $A_3 \cup \{b_1\}$
  has even length. Let $Q$ be such an antipath of length at least~5.
  Note that the interior vertices of $X_1$ are not in $Q$ since every
  vertex in $Q$ has degree at least~3.  Thus $Q$ is an  antipath
  of $G'$ whose interior is in $B'_1$ and whose ends are not in $B'_1$
  and by Lemma~\ref{l.2jAiAi}, $Q$ has even length as claimed. The
  same properties hold with $B_3 \cup \{a_1\}$. Now, $A_3, B_3$ show
  that $(X_1, X_2)$ satisfies the
  items~\ref{cond.first}--\ref{cond.penul} of the definition of
  cutting 2-joins of type~2, which contradicts that we are not in Case~2
  of the proof of our theorem.
\end{proofclaim}

\begin{claim}
  \label{c.prop2jcomp3}
   $\overline{G'}$ has no proper 2-join. 
\end{claim}

\begin{proofclaim}
  Let us consider a proper 2-join of $\overline{G'}$ with a split
  $(X'_1, X'_2, A'_1, B'_1, A'_2, B'_2)$. If $c_1 \neq a_1$ then $c_1$
  has degree $n-3$ in $\overline{G'}$. Thus, up to  symmetry, we may
  assume $c_1 \in B'_1$. Since $(X'_1, X'_2)$ is not degenerate, $c_1$
  must have a non-neighbor in $A'_1$. Thus, up to  symmetry, we may
  assume $a_1 \in A'_1$, $b_1 \in A'_2$. Now, since $(X'_1, X'_2)$ is
  not degenerate, there exists a vertex of $B'_2$ that is a common
  neighbor of $a_1, b_1$ in $G$, which contradicts $a_1\tp c_1 \tp b_1$
  being a flat path of $G$. We proved $a_1 = c_1$.

  Since $a_1 , b_1$ form a flat edge of $G'$, they must be
  non-adjacent in $\overline{G'}$ with no common non-neighbor. Thus,
  up to  symmetry we have to deal with three cases:
  
    \vspace{1ex}

    \noindent Case  $a_1 \in C'_1$, $b_1 \in X'_2$:

    Since in $G'$ $a_1 b_1$ is flat, in $\overline{G'}$ $a_1$ is
    complete to $A'_1 \cup B'_1$ or up to  symmetry $b_1 \in A'_2$
    while being complete to $B'_2$. Thus, $(X'_1, X'_2)$ is a
    degenerate 2-join, a contradiction.

    \vspace{1ex}

    \noindent Case  $a_1 \in A'_1$, $b_1 \in B'_2$:

    Since in $G'$, $a_1 b_1$ is flat, in $\overline{G'}$, $a_1$ must
    be complete to $(A'_1 \cup C'_1) \setminus \{a_1\}$.

    Suppose first $C'_1 \neq \emptyset$. There is at least one vertex
    of $C'_1$ that has a neighbor in $B'_1$ for otherwise $A'_1 \cup
    A'_2$ is a skew cutset of $\overline{G'}$, which implies $(X'_1,
    X'_2)$ being degenerate. If $a_1$ has a neighbor in $B_1$ then by
    Lemma~\ref{l.2jAiBi} every path from $A'_1$ to $B'_1$ whose
    interior is in $C'_1$ has odd length. Thus, $a_1$ must see every
    vertex of $B'_1$ that has a neighbor in $C'_1$. This implies that
    $A'_1 \cup (N(a_1) \cap B'_1)$ is a star cutset of $G'$, centered
    on $a_1$ and separating $C'_1$ from $X'_2$. Thus, $a_1$ has no
    neighbor in $B_1$. Hence, there is at least one outgoing path of
    even length from $A'_1$ to $B'_1$,  so no vertex in
    $A'_1$ has a neighbor in $B'_1$. If $|A'_1| \geq 2$ then $\{a_1\}
    \cup C'_1 \cup B'_2$ is a star cutset centered on $a_1$ that
    separates $A'_1 \setminus \{a_1\}$ from $B'_2$. Thus, $|A_1| =
    1$. Since, every path from $A'_1$ to $B'_1$ whose interior is in
    $C'_1$ has even length, we know that every path from $A'_2$ to
    $B'_2$ whose interior is in $C'_2$ has even length. Thus, $C'_2
    \neq \emptyset$. By the same proof as above, this implies $B'_2
    = \{b_1\}$. Note that every vertex in $C'_1$ has a neighbor in
    $B'_1$ because a vertex of $C'_1$ with no neighbor in $B'_1$ can
    be separated from the rest of the graph by a star cutset centered
    on $a_1$. Every vertex in $C'_1$ has a non-neighbor in $B'_1$
    because a vertex of $C'_1$ complete to $B'_1$ would imply $(X'_1,
    X'_2)$ being degenerate. Note also that every vertex in $B'_1$ has
    a neighbor in $C'_1$ for otherwise $(X'_1, X'_2)$ would be
    degenerate. Every vertex in $B'_1$ has a non-neighbor in $C'_1$
    because if there is a vertex $b \in B'_1$ complete to $C'_1$ then
    $|B'_1| \geq 2$ implies that $\{b\} \cup C'_1 \cup B'_2$ is a star
    cutset separating $B'_1 \setminus \{b\}$ from $A'_2$, and $|B'_1|
    = 1$ implies that every vertex in $C'_1$ is complete to $A'_1 \cup
    B'_1$, a case already treated.  Let us come back to $G$: in $G$,
    $X_1$ is a path from $a_1$ to $b_1$. Let us denote by $E$ its
    interior. We observe that $(C'_1, B'_1, \{b_1\}, \{a_1\}, E, A'_2
    \cup C'_2)$ is a homogeneous 2-join of $G$ (the last condition of
    the definition of homogeneous 2-joins is satisfied
    by~(\ref{c.noncutting})). This contradicts the properties of $G$.
  
    We proved $C'_1 = \emptyset$. By the same way, $C'_2 =
    \emptyset$. Thus, $(X'_1, X'_2)$ is a non-path proper 2-join of
    $G'$, contradictory to~(\ref{c.np2j3}).

    \vspace{1ex}

    \noindent Case   $a_1 \in A'_1, b_1 \in B'_1$:

    Since $a_1 \tp b_1$ is a flat edge of $G'$, $C'_2 = \emptyset$. If
    $C'_1 = \emptyset$, then just like above $(X'_1, X'_2)$ is a
    non-path proper 2-join of $G'$, contradictory to~(\ref{c.np2j3}). So,
    $C'_1 \neq \emptyset$. Hence, $(A'_2, B'_2, B'_1, A'_1, X_1
    \setminus \{a_1, b_1\}, C'_1)$ is a homogeneous 2-join of $G$
    (the last condition of the definition of homogeneous 2-joins is
    satisfied by~(\ref{c.noncutting})). This contradicts the
    properties of $G$.
\end{proofclaim}

\begin{claim}
  \label{c.linebip3}
   $G'$ is neither a bipartite graph nor the line-graph of a bipartite
  graph.
\end{claim}

\begin{proofclaim}
  Subdividing flat paths of a line-graph of a bipartite graph
  (resp. of a bipartite graph) into a path of the same parity yields a
  line-graph of a bipartite graph (resp. a bipartite graph). Thus, if
  $G'$ is the line-graph of a bipartite graph or a bipartite graph,
  then so is $G$, which contradicts the properties of $G$.
\end{proofclaim}

\begin{claim}
  \label{c.clinegraph3}
   $\overline{G'}$ is not the line-graph of a bipartite graph. 
\end{claim}

\begin{proofclaim}
  Suppose that $\overline{G'}$ is the line-graph of bipartite
  graph. If $c_1 \neq a_1$ then by the properties of $G$ there exists
  a path of even length from $A_2$ to $B_2$ whose interior is in
  $C_2$. Thus, there is a vertex $c \in C_2$. Since $(X_1, X_2)$ is
  not degenerate, $c$ has at least one non-neighbor $b$ in one of
  $A_2, B_2$, say $B_2$ up to symmetry.  Now $\{a_1, c_1, c, b\}$
  induces a diamond of $\overline{G'}$, a contradiction. This proves
  $a_1 = c_1$.

  Let $B$ be a bipartite graph such that $G' = \overline{L(B)}$. Let
  $(X,Y)$ be a bipartition of $B$. So, $a_1, b_1$ may be seen as edges
  of $B$. Let us suppose $a_1 = a_X a_Y$ and $b_1 = b_X b_Y$ where
  $a_X, b_X \in X$ and $a_Y, b_Y \in Y$. Note that these four vertices
  of $B$ are pairwise distinct since in $L(B) = \overline{G'}$, $a_1$
  misses $b_1$. Since $a_1b_1$ is flat in $G'$, every edge of $B$ is
  either adjacent to $a_X$, $a_Y$, $b_X$ or $b_Y$. Thus, the vertices
  of $L(B) = \overline{G'}$ that are different from $a_1, b_1$
  partition into six sets:
  
  \begin{itemize}
  \item
    $A_X$, the set of the edges of $B$ seeing $a_X$ and missing $b_Y$;
  \item
    $A_Y$, the set of the edges of $B$ seeing $a_Y$ and missing $b_X$;
  \item
    $B_X$, the set of the edges of $B$ seeing $b_X$ and missing $a_Y$;
  \item
    $B_Y$, the set of the edges of $B$ seeing $b_Y$ and missing $a_X$;
  \item
    possibly a single vertex $c$ representing the edge $a_Xb_Y$;
  \item
    possibly a single vertex $d$ representing the edge $a_Yb_X$.
  \end{itemize}

  Suppose $|A_X| \geq 2$. Then, $B_X \neq \emptyset$ for otherwise
  one of $\{a_1\}$, $\{a_1, c\}$ would be a star cutset of $\overline{G'}$
  separating $A_X$ from $b_1$. We observe that $(A_X \cup B_X, V(G')
  \setminus (A_X \cup B_X))$ is a 2-join of $\overline{G'}$. This
  2-join is substantial since $|A_X| \geq 2$ and by~(\ref{c.skew3}) it
  is non-degenerate and therefore proper,
  contradictory to~(\ref{c.prop2jcomp3}). Thus, $|A_X| \leq 1$, and
  similarly $|B_X| \leq 1$, $|A_X| \leq 1$, $|B_Y| \leq 1$. Note that
  if $|A_X| = 1$, $|B_X| = 1$ then there is an edge between $A_X, B_X$
  for otherwise one of $\{a_1\}$, $\{a_1, c\}$ would be a star cutset
  separating $A_X$ from $B_X$. Similarly, if $|A_Y| = 1$, $|B_Y| = 1$
  then there is an edge between $A_Y, B_Y$.  In the case when $|A_X| =
  |B_X| = |A_Y| = |B_Y| = 1$ and when $c, d$ are both vertices of
  $G'$, we observe that $\overline{G'}$ is the self-complementary
  graph $L(K_{3,3}\setminus e)$ (represented in Fig.~\ref{fig.3l}). Hence,
  $G'$ is an induced subgraph of the line-graph of a bipartite graph,
  and $G'$ is the line-graph of a bipartite graph,
  contradictory to~(\ref{c.linebip3}).
\end{proofclaim}

\begin{claim}
  \label{c.path-cobip3}
  $G'$ is not a path-cobipartite graph (and in particular, not a
  cobipartite graph).
\end{claim}

\begin{proofclaim}
  If $G'$ is a path-cobipartite graph then let $A$, $B$, $P$, $a$, $b$
  be like in the definition. Suppose first $P = \emptyset$. If $a_1
  \in A, b_1 \in A$, then since $a_1b_1$ is a flat edge of $G'$ we
  have $|A| = 2$. If a vertex $c$ of $B$ sees none of $a_1, b_1$ then
  $B\setminus c$ is a star-cutset of $G'$ separating $c$ from
  $a_1b_1$. Thus $\{a_1\} \cup N(a_1)$ and $\{b_1\} \cup N(b_1)$ are
  two cliques of $G'$ that partition $V(G')$. Thus, we may always
  assume that $a_1 \in A$, $b_1 \in B$. So, $G$ is obtained by
  subdividing $a_1b_1$, so $G$ is a path-cobipartite graph, and this
  contradicts the properties of $G$.
  
  Thus $P \neq \emptyset$. Note that $(P\cup \{a, b\}, A
  \setminus\{a\} \cup B\setminus\{b\})$ is a path 2-join of
  $G'$. Also, $G'[(A \cup B) \setminus\{a, b\}]$ is not a single edge,
  for otherwise $G'$ would be a hole, contradictory to~(\ref{c.linebip3}). Thus
  this 2-join is proper, and so it is not degenerate. In particular,
  every vertex in $A\setminus\{a\}$ has a neighbor and a non-neighbor
  in $B\setminus\{b\}$, which implies $|A| \geq 3$, $|B| \geq 3$.  If at
  least one of $a_1, b_1$ is on $P$ then the graph $G$ obtained by
  subdividing $a_1b_1$ is again a path-cobipartite graph,
  which contradicts the properties of $G$. Thus since $a_1b_1$ is a flat
  edge of $G'$, we may assume $a_1 \in A \setminus \{a\}$, $b_1 \in B
  \setminus \{b\}$. The graph $G$ is obtained by subdividing $a_1b_1$
  into a path $Q$. Now $(P \cup Q \cup \{a, b\}, V(G) \setminus (P
  \cup Q \cup \{a, b\})$ is a 2-join of $G$. By the properties of $G$
  this 2-join must be either a path 2-join or a non-proper 2-join,
  meaning that $V(G') \setminus (P \cup Q \cup \{a, b\})$ is a single
  edge. Now we observe that $G$ is the line-graph of a bipartite graph
  (such graphs are called \emph{prisms}
  in~\cite{chudnovsky.r.s.t:spgt}), which contradicts the properties of
  $G$.
\end{proofclaim}

\begin{claim}
  \label{c.doublesplit3}
  $G'$ is not a path-double split graph.
\end{claim}

\begin{proofclaim}
  Suppose that $G'$ is a path-double split graph. Let $A' = \{a'_1,$
  $\dots, a'_m\}$, $B' = \{b'_1$, $\dots,$ $b'_m\}$, $C' = \{c'_1,
  \dots, c'_n\}$, $D' = \{d'_1, \dots, d'_n\}$ and $E'$ be sets of
  vertices of $G'$ that are like in the definition. If $a_1 \in A'
  \cup E'$ and $b_1 \in B' \cup E'$, then $G$ is obtained from $G'$ by
  subdividing the flat path $a_1 \tp c_1 \tp b_1$. If this yields a
  path of even length between a vertex $a'_i$ and $b'_i$, then this
  path together with a neighbor of $a'_i$ in $C'\cup D'$ and a
  neighbor of $b'_i$ in $C' \cup D'$ that are adjacent, yields an odd
  hole of $G$. Thus every path with an end in $A'$, and end in $B'$
  and interior in $E$ has odd length, and $G$ is a path-double split
  graph, which contradicts the properties of $G$. The case when $a_1 \in B'
  \cup E, b_1 \in A' \cup E$ is symmetric. Since $a_1 \tp c_1 \tp b_1$
  is a flat path of $G'$, there is only one case left up to 
  symmetry: $a_1 = c_1$, $|C'| = |D'| = 2$, $a_1 = c'_1$, $b_1 = c'_2$
  and for every $i \in \{1, \dots, m\}$, $a'_i$ sees $c'_1, d'_2$ and
  $b'_i$ sees $d'_1, c'_2$. So, $G$ is obtained by subdividing $c'_1
  c'_2$ into a path $P$. We see that $(P \cup \{d'_1, d'_2\}, A' \cup
  B' \cup E')$ is a proper non-path 2-join of $G$, which contradicts the
  properties of $G$.
\end{proofclaim}

\begin{claim}
  \label{c.homo3}
  $G'$ has no homogeneous 2-join. 
\end{claim}

\begin{proofclaim}
  Suppose that $G'$ has a homogeneous 2-join $(A, B, C, D, E, F)$.  If
  $c_1 \neq a_1$ then since $c_1$ has degree~2, $c_1$ must be in
  $E$. Thus, by subdividing $a_1 \tp c_1 \tp b_1$ into a path $P$ we
  obtain a graph $G$ with a homogeneous 2-join.  If $c_1 = a_1$ then
  $a_1 b_1$ is a flat edge of $G'$, thus, up to symmetry, either $a_1
  \in C$, $b_1 \in E \cup D$ or $a_1 \in C$, $b_1 \in A$. But the last
  case is impossible since $a_1 b_1$ being flat implies $N(a_1)
  \subset A \cup D \cup E$, which implies $(A, B, C, D, E, F)$ being
  degenerate, which contradicts~(\ref{c.skew3}). Hence, $a_1 \in C$
  and $b_1 \in D \cup E$. So, by subdividing $a_1 b_1$ we obtain a
  graph $G$ that has a homogeneous 2-join. The last condition of the
  definition of homogeneous 2-joins is satisfied
  by~(\ref{c.noncutting}).
\end{proofclaim}

\begin{claim}
  \label{compbas3}
  $\overline{G'}$ is not a path-cobipartite graph, not a path-double
  split graph, has no homogeneous 2-join and no flat path of length at
  least~3.
\end{claim}

\begin{proofclaim}
  Else, by Lemma~\ref{l.thimplies}  either $\overline{G'}$ has a
  proper 2-join, contradictory to~(\ref{c.prop2jcomp3}) or
  $\overline{G'}$ has \an\  \even\  skew partition
  contradictory to~(\ref{c.skew3}), or $\overline{G'}$ is bipartite
  contradictory to~(\ref{c.path-cobip3}), or ${G'}$ is bipartite
  contradictory to~(\ref{c.linebip3}), or $\overline{G'}$ is a double
  split graph and so is $G'$, contradictory to~(\ref{c.doublesplit3}).
\end{proofclaim}

\begin{claim}
  \label{cfgfgp3}
  $f(G') + f(\overline{G'}) < f(G) + f(\overline{G})$.  
\end{claim}

\begin{proofclaim}
  Every flat path of $G'$ is a flat path of $G$ thus $f(G') \leq
  f(G)$. But in fact $f(G') < f(G)$ since $X_1$ is a flat path of $G$
  and not of $G'$.  By~(\ref{compbas3}), $0 = f(\overline{G'}) \leq
  f(\overline{G})$.  We add these two inequalities.
\end{proofclaim}

Let us now finish the proof.  

\begin{itemize}
\item
By~(\ref{c.berge3}), $G'$ is Berge.

\item
By~(\ref{c.linebip3}, \ref{c.clinegraph3}), none of $G',
\overline{G'}$ is the line-graph of a bipartite graph and $G'$ is not
bipartite.

\item
By~(\ref{c.path-cobip3}), $G'$ is not a path-cobipartite
graph. By~(\ref{c.doublesplit3}), $G'$ is not a path-double split
graph.  By~(\ref{c.homo3}), $G'$ has no homogeneous 2-join.
By~(\ref{compbas3}), $\overline{G'}$ is not a path-cobipartite graph,
not a path-double split graph and has no homogeneous 2-join.

\item
By~(\ref{c.np2j3}), $G'$ has no proper non-path
2-join. By~(\ref{c.prop2jcomp3}), $\overline{G'}$ has no proper 2-join.

\item
By~(\ref{c.skew3}), $G'$ has no \even\ skew partition.
\end{itemize}

So, $G'$ is a counter-example to the theorem we are proving now. Hence
there is a contradiction between the minimality of $G$
and~(\ref{cfgfgp3}).  This completes the proof of Theorem~\ref{th.th}.

\section{Proof of Theorem~\ref{th.case}}
\label{s:proofcase}

Let $G$ be a Berge graph. Note that it is impossible that both $G,
\overline{G}$ have a path proper 2-join because in a graph with a
proper path 2-join, no vertex has degree $n-3$, and this should be the
degree of an interior vertex of the path side of a 2-join of
$\overline{G}$.  Let us now apply Theorem~\ref{th.th} to $G$. If one
of $G, \overline{G}$ is basic, has a non-path proper 2-join, or
\an\ \even\ skew partition, we are done. From now on, we assume that
$G$ has no \even\ skew partition and is not basic. So up to a
complementation we have three cases to consider. In each case, we have
to check that $G$ has at least one path proper 2-join, and that the
contraction of any path proper 2-join leaves the graph \even\ skew
partition-free.

If $G$ has a homogeneous 2-join $(A, B, C, D, E, F)$ then it is not
degenerate since $G$ has no \even\ skew cutset. So, every vertex in $A
\cup B \cup C \cup D \cup F$ has degree at least~3. So every flat path
of length at least~3 in $G$ has an end in $C$, an end in $D$ and
interior in $E$. Let $P$ be such a flat path. By definition of
homogeneous 2-joins, such a path is the path side of a non-cutting
2-join that is also proper.  Hence, by Lemma~\ref{l.skew2join}, the
graph obtained by contracting $P$ has no \even\ skew partition.

If $G$ is path-cobipartite then let $A, B, P$ be three sets that
partition $V(G)$ like in the definition. Since $G$ is not basic, $P$
is not empty and is the interior of the unique maximal flat path $P'$
of $G$ with ends $a\in A$ and $b\in B$. Since $A$ and $B$ are cliques,
$(P', V(G)\setminus P')$ is not a cutting 2-join of type~1 of $G$. If
$(P', V(G)\setminus P')$ is cutting of type~2, this means that there
are non-empty sets $A_3 \subset A \setminus \{a\}$ and $B_3 \subset
B\setminus \{b\}$, complete to one another and such that $H = G
\setminus (P' \cup A_3 \cup B_3)$ is disconnected. But since $A, B$
are cliques, this means that $H$ has exactly two components, say $A'
\subset A$, and $B' \subset B$. We observe that $A_3 \cup B_3 \cup
\{a\}$ is a star cutset of $G$, centered on any vertex of $A_3$, that
separates $A'$ from $B' \cup P$. This is a contradiction since $G$ has
no \even\ skew partition. We proved that the unique proper path 2-join
of $G$ is not cutting. Hence, its contraction does not create
\an\ \even\ skew partition.

If $G$ is a path double split graph then let $V(G)$ be partitioned
into sets $A, B, C, D, E$ like in the definition. Since $G$ is not
basic, we know $E \neq \emptyset$. Hence, there is a flat path $P$ in
$G$ that is the path side of a proper 2-join of $G$. The contraction
of any such path $P$ yields a graph $G'$ that is also a path-double
split graph.  By Lemma~\ref{l.dsg}, $G'$ has  no  \even\ 
skew partition.

This proves Theorem~\ref{th.case}. 

\section{Algorithms}
\label{algos}

By Lemma~\ref{espcomp}, the \even\ skew partition is a
self-complementary notion. Thus, for basic graphs, we have to deal
only with bipartite graphs, line-graphs of bipartite graphs and
double-split graphs. When decomposing, we may switch from the graph to
its complement as often as needed.

\subsection{\Even\ skew partitions in basic graphs}

\begin{lemma}
  Let $G$ be a bipartite graph. Then $(A,B)$ a skew partition of $G$
  if and only if it is \an\ \even\ skew partition of $G$.
\end{lemma}

\begin{proof}
  \An\ \even\ skew partition of $G$ is clearly a skew partition. Let
  us prove the converse. Since $G$ is bipartite, $B$ is a complete
  bipartite graph. Every path of length at least~2 with its ends in
  $B$ and its interior in $A$ has even length, because its ends are in
  the same side of the bipartition.  Since $G$ is triangle-free, every
  antipath of $G$ has length at most~3. Hence, every antipath of
  length at least~2, with its ends in $A$ and its interior in $B$ has
  even length. Because otherwise such an antipath has length~3 and may
  be viewed as a path with its ends in $B$ and interior in $A$.
\end{proof}

By the lemma above, detecting \even\  skew partitions in bipartite graphs
can be performed by running an algorithm for general skew
partitions. Such a fast algorithm for bipartite graphs has been given
by Reed~\cite{reed:skewhist}. It  has complexity $O(n^5)$.

Now, we have to decide if the line-graph of a bipartite graph has \an\ 
\even\  skew partition.  Note that every case is possible: line-graphs
of bipartite graphs may have \even\  skew partitions, skew partitions and
no \even\  skew partition, or no skew partition at all, see
Fig.~\ref{fig.3l}. By Theorem~\ref{th.lgbg} the line-graph of a
bipartite graph has no claw and no diamond.

\begin{figure}[h]
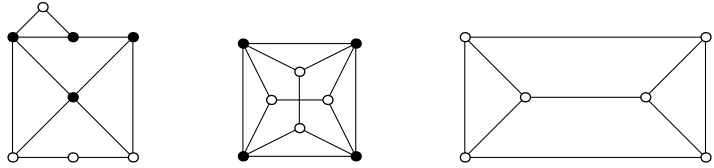

  \center \includegraphics{evenskew.9}\rule{3em}{0em}
  \includegraphics{evenskew.5}\rule{3em}{0em}
  \includegraphics{evenskew.1}
  \caption{Three line-graphs of bipartite graphs. The second one is
    $L(K_{3,3}\setminus e)$\label{fig.3l}}
\end{figure}

The following is implicitly stated and proved in~\cite{reed:skewhist}
in the more general context of line-graphs (of possibly non-bipartite
graphs). We state it and prove it for the sake of completeness.

\begin{lemma}[Reed~\cite{reed:skewhist}]
  \label{l.butorsq}
  Let $G$ be the line-graph of a bipartite graph with a skew partition
  $(A,B)$. Then $B$ is a star or $B$ is a square.
\end{lemma}

\begin{proof}
  Suppose that $G$ has a skew partition $(A,B)$ such that $B$ has at
  least~5 vertices.  We may assume that $B$ is not a star, so every
  anticomponent of $B$ has at least~2 vertices. Let $B_0$ be such an
  anticomponent, and let $b, b'$ be non-adjacent in $B_0$ (because
  $B_0$ is anticonnected). If $B$ has at least~3 anticomponents say
  $B_0, B_1, B_2, \dots$, then for $b_1 \in B_1$, $b_2 \in B_2$, $\{b,
  b', b_1, b_2\}$ induces a diamond, a contradiction. Thus, $B$ has 2
  anticomponent $B_0, B_1$ and we may assume that $B_0$ has at least~3
  vertices.  If $B_0$ has no edge, then we can pick 3 vertices $b_1,
  b_2, b_3$ in $B_0$ and a vertex $c$ in $B_1$ and $\{c, b_1, b_2,
  b_3\}$ induces a claw, a contradiction. Thus, $B_0$ has at least one
  edge, say $bb'$.  Now consider a non-edge $c,c'$ in $B_1$: $\{b, b',
  c, c'\}$ induces a diamond, a contradiction.  So, we are left with
  the case where $B$ has at most~4 vertices. The only candidate for a
  non-star non-anticonnected graph is the square.
\end{proof}

\begin{lemma}
  \label{butterequiv}
  Let $G$ be the line-graph of a bipartite graph. Suppose that $G$ has
  at least one edge and size at least~5.  Then $G$ has \an\  \even\  skew
  partition if and only if $G$ has a star cutset.
\end{lemma}

\begin{proof}
  By Lemma~\ref{l.starcutset}, we know that if $G$ has a star cutset,
  then it has \an\ \even\ skew partition. Let us prove the
  converse. Suppose that $G$ has \an\ \even\ skew partition
  $(A,B)$. We may assume that $B$ is not a star. So by
  lemma~\ref{l.butorsq}, $B$ is a square with vertices say $b_1, b_2,
  b_3, b_4$ and edges $b_1 b_2$, $b_2 b_3$, $b_3 b_4$, $b_4 b_1$. Note
  that in Fig.~\ref{fig.3l}, the first graph represented has a square
  cutset that is \an\ \even\ skew cutset. Let $X$ be a connected
  component of $G \setminus B$.  To finish the proof, it suffices to
  show that one of the stars $\{b_1, b_2, b_4\}$, $\{b_2, b_3, b_4\}$
  or $\{b_1, b_2, b_3\}$ is a cutset. So, let us suppose for a
  contradiction that none of these sets is a cutset.

  Since $\{b_2, b_3, b_4\}$ is not a cutset, $b_1$ has a neighbor in
  $X$ and similarly since $\{b_1, b_2, b_4\}$ is not a cutset, $b_3$
  has a neighbor in $X$. Since $X$ is connected, we know that there is
  a path from $b_1$ to $b_3$ that goes through none of $b_2, b_4$.  We
  may choose this path as short as possible, so it is an induced path,
  say $P = \bp v_1 \tp v_2 \tp \cdots \tp v_{k-1} \tp v_k \ep$, with
  $v_1 = b_1$ and $v_k = b_3$.  Since $(A, B)$ is \even, $P$ has even
  length. One of $b_2, b_4$ (say $b_2$ up to symmetry) must see $v_2$
  for otherwise $\{b_1, v_2, b_2, b_4\}$ induces a claw with center
  $b_1$. If $P$ has length $2$, then, $\{b_1, b_2, b_3, v_2\}$ induces
  a diamond, a contradiction. So, $P$ has length at least~$4$.  But
  then, $\bp v_2 \tp P \tp v_k \tp b_2 \tp v_2 \ep$ is a cycle of odd
  length $\geq 5$, thus it has a chord $b_2 v_i$. But $i$ must equal
  $k-1$ for otherwise, $b_2, b_1, b_3, v_i$ induce a claw. So $H = \bp
  b_2 \tp v_2 \tp P \tp v_{k-1} \tp b_2 \ep$ is a hole. We rename its
  vertices $h_1, \dots, h_l$.

  Since $\{b_1, b_2, b_3\}$ is not a cutset, there is a path $Q$ that
  goes through none of $b_1, b_2, b_3$, from $b_4$ to a vertex that
  has a neighbor in $H$. Let us choose $Q = b_4 \tp \cdots \tp x'
  \tp x$ of minimal length. Note that $Q$ has length at least~1, for
  otherwise, $b_4$ has a neighbor $v_i\in H$. If $2<i<k-1$ then
  $\{b_4, b_1, b_3, v_i\}$ induces a claw and if $i=2$ then $\{b_1,
  b_2, b_4, v_i\}$ induces a diamond ($i=k-1$ is symmetric).  If $x$
  sees two non-adjacent vertices $y, z$ in $H$, then $\{x, x', y, z\}$
  induces a claw. If $x$ sees only one vertex $h_i$ in $H$ then
  $\{h_i, h_{i-1}, h_{i+1}, x\}$ induces a claw. So, $x$ has exactly
  two adjacent neighbors in $H$, say $h_i, h_{i+1}$. Since $H$ is an
  even hole, the induced paths $b_4 \tp Q \tp x \tp h_i \tp H\setminus
  h_{i+1} \tp b_2$ and $b_4 \tp Q \tp x \tp h_{i+1} \tp H \setminus
  h_i \tp b_2$ have different parity.  So one of them has odd length,
  which contradicts $(A, B)$ being \even.
\end{proof}

By the previous lemma we know that an algorithm that detects star
cutsets is sufficient to decide whether a line-graph of a bipartite
graph has or not \an\ \even\ skew
partition. Chv\'atal~\cite{chvatal:starcutset} gave such an
$O(nm)$-time algorithm.  Note that in~\cite{reed:skewhist}, Reed gives
a fairly optimised algorithm for detecting general skew partitions in
line graphs with complexity $O(n^2m)$.  So, the obvious algorithm for
detecting \an\ \even\ skew partition in the line-graph of a bipartite
graph is faster than the optimised algorithm for general skew
partition. This might be general: detecting a skew partition might be
harder than \an\ \even\ skew partition for perfect graphs.

The detection of \even\  skew partitions in double split graphs takes
constant time by Lemma~\ref{l.dsg}: answer ``No''.

Our main algorithm needs also to recognize basic graphs. This can be
done in linear time for bipartite graphs (this is a classical result)
and for line-graphs of bipartite graphs
(see~\cite{lehot:root,roussopoulos:linegraphe}). For double split
graphs, this can be done in linear time by looking at the degrees
since vertices of the matching all have degree $1 + n$ and vertices of
the anti-matching all have degree $2n-2 + m$ (these numbers are
different since $n \geq 2, m \geq 2$ implies $2n -2 + m > 1 + n$).
Hence, the recognition can be performed as follows: compute the
degrees, check whether the vertices of smallest degree induce a
matching, that the rest of the graph induces the complement of
a matching, and check for every edge $xy$ of the matching and every
non-edge $\overline{uv}$ of the antimatching, that $\{x,y,u,v\}$
induces a path on 4 vertices. The computing of degrees takes linear
time, and the checking to be done afterward does not take more than
$O(m)$ time.

Let us sum up this subsection.

\begin{theorem}[Several authors]
  \label{th.basicalgo}
  There is an $O(n+m)$ algorithm that decides whether a given graph is
  basic. There is an $O(n^5)$ algorithm that given a basic graph $G$
  decides whether $G$ has \an\  \even\  skew partition or not.
\end{theorem}

\subsection{2-join decomposition}

Let us define a decomposition tree $T_G$ of a Berge graph $G$: 

\begin{itemize}
\item
  The root of $T_G$ is $G$ itself.
\item
  If a node $F$ of the tree is a basic graph then it is a leaf marked
  with label ``basic''.
\item
  Else, if $F$ is a graph on at most 10 vertices, then it is a leaf
  marked with label ``small''. 
\item
  Else, if none of $F, \overline{F}$ has a substantial 2-join then $F$
  is a leaf marked with label ``no decomposition''.  
\item
  Else, one of $F, \overline{F}$ has a substantial 2-join and has at
  least 11~vertices. If possible, we choose this substantial 2-join
  $(X_1, X_2)$ non-path. If $(X_1, X_2)$ is degenerate then $F$ is
  a leaf marked with label ``degenerate''.
\item
  Else, note that $(X_1, X_2)$ is connected and proper. Up to a
  complementation, we suppose that the 2-join is in $F$. Up to 
  symmetry we suppose $|X_2| \leq |X_1|$.  If $|X_2| = 4$ and if there
  exist vertices $a \in A_1$, $b\in B_1$ such that $\{a, b\}$ is a
  component of $G[X_1]$ then we replace $X_2$ by $X_2\cup \{a,b\}$ and
  $X_1$ by $X_1 \setminus \{a, b\}$. We obtain again a proper 2-join
  of $G$ because $|X_2| = 4$ implies $|X_1| \geq 7$.  Finally we
  define the children of $F$ to be the blocks of $F$ with respect to
  $(X_1, X_2)$ (these blocks are defined
  in Subsection~\ref{sec.defpiece}).
\end{itemize}

Note that every node of $G$ is Berge by Lemma~\ref{l.blockberge}.  We
claim that $T_G$ has size at most $O(n)$. To prove this we need to
study these 2-joins of $G$ with one side of size 4 or 5, because such
2-joins are likely to yield blocks of the same size as the graph,
possibly leading us to a decomposition tree with infinitely many
nodes. A \emph{check} in a graph $F$ is a set of 4 vertices that can
be named $a, b, c, d$ in such a way that:

\begin{itemize}
\item
  either $E(F) \cap {\{a, b, c, d\} \choose 2} = \{ab, bd, dc, ca\}$
  or $E(F) \cap {\{a, b, c, d\} \choose 2} = \{ad, bc\}$;
\item
  $N(a) \setminus \{a, b, c, d\} = N(b) \setminus \{a, b, c, d\}$;
\item
  $N(c) \setminus \{a, b, c, d\} = N(d) \setminus \{a, b, c, d\}$;

\item
  there is a vertex in $F$ that sees both $a,b$ and misses both $c, d$;
\item
  there is a vertex in $F$ that sees both $c,d$ and misses both $a, b$.
\end{itemize}

For any graph $F$, we denote by $c(F)$ the maximum size of a set of
checks of $F$ that are pairwise disjoint. Note that $c(F) \leq
\left\lfloor |V(F)| / 4 \right\rfloor$.  We put: 

\[\psi(F) = 2|V(F)| - 20 + c(F)\]

\[\phi(F) = \max (\psi(F) , 1)\]

Note that a check of $F$ is a check of $\overline{F}$ so that $c(F) =
c(\overline{F})$, $\psi(F) = \psi(\overline{F})$ and $\phi(F) =
\phi(\overline{F})$.

\begin{lemma}
  \label{l.counting}
  Let $H$ be a non-leaf node of $T_G$ and let $H_1, H_2$ be its two
  children. Then $\phi(H) \geq \phi(H_1) + \phi(H_2)$.
\end{lemma}

\begin{proof}
  Since $H$ is not a leaf of $T_G$, $H_1$ and $H_2$ are the blocks of
  $H$ with respect to a 2-join $(X_1, X_2)$ of $G$. Let us denote by
  $P$ the flat path of $H_2$ that represents $X_1$. By definition of
  the blocks, $P$ has length 1, 2, 3 or 4. We denote by $p_1, \dots,
  p_k$ $(k \in \{2, 3, 4, 5\})$ the vertices of $P$. Up to  symmetry
  we suppose $p_1$ complete to $A_2$ and $p_k$ complete to $B_2$. The
  following six claims are concerned with $X_2$ and $H_2$, but similar
  claims can be proved with $X_1$ and $H_1$. Here, a 2-join is said to
  be \emph{even} or  \emph{odd} according to the parity of the paths
  with an end in $A_1$, an end in $B_1$ and interior in $C_1$. This is
  well defined by Lemma~\ref{l.2jAiBi}.

  \begin{claim}
    \label{c.oddcheck}
    Suppose that $(X_1, X_2)$ is an odd 2-join, either non-path or
    such that $X_2$ is not its path-side. If $|X_2| \leq 4$ then $X_2$
    is a check of $H$.
  \end{claim}

  \begin{proofclaim}
    If $|X_2| = 3$ then $(X_1, X_2)$ is degenerate, which contradicts $H$
    being a non-leaf node of $T_G$. So $|X_2| = 4$. Every path from
    $A_2$ to $B_2$ whose interior is in $C_2$ has length~1 because
    $(X_1, X_2)$ is non-path while $|X_2| = 4$. Let $a$ be in
    $A_2$. Since $(X_1, X_2)$ is not degenerate, $a$ must have a
    neighbor $c$ in $B_2$ and a non-neighbor $d$ in $B_2$. Similarly
    $c$ must have a non-neighbor $b$ in $A_2$. Now $X_2 = \{a, b, c,
    d\}$. Vertex $d$ must have a neighbor in $A_2$ and the only
    candidate is $b$. Either $ab, cd$ are both edges of $H$ or both
    non-edge of $H$ for otherwise $H$ contains a $C_5$. In either
    cases, $\{a, b, c, d\}$ is a check of $H$ because $(X_1, X_2)$ is
    a 2-join of $H$.
  \end{proofclaim}

  \begin{claim}
    \label{c.evencheck}
    Suppose that $(X_1, X_2)$ is an even 2-join, either non-path or
    such that $X_2$ is not its path-side. If $|X_2| \leq 5$ then
    $|X_2| = 5$ and $X_2$ contains a check of $H$.
  \end{claim}

  \begin{proofclaim}
    Each path from $A_2$ to $B_2$ whose interior is in $C_2$ has
    length~2 because $(X_1, X_2)$ is non-path while $|X_2| \leq
    5$. Since $(X_1, X_2)$ is even there is at least one vertex $a \in
    C_2$. Since $(X_1, X_2)$ is not degenerate, $a$ is not complete to
    $A_2 \cup B_2$. Hence there is another vertex $b\in C_2$, and $a$
    has a non-neighbor $d$ in one of $A_2, B_2$, say $A_2$ up to
    symmetry. Vertex $a$ must have a neighbor $c \in A_2$ for
    otherwise $B_1 \cup B_2$ is a skew cutset of $H$, which
    contradicts $(X_1, X_2)$ being non-degenerate. Since $|X_2| \leq
    5$, we have $A_2 = \{c, d\}$, $C_2 = \{a, b\}$ and there is a
    single vertex $e$ in $B_2$. Vertex $d$ must have a neighbor in
    $C_2$ because $(X_1, X_2)$ is not degenerate and $b$ is the only
    candidate. Vertex $e$ sees at least one of $a, b$ say $a$ up to
    symmetry. If $e$ misses $b$ then $b$ must see $a$ for otherwise
    $A_1 \cup B_1$ is a skew cutset separating $b$ from the rest of
    the graph. But then $e \tp a \tp b \tp d$ is a path of odd length
    from $B_2$ to $A_2$, which contradicts $(X_1, X_2)$ being even. We
    proved that $e$ sees both $a, b$.  Also, $b$ must have a
    non-neighbor in $A_2 \cup B_2$ and $c$ is the only
    candidate. Either $ab, cd$ are both edges of $H$ or both non-edges
    of $H$ for otherwise $H$ contains a $C_5$.  In either cases, $\{a,
    b, c, d\}$ is a check of $H$ because $(X_1, X_2)$ is a 2-join of~$H$.
  \end{proofclaim}

  \begin{claim}
    \label{c.notinter}
    Every check of $H_2$ that does not intersect $\{p_1, \dots, p_k\}$
    is a check of $H$.
  \end{claim}

  \begin{proofclaim}
    Clear by the definition of the checks.
  \end{proofclaim}

  \begin{claim}
    \label{c.5inter}
    If $k=5$ then none of $p_1, \dots, p_k$ is in a check of $H_2$.
  \end{claim}

  \begin{proofclaim}
    Let $C$ be a check of $H$ that does intersect $\{p_1, \dots,
    p_5\}$. Let $i \in \{1, \dots, 5\}$ be an integer closest to $3$
    such that $p_i\in C$. Up to symmetry, $i\in \{1, 2, 3\}$.
    
    If $i=3$ then $p_3$ must have a neighbor in $C$ and a neighbor out
    of $C$ (by definition of checks). Hence we may assume $p_2 \in C,
    p_4\notin C$. There is a vertex $a\in C$ satisfying $N(a)
    \setminus C = N(p_2) \setminus C$. Hence, $p_1\notin C$ and $a\in
    A_2$. For the same reason, $p_5\in C$. There is a contradiction
    since $ap_5 \notin E(H_2)$.

    If $i=2$ then $p_3 \notin C$. So, $p_1\in C$. By definition of
    checks, there exists $a\in C$ such that $N(a) \setminus C = N(p_2)
    \setminus C$. So, $a = p_4$. Hence, $p_5 \in C$ and $C=\{p_1, p_2,
    p_4, p_5\}$, a contradiction since no vertex in $H_2$ sees both
    $p_1, p_5$.

    If $i=1$ then $p_2 \notin C$. By definition of checks, there
    exists a vertex $a$ in $C$ such that $N(a) \setminus C = N(p_1)
    \setminus C$. This is impossible because of $p_2$.
  \end{proofclaim}

  \begin{claim}
    \label{c.4inter}
    If $k=4$ and if $C$ is a check of $H_2$ that does intersect
    $\{p_1, \dots, p_k\}$ then $C \cap \{p_1, \dots, p_4\} = \{p_2,
    p_3\}$ and $|X_1| > 4$.
  \end{claim}

  \begin{proofclaim}
    Suppose $p_1 \in C$. If $p_2\notin C$ then by definition of checks
    $p_3 \in C, p_4 \in C$ and there is a vertex $a\in A_2$ in
    $C$. Since $N(a) \setminus C = N(p_4) \setminus C$, vertex $a\in
    A_2$ must be complete to $B_2$, which contradicts $(X_1, X_2)$ being
    non-degenerate. Hence, $p_1 \in C$, $p_2\notin C$ is impossible, and
    symmetrically, $p_4 \in C$, $p_3\notin C$ is impossible. If $p_2 \in
    C$ then $p_4\in C$, $p_3\notin C$, a contradiction by the preceding
    sentence.

    We proved $p_1\notin C$. Similarly, $p_4 \notin C$. Hence, up to
    symmetry, $p_2 \in C$ and $p_3 \in C$ by definition of
    checks. Hence, $C = \{p_2, p_3, a, b\}$ where $a\in A_2$, $b \in
    B_2$. So, $N(a) \setminus C = N(p_2) \setminus C$ and $N(b)
    \setminus C = N(p_3) \setminus C$ and $\{p_2, p_3\}$ is a
    component of $H[X_2]$. By the way $T_G$ is constructed, this
    implies $|X_1| > 4$.
  \end{proofclaim}

  \begin{claim}
    \label{c.3inter}
    If $k=3$ and if $C$ is a check of $H_2$ that does intersect
    $\{p_1, \dots, p_k\}$ then $|\{p_1, p_2, p_3 \} \cap C| \geq 2$.
  \end{claim}

  \begin{proofclaim}
    If $p_2\in C$ then $p_2$ must have a neighbor in $C$, so at least
    one of $p_1, p_3$ is in $C$. Hence we may assume up to symmetry $C
    \cap \{p_1, p_2, p_3\} = \{p_1\}$. In $C$ there is a vertex $a$
    such that $N(a) \setminus C = N(p_1) \setminus C$. This is a
    contradiction because of $p_2$.
  \end{proofclaim}

  To finish the proof of the lemma we will study eight cases described
  Table~\ref{table.cases}. In the description of the cases, when we
  write ``$X_i$ non-path'', we mean either $(X_1, X_2)$ is not a path
  2-join or $X_i$ is not the path-side of $(X_1, X_2)$. Note that up
  to symmetry the eight cases cover all the possibilities because
  $X_1$ and $X_2$ cannot both be path-side of $(X_1, X_2)$ since $H$
  is not bipartite, and because at least one of $X_1, X_2$ has size at
  least~6 since $|V(H)| \geq 11$.  

  \begin{table}[h]
    \center
    \begin{tabular}{|l|c|c|}\hline
      $(X_1, X_2)$ is an even 2-join & $X_2$ non-path, $|X_2| \geq 6$
      & $X_2$ non-path, $|X_2| \leq 5$\\\hline $X_1$ non-path,
      $|X_1|\geq 6$ & Case 1 & Case 2\\\hline $X_1$ path, $|X_1| \geq
      5$ & Case 3 & Case 4\\\hline
    \end{tabular}

    \vspace{5ex}

    \begin{tabular}{|l|c|c|}\hline
      $(X_1, X_2)$ is an odd 2-join& $X_2$ non-path, $|X_2| \geq 5$ & $X_2$
      non-path,   $|X_2| \leq  4$\\\hline
      $X_1$ non-path,  $|X_1|\geq 5$ & Case 5 & Case 6\\\hline
      $X_1$ path,  $|X_1| \geq 4$ & Case 7 & Case 8\\\hline
    \end{tabular}
    \caption{The eight cases\label{table.cases}}
  \end{table}

  \noindent{\bf Case 1: }By definition of the blocks:
  $|V(H_1)| =  |X_1| +5$ and $|V(H_2)| = |X_2| +5$. 

  \noindent By~(\ref{c.notinter}) and~(\ref{c.5inter}):

  $c(H) \geq c(H_1) + c(H_2)$. 

  \noindent Since $|X_1|, |X_2| \geq 6$ we have
  $|V(H_1)|, |V(H_2)| \geq 11$, which implies:

  $\psi(H_1), \psi(H_2) \geq 2$
  and $\psi(H_1) = \phi(H_1)$, $\psi(H_2) = \phi(H_2)$. 

  \noindent So:
  
  \begin{tabular}{lcl}
    $\phi(H)$ & $=$    & $2|X_1| + 2|X_2| -20 + c(H)$\\
              & $\geq$ & $2|V(H_1)| + 2|V(H_2)| + c(H_1) + c(H_2) -40$\\
              & $\geq$ & $\phi(H_1) + \phi(H_2)$
  \end{tabular}

  \vspace{2ex}

  In Cases 2--8 we will prove the following three inequalities:
  $\phi(H) > \psi(H_1)$, $\phi(H) > \psi(H_2)$, $\phi(H) \geq
  \psi(H_1) + \psi(H_2)$. Since $H$ is not a leaf of $T_G$ we have
  $|H| \geq 11$ so that $\phi(H) = \psi(H) \geq 2$. So the three
  inequalities mentioned above will imply $\phi(H) \geq \phi(H_1) +
  \phi(H_2)$.

  \vspace{2ex}

  \noindent{\bf Case 2: } By definition of the blocks:
  $|V(H_1)| = |X_1| +5$ and $|V(H_2)| \leq 10$

  \noindent By~(\ref{c.notinter}), (\ref{c.evencheck})
  and~(\ref{c.5inter})

  $|X_2| = 5$.  

  $c(H) \geq c(H_1) +1$.

  \noindent So:

  \begin{tabular}{lcl}
    $\phi(H)$ & $=$    & $2|X_1| + 2|X_2| - 20 + c(H)$\\ 
              & $\geq$ & $2|V(H_1)| - 20 + c(H_1) + 2|X_2| - 9$\\
              & $>$    & $\psi(H_1)$
  \end{tabular}

  The inequalities $\phi(H) > \psi(H_2)$ and $\phi(H) \geq \psi(H_1) +
  \psi(H_2)$ are easy since $|V(H_2)| = 10$ and~(\ref{c.notinter})
  implies $\psi(H_2) \leq 1$.

  \vspace{2ex}

  \noindent{\bf Case 3: } By definition of the blocks:
  $|V(H_1)| = |X_1| +5$ and $|V(H_2)| = |X_2| +3$. 

  \noindent By~(\ref{c.notinter}), (\ref{c.3inter}), (\ref{c.5inter})
  and since $H_1$ is a hole:

  $c(H_1) = 0$ 

  $c(H) \geq c(H_2) -1$
  
  \noindent So:

  \begin{tabular}{lcl}
    $\phi(H)$ & $=$    & $2|X_1| + 2|X_2| - 20 + c(H)$ \\ 
              & $\geq$ & $2|V(H_2)| - 20 + c(H_2) + 2|X_1| - 7$ \\
              & $>$    & $\psi(H_2)$
  \end{tabular}

  \begin{tabular}{lcl}
    $\phi(H)$ & $=$    & $2|X_1| + 2|X_2| - 20 + c(H)$\\
              & $\geq$ & $2|V(H_1)| - 20 + c(H_1) + 2|X_2| - 10$ \\
              & $>$    & $\psi(H_1)$
  \end{tabular}

  \noindent Finally, $\phi(H) \geq \psi(H_1) + \psi(H_2)$ holds similarly.

  \vspace{2ex}

  \noindent{\bf Case 4: } By definition of the blocks:
  $|V(H_1)| = |X_1| +5$ and $|V(H_2)| \geq 8$. 

  \noindent By~(\ref{c.notinter}), (\ref{c.evencheck}),
  (\ref{c.3inter}) and since $H_1$ is a hole:

  $c(H_1) = 0$ 

  $|X_2| = 5$

  $c(H) \geq c(H_2) \geq 1$
  
  \noindent So:

  \begin{tabular}{lcl}
    $\phi(H)$ & $=$    & $2|X_1| + 2|X_2| - 20 + c(H)$\\
              & $\geq$ & $2|V(H_1)| - 20 + c(H_1) + 2|X_2| - 9$ \\
              & $>$    & $\psi(H_1)$
  \end{tabular}

  The inequalities $\phi(H) > \psi(H_2)$ and $\phi(H) \geq \psi(H_1) +
  \psi(H_2)$ are easy since $V(H_2) = 8$ implies $\psi(H_2) \leq 0$.

  \vspace{2ex}

  \noindent{\bf Case 5: } By definition of the blocks:
  $|V(H_1)| = |X_1| + 4$ and $|V(H_2)| = |X_2| + 4$. 

  \noindent By~(\ref{c.notinter}) and(\ref{c.4inter}):

  $c(H) \geq c(H_1) + c(H_2) - 2$

  $c(H) \geq c(H_1) - 1$

  $c(H) \geq c(H_2) - 1$
  
  \noindent So:

  \begin{tabular}{lcl}
    $\phi(H)$ & $=$    & $2|X_1| + 2|X_2| - 20 + c(H)$\\
              & $\geq$ & $2|V(H_1)| - 20 + c(H_1) + 2|X_2| - 9$ \\
              & $>$    & $\psi(H_1)$
  \end{tabular}

  The inequalities $\phi(H) > \psi(H_2)$ and $\phi(H) \geq \psi(H_1) +
  \psi(H_2)$ hold similarly.

  \vspace{2ex}

  \noindent{\bf Case 6: } By definition of the blocks:
  $|V(H_1)| = |X_1| + 4$ and $|V(H_2)| \leq 8$. 

  \noindent By~(\ref{c.notinter}), (\ref{c.oddcheck}) and since by
  (\ref{c.4inter}) viewed in $H_1$ no check of $H_1$ intersect
  $H_1\setminus X_1$ because $|X_2| \leq 4$:

  $c(H) \geq c(H_1) + 1$
  
  $|X_2| = 4$

  \noindent So:

  \begin{tabular}{lcl}
    $\phi(H)$ & $=$    & $2|X_1| + 2|X_2| - 20 + c(H)$\\
              & $\geq$ & $2|V(H_1)| - 20 + c(H_1) + 2|X_2| - 7$ \\
              & $>$    & $\psi(H_1)$
  \end{tabular}

  The inequalities $\phi(H) > \psi(H_2)$ and $\phi(H) \geq \psi(H_1) +
  \psi(H_2)$ are easy since $|V(H_2)| \leq 8$ implies $\psi(H_2) \leq 0$.

  \vspace{2ex}

  \noindent{\bf Case 7: } By definition of the blocks:
  $|V(H_1)| = |X_1| + 4$ and $|V(H_2)| = |X_2| + 2$. 

  \noindent By~(\ref{c.notinter}), (\ref{c.oddcheck}),
  (\ref{c.4inter}) and since $H_1$ is a hole:

  $c(H_1) = 0$

  $c(H) \geq c(H_2) - 2$
 
  \noindent So:

  \begin{tabular}{lcl}
    $\phi(H)$ & $=$    & $2|X_1| + 2|X_2| - 20 + c(H)$\\
              & $\geq$ & $2|V(H_1)| - 20 + c(H_1) + 2|X_2| - 8$ \\
              & $>$    & $\psi(H_1)$
  \end{tabular}

  \begin{tabular}{lcl}
    $\phi(H)$ & $=$    & $2|X_1| + 2|X_2| - 20 + c(H)$\\
              & $\geq$ & $2|V(H_2)| - 20 + c(H_2) + 2|X_1| - 6$ \\
              & $>$    & $\psi(H_2)$
  \end{tabular}

  The inequality $\phi(H) \geq \psi(H_1) + \psi(H_2)$ holds similarly. 

  \vspace{2ex}

  \noindent{\bf Case 8: } By definition of the blocks: 
  $|V(H_1)| = |X_1| + 4$ and $|V(H_2)| \leq 6$.

  \noindent By~(\ref{c.notinter}), (\ref{c.oddcheck}) and since $H_1$
  is a hole:

  $c(H) > c(H_1) = 0$ 

  $|X_2| = 4$

  \noindent So:

  \begin{tabular}{lcl}
    $\phi(H)$ & $=$    & $2|X_1| + 2|X_2| - 20 + c(H)$\\
              & $\geq$ & $2|V(H_1)| - 20 + c(H_1) + 2|X_2| - 7$ \\
              & $>$    & $\psi(H_1)$
  \end{tabular}

  The inequalities $\phi(H) \geq \psi(H_1) + \psi(H_2)$ and $\phi(H) >
  \psi(H_2)$ are easy since $|V(H_2)| = 6$ implies $\psi(H_2) \leq 0$.

\end{proof}

In $T_G$ every node $F$ satisfies $\phi(F) \geq 1$. Hence, by
Lemma~\ref{l.counting}, $T_G$ has at most $\phi(G) = O(n)$
leaves. Since every non-leaf node of $T_G$ has exactly two children,
$T_G$ has at most $2\phi(G) = O(n)$ nodes as claimed above.

We claim that $T_G$ can be constructed in time $O(n^9)$.  Indeed,
testing whether $G$ is basic is easy (see Theorem~\ref{th.basicalgo}).
In~\cite{conforti.c.k.v:eh2}, an $O(n^8)$ algorithm, due to
Cornu\'ejols and Cunningham~\cite{cornuejols.cunningham:2join}, for
constructing a substantial non-path 2-join of an input graph is
given. Note that what we call non-path substantial 2-join is simply
called 2-join in~\cite{conforti.c.k.v:eh2}. Finding substantial path
2-joins is easy in linear time by checking every vertex of
degree~2. Testing whether a 2-join is degenerate is easy in linear
time. By the paragraph above, to construct $T_G$ in the worst case, we
will have to run $O(n)$ times the $O(n^8)$ algorithm that detects
non-path substantial 2-joins.

We claim that $G$ has \an\ \even\ skew partition if and only if one of
the leaves of $T_G$ has \an\ \even\ skew partition. Indeed, if $G$ has
\an\ \even\ skew partition then Lemma~\ref{l.skew2joind} shows by an
easy induction that at least one of the leaves of $T_G$ has
\an\ \even\ skew partition.  Conversely, if a leaf $F$ of $T_G$ has
\an\ \even\ skew partition then suppose for a contradiction that $G$
has no \even\ skew partition. Among the nodes of $T_G$, let $H$ be the
graph with no \even\ skew partition, closest to $F$ along the unique
path of $T_G$ from $G$ to $F$. The graph $H$ is Berge, has no
\even\ skew partition, and is not basic.  Since it is not a leaf, $H$
has a proper 2-join by definition of $T_G$.  If $H$ has a non-path
proper 2-join, then by Lemma~\ref{l.skew2joinnonpath} the children of
$H$ in $T_G$ has no \even\ skew partitions contradictory to the
definition of $H$.  Else, by Theorem~\ref{th.case}, the children of
$H$ has no \even\ skew partition, a contradiction again.

We claim that we can test whether a leaf $L$ of $T_G$ has \an\ \even\
skew partition in $O(n^5)$. If $L$ is marked ``basic'', this is true
by Theorem~\ref{th.basicalgo}. If $L$ is marked ``small'', this is
trivial. If $L$ is marked ``no decomposition'', this is done in constant
time by answering ``YES'', the correct answer by
Theorem~\ref{th.case}. If $L$ is marked ``degenerate'', this is done
in constant time by answering ``YES'', the correct answer by
Lemma~\ref{l.degenerate}.

By the claims above, detecting \even\ skew partitions in a Berge graph
$G$ can be performed as follows: construct $T_G$ and test whether a
leaf has or not \an\ \even\ skew partition.  Note that in the case
when $G$ has no \even\ skew partition, then the leaves of $T_G$ are
all basic. We proved:

\begin{theorem}
  There is an $O(n^9)$-time algorithm that decides whether a Berge
  graph has or not \an\  \even\  skew partition.
\end{theorem}

\section{NP-hardness}
\label{NPC}

We recall here a construction due to
Bienstock~\cite{bienstock:evenpair}.  Let us call \emph{Bienstock
  graph} any graph $G$ that can be constructed as follows. Let $n \geq
3$, $m \geq 1$ be two integers. For every $1 \leq i \leq n$ let
$\alpha_i$ be the graph represented in Fig.~\ref{fig.bienstock1}, with
vertex-set $\{t_{i,1}$, $t_{i,2}$, $t_{i,3}$, $t_{i,4}$, $f_{i,1}$,
$f_{i,2}$, $f_{i,3}$, $f_{i,4}$, $c_{i,1}$, $c_{i,2}$, $c_{i,3}$,
$c_{i,4}\}$ and with edge-set $\{c_{i,1}t_{i,1}$, $t_{i,1}c_{i,3}$,
$c_{i,1}f_{i,1}$, $f_{i,1}c_{i,3}$, $c_{i,2}t_{i,2}$,
$t_{i,2}t_{i,3}$, $t_{i,3}t_{i,4}$, $t_{i,4}c_{i,4}$,
$c_{i,2}f_{i,2}$, $f_{i,2}f_{i,3}$, $f_{i,3}f_{i,4}$,
$f_{i,4}c_{i,4}$, $t_{i,1}f_{i,2}$, $t_{i,1}f_{i,3}$,
$f_{i,1}t_{i,2}$, $f_{i,1}t_{i,3}$, $t_{i,3}f_{i,3} \}$.  For every $1
\leq j \leq m$, let $\beta_j$ be the graph represented in
Fig.~\ref{fig.bienstock2}, with vertex-set $\{d_{j,1}$, $d_{j,2}$,
$d_{j,3}$, $d_{j,4}$, $r_j$, $z_{j,1}$, $z_{j,2}$, $z_{j,3}\}$ and
edge-set $\{d_{j,1}r_j$, $r_j d_{j,3}$, $d_{j,2}z_{j,1}$,
$z_{j,1}d_{j,4}$, $d_{j,2}z_{j,2}$, $z_{j,2}d_{j,4}$,
$d_{j,2}z_{j,3}$, $z_{j,3}d_{j,4}\}$.

All the graphs $\alpha_i, \beta_j$ are pairwise vertex-disjoint
subgraphs of $G$ that are assembled by adding the following edges:
$c_{i, 3} c_{i+1, 1}$ and $c_{i, 4} c_{i+1, 2}$ for $1 \leq i <n$,
$d_{j, 3} d_{j+1, 1}$ and $d_{j, 4} d_{j+1, 2}$ for $1 \leq j <m$. Add
a vertex $u$ adjacent to $c_{1,2}$, a vertex $w$ adjacent to
$c_{1,1}$, a vertex $s$ adjacent to $w$ and a vertex $v$ adjacent to
$d_{m,3}$, $d_{m,4}$.  See Fig.~\ref{fig.bienstock4}.  For every
$1\leq j \leq m$ and every $k \in \{1, 2, 3\}$ we add exactly 2 edges
incident to $z_{j,k}$. These edges are either $z_{j, k} f_{i,1}, z_{j,
  k} f_{i,3}$ for some $i$, or $z_{j, k} t_{i,1}, z_{j, k} t_{i,3}$
for some $i$. See a possibility in
Fig.~\ref{fig.bienstock3}. Moreover, for every $1 \leq k < k' \leq 3$
and every $1 \leq j \leq m$, $z_{j,k}$ and $z_{j,k'}$ are required to
have their neighbors in different $\alpha_i$'s.

  \begin{figure}[p]
    \begin{center}
      \includegraphics{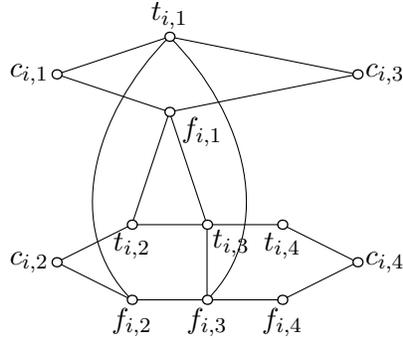}
      \caption{Graph $\alpha_i$\label{fig.bienstock1}}
    \end{center} 
  \end{figure}

  \begin{figure}[p]
    \begin{center}
      \includegraphics{bienstock.2}
      \caption{Graph $\beta_j$\label{fig.bienstock2}}
    \end{center} 
  \end{figure}

  \begin{figure}[p]
    \begin{center}
      \includegraphics{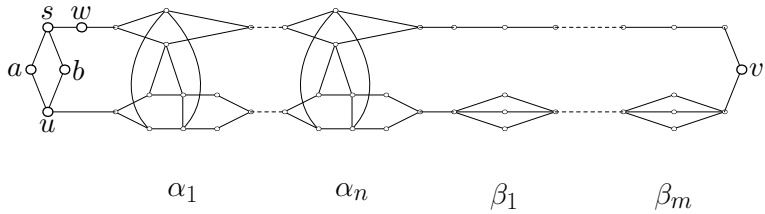}
      \caption{$G'$, that is the whole graph $G$ plus two vertices $a,b$\label{fig.bienstock4}}
    \end{center} 
  \end{figure}

  \begin{figure}[p]
    \begin{center}
      \includegraphics{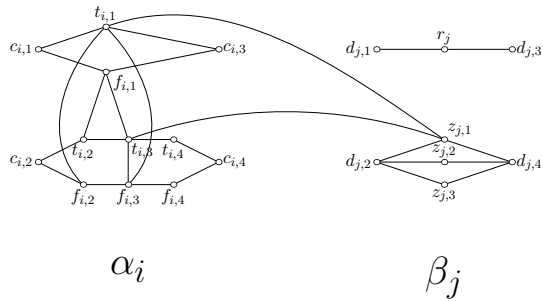}
      \caption{The two edges out from $z_{j,1}$, a possibility\label{fig.bienstock3}}
    \end{center} 
  \end{figure}

By 3-SAT' we mean the usual 3-SAT problem
(see~\cite{garey.johnson:np}) restricted to the sets of clauses on 3
variables such that every clause is on three pairwise distinct
variables. Bienstock proved an NP-completeness reduction from 3-SAT
that when restricted to 3-SAT' yields:

\begin{theorem}[Bienstock~\cite{bienstock:evenpair}]
  \label{th.bienstock}
  For every instance $\cal I$ of size $x$ of the NP-complete problem
  3-SAT', there is a Bienstock graph $G_{\cal I}$ of size $O(x)$,
  obtained from $\cal I$ by a linear time algorithm and such that the
  answer to $\cal I$ is YES if and only if there is a path of $G_{\cal
    I}$ of odd length joining $u$ and $s$.
\end{theorem}

Here is why Bienstock's construction is related to the \Even\ Skew
Partition Problem:

\begin{lemma}
  \label{l.skewbienstock}
  Let $G$ be a Bienstock graph. Let $G'$ be the graph obtained by
  adding two vertices: a vertex $a$ seeing both $u$, $s$ and a vertex
  $b$ also seeing both $u$, $s$.  Then $G'$ has \an\  \even\  skew partition
  if and only if there is no path of odd length in $G$ joining $u$ and
  $s$.
\end{lemma}

\begin{proof}
  The graph $G'$ is represented in Fig.~\ref{fig.bienstock4}.  The sets
  $\{a, u, s\}$ and $\{b, u, s\}$ are clearly skew cutsets.  If there
  is a path of odd length in $G$ between $u$ and $s$ then these two
  skew cutsets are non-\even. Else they are clearly both \even. To
  check the condition on antipaths, note that since the cutset has
  size~3, a bad antipath would have length~3, and so could be seen as
  a path. Hence if suffices to prove that $G'$ has no other skew
  cutset.  Note that $G'$ has no diamonds and no $K_4$. Hence, every
  skew cutset of $G'$ is either a star cutset or is a complete
  bipartite graph. Let us check every star and every square in $G'$.

  We observe that $G'$ has no star cutset centered on: $s$, $u$, $w$,
  $v$; $c_{i,k}$, $t_{i,4}$, $f_{i,4}$, $t_{i,2}$, $f_{i,2}$ for
  $1\leq i \leq n$; $d_{j,k}$ for $1\leq j \leq m$, $k\in \{1, 2,
  3\}$. Also $G'$ has no star cutset centered on $z_{j,k}$ since
  $z_{j,k}$ has degree~4 and since for $k'\in \{1, 2, 3\} \setminus
  k$, $z_{j,k'}$ does not have its neighbors in the same $\alpha_i$
  than $z_{j,k}$. A star centered on a vertex $x$ among $t_{i,1}$,
  $f_{i,1}$, $t_{i,3}$, $f_{i,3}$ is dangerous since $x$ may have
  large arbitrarily large degree. But this is not enough to disconnect
  $G'$ since $x$ has at most one neighbor in every $\beta_j$.

  The square $G'[a,b,s,u]$ is not a skew cutset of $G'$. Moreover,
  since $s,u$ (resp. a,b) have no common neighbors in $G'$, no skew
  cutset can contain $\{a,b,s,u\}$. Similarly, for $1\leq i \leq n$,
  no skew cutset of $G'$ can contain $\{c_{i,1}, t_{i,1}, c_{i,3},
  f_{i,1}\}$. No skew cutset of $G'$ can contain $\{d_{1,2}, z_{1,1},
  d_{1,4}, z_{1,2}\}$ since $z_{1,3}$ is the only possible vertex to
  be added to the potential skew cutset, and since $z_{1,3}$ has a
  neighbor in some $\alpha_i$. By the same way, no skew cutset can be
  contained in $\beta_j$, $1\leq j \leq m$. The last squares to be
  checked are those contained in sets consisting of some $t_{i,1},
  t_{i,3}$ (resp. $f_{i,1}, f_{i,3}$) plus a collection of $z_{j,k}$'s
  complete to $\{t_{i,1}, t_{i,3}\}$ (resp. $f_{i,1}, f_{i,3}$). Note
  that the $z_{j,k}$'s are all in different $\beta_j$'s. Hence such a
  set is not a skew cutset.  
\end{proof}

\begin{theorem}
  \label{th.skewnpc}
  The decision problem whose instance is any graph $G$ and whose
  question is ``does $G$ have \an\  \even\  skew partition ?'' is
  NP-hard.
\end{theorem}

\begin{proof}
  Let $\cal I$ be an instance of 3-SAT'.  By
  Theorem~\ref{th.bienstock}, we construct a graph $G_{\cal I}$.  By
  Lemma~\ref{l.skewbienstock} we construct a graph $G'_{\cal I}$. By
  these two results $G'_{\cal I}$ has \an\ \even\ skew partition if
  and only if the answer to $\cal I$ is NO.
\end{proof}

\section{Conclusion}
\label{conclu}

As a conclusion we would like to give two conjectures suggested by
this work. The first one is motivated by a remark of an anonymous
referee who noticed that our NP-hardness reduction is more coNP than
NP.  For NP-hardness reductions, this makes no difference, but an
NP-completeness result would be:

\begin{conjecture}
  \ESPD\ is coNP-complete. 
\end{conjecture}

The following would be quite natural:

\begin{conjecture}
  There is a polynomial-time algorithm which, given a Berge graph $G$,
  outputs \an\ \even\ skew partition of $G$ if any, and otherwise
  certifies that $G$ has none.
\end{conjecture}

\section*{Acknowledgement}

I am grateful to Fr\'ed\'eric Maffray, Guyslain Naves, B\'eatrice
Trotignon, \'Elise Vidal and two anonymous referees for several
suggestions. Thanks to Maria Chudnovsky for useful discussions, in
particular for pointing out Theorem~\ref{th.2}.


\end{document}